\documentclass{siamart220329}
\usepackage{mathrsfs}
\usepackage{enumerate}
\usepackage{amssymb,amsmath,amsfonts}
\usepackage{graphicx,subfigure,epstopdf}
\usepackage{color}
\usepackage{multirow}
\usepackage{booktabs}
\usepackage{tikz}
\usepackage{cleveref}
\usepackage[text={145mm, 216mm}, asymmetric]{geometry}

\allowdisplaybreaks

\usepackage[draft]{changes}
\definechangesauthor[name=YANG Zongze, color=green]{ZZ}
\definechangesauthor[name=DUAN Beiping, color=blue]{BP}
\let\oldsout\sout
\renewcommand{\sout}[1]{\ifmmode\text{\oldsout{\ensuremath{#1}}}\else\oldsout{#1}\fi}

\definecolor{darkred}{rgb}{.7,0,0}

\definecolor{green}{rgb}{0,0.7,0}

\let\burlalt\href

\newtheorem{remark}{Remark}[section]

\def\cC{\mathcal{C}}
\def\cA{\mathcal{A}}
\def\bR{\mathbb{R}}
\def\i{\mathrm{i}}
\def\cI{\mathcal{I}}

\def\cB{\mathcal{B}}

\def\cS{\mathcal{S}}
\def\wrho{\widetilde{\rho}}

\def\mbC{\mathbf{C}}
\def\mba{\mathbf{a}}
\def\wg{\widetilde{g}}
\def\IM{\textup{\texttt{Im}}}
\def\RE{\textup{\texttt{Re}}}

\def\wsigma{\widetilde{\sigma}}
\renewcommand{\theequation}{\thesection.\arabic{equation}}

\newcommand{\vertiii}[1]
{{\left\vert\kern-0.25ex\left
\vert\kern-0.25ex\left\vert #1
    \right\vert\kern-0.25ex\right
\vert\kern-0.25ex\right\vert}}

\newtheorem{assumption}{\textsc{Assumption}}

\graphicspath{{./figures/}}

\headers{approximation for nonlocal diffusion equations}{B. Duan and Z. Yang}
\title{A quadrature scheme for steady-state diffusion equations 
    involving fractional power of regularly accretive operator\thanks{
        Submitted to the editors \today.
    }
}

\author{Beiping Duan\thanks{Faculty of Computational Mathematics and Cybernetics, 
    Shenzhen MSU-BIT University, Shenzhen  518172, P.R. China
    (\email {duanbeiping@smbu.edu.cn}).}
\and Zongze Yang\thanks{Corresponding author.
	 Department of Applied Mathematics, 
    The Hong Kong Polytechnic University, Kowloon, Hong Kong
    (\email{zongze.yang@polyu.edu.hk}).}
}

\begin{document}
	
\maketitle

\begin{abstract}
In this paper we construct a quadrature scheme to numerically solve the nonlocal diffusion equation $(\cA^\alpha+b\cI)u=f$ with $\cA^\alpha$ the $\alpha$-th power of the regularly accretive operator $\cA$.  Rigorous error analysis is carried out and sharp error bounds (up to some negligible constants) are obtained. The error estimates include a wide range of cases in which the regularity index and spectral angle of $\cA$, the smoothness of $f$, the size of $b$ and $\alpha$ are all involved. The quadrature scheme is exponentially convergent with respect to the step size and is root-exponentially convergent with respect to the number of solves. Some numerical tests are presented in the  last section to verify the sharpness of our estimates. Furthermore, both the scheme and the error bounds can be utilized directly to  solve and analyze time-dependent problems.
\end{abstract}



\section{Introduction}\label{Se:1}

\subsection{Motivation and problem formulation}
Mathematical investigations for equations involving nonlocal operators have
received much attention  due to their wide range of applications, see  e.g.,
\cite{constantin1999behavior,bakunin2008turbulence,eringen2002nonlocal,gilboa2007nonlocal,gilboa2008nonlocal,duvant2012inequalities,mikki2020theory},
and we also recommend the nice review paper \cite{d2020numerical} to the
interested readers. Two most commonly used nonlocal operators in space are
integral fractional Laplacian and fractional powers of elliptic operators, where
the former can be derived from the Fourier representation and the latter 
originates from functional calculus. Numerical approaches for both kinds
of nonlocal operators are attracting more and more attention and many
numerical methods have been proposed and developed, such as the mapped
Chebyshev spectral-Galerkin method  \cite{sheng2020fast}, spectral methods
\cite{zhang2019error}, finite element methods and adaptive finite element
methods \cite{acosta2017fractional,bonito2019numerical,ainsworth2017aspects},
and the hybrid finite element-spectral method  \cite{ainsworth2018hybrid}
et al. For numerical approximation for fractional powers of elliptic operators,
most of the research papers focus on transplanting the nonlocal target problem
into local systems, for example, by applying Gauss-Jacobi quadrature or
trapezoidal rule to corresponding integral representations
\cite{aceto2019rational,bonito2015numerical}, by Vabishchevich's idea
\cite{Vabishchevich15,DLP-Pade,duan1}, by the best uniform rational approximation
\cite{Harizanov-JCP,harizanov2018optimal}, and by the Caffarelli-Silvestre
extension \cite{nochetto2015pde,chen2020efficient}, et al. Interested readers
may refer to the comprehensive review paper  \cite{harizanov2020survey}. 

We begin with the definition of the fractional power of a second-order elliptic operator in a bounded
domain. Let $\Omega$ denote a bounded polygonal domain in $\bR^d(d\ge 1)$ with a Lipschitz continuous boundary $\Gamma=\partial \Omega$ and suppose $\Gamma=\Gamma_D\cup \Gamma_N$ and $\Gamma_D\cap \Gamma_N=\emptyset$. We further suppose $\Gamma_D$ is nonempty and does not contain any measure whose dimension is less than $d-1$. To avoid complicated discussions for inhomogeneous boundary conditions, we only consider homogeneous Dirichlet and Neumann boundary conditions, correspondingly on $\Gamma_D$ and $\Gamma_N$. We use $V\subset H^1(\Omega)$ to represent the function set equipped with $H^1(\Omega)$-norm in which the traces of the functions vanish on $\Gamma_D$. 

	We introduce a sesquilinear form(see the definition in the next section) $A(\cdot,\cdot)$: 
	\begin{equation}\label{eqn-sesquilinear}
		A(w,v)=\int_{\Omega}\nabla w\, \mbC (x)(\overline{\nabla v})^T + (\mba(x)\cdot\nabla) w \overline{v} +r(x)w\overline{v}\,dx\qquad \forall w,v\in V,
	\end{equation}
	where $\mbC$ is a  complex-valued  $d\times d$  matrix, $\mba(x)$ is a complex-valued $d$-dimensional vector and $r(x)$ is a complex-valued scalar function. We further assume that the sesquilinear form $A(\cdot,\cdot)$ satisfies strong ellipticity
	\begin{equation}\label{prop-e}
		\Re (A(v,v))\ge c_0\|v\|_{V}^2 \quad \mbox{for all } v\in V,
	\end{equation}
	and continuity
	\begin{equation}\label{prop-c}
		|A(w,v)|\le c_1\|w\|_{V}\|v\|_{V} \quad \mbox{for all } w,v \in V,
	\end{equation}
	where $c_0,c_1$ are positive real numbers. Note that the two inequalities above also imply $c_0\le c_1$.
	
	Let $V_a^*$ denote the  set of all antilinear functionals on $V$, then  by the Lax-Milgram theorem we know for any $G_a\in V_a^*$ there exists a unique $w=T G_a\in V$ with $T:V_a^*\rightarrow V$ such that for any $v\in V$ it holds $A(T G_a,v)=G_a(v)$. Let $\cA=T^{-1}$ denote the second order operator associated with the sesquilinear form $A(\cdot,\cdot)$.

Most work related to the fractional powers of elliptic operators focuses on
solving $\cA^\alpha u=f$ or equivalently $u=\cA^{-\alpha}f$ with $\cA$ a
self-adjoint operator, and to our best knowledge,
\cite{bonito2016numerical} and \cite{bonito2019sinc} are the only
papers to consider non-Hermitian cases, which revisited the error estimates for
the Bonoti-Pasciak quadrature scheme proposed in \cite{bonito2015numerical} when
$\cA$ is regularly accretive. 

In this paper we investigate numerical approaches to solve the
 following nonlocal diffusion equation:
\begin{equation}\label{eqn-p1}
    (\cA^\alpha +b\cI)u=f, \quad \mbox{with }u|_{x\in\Gamma_D}=0,\;\frac{\partial u}{\partial n}\Big|_{x\in \Gamma_N}=0,
\end{equation} 
where $\alpha\in (0,1)$, $b>0$, 
and $\cA^\alpha$ denotes the $\alpha$-th power of $\cA$, defined by the Dunford-Taylor integral formula
\begin{equation}\label{Dunford-Taylor}
    \cA^\alpha=\frac{1}{2\pi \i}\int_{\cC'}z^\alpha(z\cI -\cA)^{-1}dz,
\end{equation}
with $\cC'$ an integral contour in complex plane surrounding the spectrum of $\cA$. Without loss of generality,  we assume $b\in[1,+\infty)$  hereafter to avoid introducing more tedious discussions.

Apart from purely mathematical interest,
\eqref{eqn-p1} can be obtained from time discretization of parabolic equations
involving $\cA^\alpha$. For example, when we apply linearized backward Euler
scheme to the model problem
\begin{equation}\label{super-diffusion}
    \frac{\partial w(x,t)}{\partial t} +\cA^\alpha w(x,t)=g(x,t,w),
\end{equation}
we obtain
\begin{equation}\label{back-Euler}
    \left(\cA^\alpha+\frac{1}{\Delta t}\right)w^{n+1}(x)=g(x,t^{n+1},w^n)+\frac{w^{n}(x)}{\Delta t}
\end{equation}
with $\Delta t$ the temporal step size and $w^{n+1}(x)$ the semidiscrete solution at $t_{n+1}=(n+1)\Delta t$. It is easy to show that \eqref{back-Euler} is unconditionally stable for $g(x,t,w)=g(x,t)$, see \cite{petr2016numerical}.  So at each time level the target system has the same formulation as \eqref{eqn-p1} with $b=\Delta t^{-1}$. Compared with the case of $b\equiv 0$, fewer papers study numerical approaches for \eqref{eqn-p1}. Methods based on the  Caffarelli–Silvestre extension and Vabishchevich's idea can not be applied directly any more even if $\cA$ is Hermitian. In \cite{harizanov2020numerical} the authors proposed a rational approximation scheme to solve the discrete counterpart of \eqref{eqn-p1} for a Hermitian $\cA$. Two approaches are constructed in their paper that  are based on the best diagonal  rational approximations for $\xi^\alpha$ and $\frac{1}{\xi^{-\alpha}+b}$ on $\xi\in[0,1]$, respectively. In \cite{burrage2012efficient}, several approaches were employed to solving \eqref{super-diffusion} for $\alpha\in(\frac{1}{2},1]$. The approaches are based on the contour integral method proposed in \cite{hale2008computing}, and rational Krylov subspace-based techniques.

\subsection{Our approach}
In this paper, we construct the following integral approach to evaluate the inverse of $\cB^\alpha+t\cI$ for $t\ge 0$
\begin{equation}\label{int-formula-A}
        (\cB^\alpha +t\cI)^{-1}=	\frac{\sin\pi\alpha}{\alpha\pi}\int_{-\infty}^\infty\frac{(1+e^{-s/\alpha}\cB)^{-1}}{e^{s}+2t\cos\pi\alpha +t^2e^{-s}}ds,
\end{equation} 
where $\cB$ can be any operator satisfying some loose restrictions. We first approximate $\cA$ by $\cA_h$, which in practice can be obtained by a finite difference method, finite element method or other numerical schemes. The  `semidiscrete' solution $u_h=(\cA_h^\alpha +b\cI)^{-1}f_h$ then can be represented by \eqref{int-formula-A} and the `fully discrete' solution $U_h$ follows from applying the trapezoidal rule to the integral representation \eqref{int-formula-A} and cutting the summation to finite terms
\begin{equation}\label{quadrature-Uh}
    U_h:=U_{h,\tau}^{M,N}=\frac{\sin \pi \alpha}{\alpha\pi}\tau\sum_{n=-M}^{N}\frac{(1+e^{-n\tau/\alpha}\cA_h)^{-1}f_h}{e^{n\tau}+2b\cos\pi\alpha +b^2e^{-n\tau}},
\end{equation}
where $\tau$ is the quadrature step size. 

We give rigorous analysis for this scheme, and error bounds for $\|u-u_h\|_{L^2(\Omega)}$ and $\|u_h-U_h\|_{L^2(\Omega)}$ are demonstrated. In the two main theorems of this paper -- \Cref{theorem-space} and \Cref{theorem-quadrature}, the parameters appeared in the target equation \eqref{eqn-p1} including $\alpha,b$, the smoothness of $f$, and the elliptic regularity index and the spectral angle of $\cA$, are all taken into consideration. 

The outline of this paper is as follows. In Section 2 we introduce some basic concepts of accretive operators and give the integral representation of $(\cB^\alpha+t\cI)^{-1}$. In Section 3 we present the error estimates for $\|u-u_h\|_{L^2(\Omega)}$.  The error bound for $\|u_h-U_h\|_{L^2(\Omega)}$ is given in Section 4  that shows that it is exponentially decaying with respect to the quadrature step size $\tau$, and is root-exponentially convergent with respect to the number of solves  under optimal choice of $M, N$ and $\tau$. In Section 5 some numerical examples are presented to confirm our theoretical analysis.

\section{Preliminary and integral approach}\label{sec:pre}
We first introduce some basic concepts related to our topic, which can be found in \cite{kato1961}. For the readers' convenience we list them below.

Let $X$ be a Hilbert space. A complex-valued function $\Psi(v,w)$ defined for $v,w$ belonging to a linear subset of $X$ is called a \textit{sesquilinear} form if it is linear in $v$ and semilinear in $w$. We use $D(\Psi)$ to denote the domain of $\Psi$. For a given sesquilinear form $\Psi(\cdot,\cdot)$, $\Psi^*(v,w)=\overline{\Psi(w,v)}$ defines another sesquilinear form $\Psi^*$ with $D(\Psi^*)=D(\Psi)$. We say $\Psi^*$ is the {\textit{adjoint form}} of $\Psi$. $\Psi$ is said to be Hermitian or symmetric if $\Psi=\Psi^*$.

A Hermitian form $\Phi$ is said to be {\textit{nonnegative}} if $\Phi(v,v)\ge 0$ for all $v\in D(\Phi)$. A nonnegative  Hermitian form $\Phi$ is said to be {\textit{closed}} if $v_n\in D(\Phi)$, $v_n\rightarrow v\in X$ and $\Phi(v_n-v_m,v_n-v_m)\rightarrow 0$ is sufficient to derive $v\in D(\Phi)$ and $\Phi(v_n,v_n)\rightarrow \Phi(v,v)$. 

For a given sesquilinear form, one can decompose it into two Hermitian forms by
\begin{equation*}
    \Psi=\frac{1}{2}\left(\Psi+\Psi^*\right)+\i\frac{1}{2\i}\left(\Psi-\Psi^*\right):=\Psi_{\RE}+\i\Psi_{\IM},
\end{equation*}
which are referred as the {\textit{real}} and {\textit{imaginary}} parts of $\Psi$, respectively. It is worth to point out that $\Psi_{\RE}$ and $\Psi_{\IM}$ are not real-valued. 

\begin{definition}\label{def-regular}
 A sesquilinear form $\Psi(\cdot,\cdot)$ will be said to be {\rm regular} if 
    \begin{enumerate}[1)]
        \item the domain of $\Psi(\cdot,\cdot)$ is dense in $X$;
        \item $\Psi_{\RE}$ is a closed, nonnegative Hermitian form;
        \item there exists some $\widetilde\beta\ge 0$ such that $|\Psi_{\IM}(v,v)|\le \widetilde\beta \Psi_{\RE}(v,v)$ holds for $\forall v\in D(\Psi)$. 
    \end{enumerate}

An operator will be said to be {\rm regular} if it is associated with a regular sesquilinear form. The smallest number $\widetilde\beta$ is called {\rm the index of $\Psi$} or {\rm the index of the associated operator}.
\end{definition}

An operator $\cB$ is said to be \textit{accretive} if $\Re(\cB v,v)\ge 0$ for $\forall v\in D(\cB)$. Further, we say $\cB$ is \textit{maximal accretive} if $\lambda+\cB$ is surjective for $\lambda\ge 0$.

Obviously, the sesquilinear form $A(\cdot,\cdot)$ satisfies 1) and 2) in Definition \eqref{def-regular} with $X=L^2(\Omega)$. In addition,  \eqref{prop-e} and \eqref{prop-c} imply 
\begin{equation*}
    \left|A_{\IM}(v,v)\right| \le \frac{\sqrt{c_1^2-c_0^2}}{c_0}A_{\RE}(v,v).
\end{equation*}
So $\cA$ is maximal regularly accretive. Hereafter we use $\beta$ to denote the index of $\cA$ and $\|\cdot\|$ to represent the $L^2(\Omega)$-norm.

\begin{lemma}\label{lemma-integral}
    Suppose $\cB:D(\cB) \rightarrow X$ is an operator on the Banach space $X$ and its resolvent  contains $\Sigma_{\omega}\cup \{z: |z|< \lambda_0\}$, where $\Sigma_\omega=\{z\in \mathbb{C}: \pi-|\arg z|<\omega\}$ 
    with constant $\lambda_0 > 0$ and $\omega \in (0,\pi)$. Then the following integral formula holds for $\alpha\in(0,1)$ and $t\ge 0$
    \begin{equation}\label{int_formula}
        (\cB^\alpha +t\cI)^{-1}=\frac{\sin\pi\alpha}{\pi}\int_{0}^{+\infty}(\rho+\cB)^{-1}\frac{\rho^\alpha}{\rho^{2\alpha}+2t\cos\pi\alpha \rho^\alpha+t^2}d\rho.
    \end{equation}
\end{lemma}
\begin{proof}
    We cut the $z$-plane along the negative real line then by the Cauchy integral formula we know
    \begin{equation}\label{int-ori}
        (\cB^\alpha +t\cI)^{-1}=\frac{1}{2\pi \i}\int_{\cC}(z^\alpha +t)^{-1}(z\cI -\cB)^{-1}dz,
    \end{equation}
    where $\cC$ is an integral contour in the complex plane surrounding the spectrum of $\cB$ and avoiding the branch point $z=0$.
    \begin{figure}
        \centering
        \begin{tikzpicture}[scale=0.5]
            \draw [-stealth](-4,0) -- (4,0)node[right]{Re};
            \draw [-stealth](0,-3) -- (0,3)node[above]{Im};
            \draw [-stealth](-3.6,3) -- (-1.8,1.5);
            \draw (-1.8,1.5) -- (-0.36,0.3);
            \draw [-stealth](-0.36,-0.3) -- (-1.8,-1.5);		
            \draw(-3.6,-3) -- (-1.8,-1.5);
            \draw (-0.36,-0.3) arc(219.8:500.2:0.4686);
            \draw [->](0.4686,0.01) -- (0.4686,0);
            \node at (-1.6,2.1) {$\cC_1$};
            \node at (0.7,0.7) {$\cC_2$};
            \node at (-1.6,-2.1) {$\cC_3$};
        \end{tikzpicture}
        \caption{Integral curve on the complex plane.} \label{fig:IntCur}
    \end{figure}
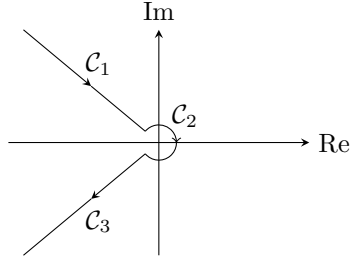
    We choose $\cC=\cC_1\cup\cC_2\cup\cC_3$, as shown in  \Cref{fig:IntCur}, where $\cC_1=\rho e^{\kappa\pi \i}$ and $\cC_3=\rho e^{-\kappa\pi \i}$ with $\rho\in[\delta,+\infty]$, and $\cC_2=\delta e^{\theta \i}$ with $\theta\in [-\kappa\pi,\kappa\pi]$, $\kappa\in (\frac{\pi-\omega}{\pi},1]$. Here $\delta$ is chosen small enough such that the spectrum lies on the right side of $\cC$. Set $\kappa=1$, then
    \begin{equation*}
        \begin{aligned}
            (\cB^\alpha +t\cI)^{-1}=&\frac{1}{2\pi \i}\int_{\cC}(z^\alpha +t)^{-1}(z\cI -\cB)^{-1}dz\\
            =&\frac{1}{2\pi \i}\int_{+\infty}^{\delta}e^{\pi \i}(\rho^\alpha e^{\alpha\pi \i}+t)^{-1}(\rho e^{\pi \i}\cI-\cB)^{-1}d\rho\\
            &+\frac{1}{2\pi \i}\int_{\delta}^{+\infty}e^{-\pi \i}(\rho^\alpha e^{-\alpha\pi \i}+t)^{-1}(\rho e^{-\pi \i}\cI-\cB)^{-1}d\rho\\
            &+\frac{\delta}{2\pi}\int_{-\pi}^{\pi}e^{\theta \i}(\delta^\alpha e^{\alpha\theta \i}+t)^{-1}(\delta e^{\theta \i}\cI-\cB)^{-1}d\theta\\
            :=&I_1+I_3+I_2.
        \end{aligned}
    \end{equation*}
    After some manipulations one can obtain 
    \begin{equation*}
        I_1+I_3=\frac{\sin\pi\alpha}{\pi }\int_{\delta}^{+\infty}(\rho+\cB)^{-1}\frac{\rho^\alpha}{\rho^{2\alpha}+2t\cos\pi\alpha \rho^\alpha+t^2}d\rho
    \end{equation*}
    and 
    \begin{equation*}
        \|I_2\|_X\le \frac{\delta}{2\pi}\int_{-\pi}^{\pi}\left|(\delta^\alpha e^{\alpha\theta \i}+t)^{-1}\right| \|(\delta e^{\theta \i}\cI-\cB)^{-1}\|_X d\theta.
    \end{equation*}
    Since the resolvent of $\cB$ contains $|z|\le \lambda_0$  we know for $\delta\le \frac{\lambda_0}{2}$,  $\delta e^{\theta \i}\cI-\cB$ has continuous inverse, so
    \begin{equation*}
    \begin{aligned}
        \|I_2\|_X\le \frac{c\delta}{2\pi}\int_{-\pi}^{\pi}\left|(\delta^\alpha e^{\alpha\theta \i}+t)^{-1}\right| d\theta.
    \end{aligned}
    \end{equation*}	
    If $t=0$, we have $|(\delta^\alpha e^{\alpha\theta \i}+t)^{-1}|\le \delta^{-\alpha}$. Otherwise taking $\delta^\alpha\le \frac{t}{2}$ we can get its upper bound $\frac{2}{t}$. 
    Thus taking $\delta\rightarrow 0^+$ it follows $\|I_2\|_X\rightarrow 0$ then the  desired formula follows.
\end{proof}

Since $\cA$ is maximal regularly accretive, $\sigma(\cA)$ is contained in the numerical range of $\cA$,
 i.e.~\[\sigma(\cA) \subset \big\{\mu: \mu = \frac{(\cA v, v)}{\|v\|^2}, v\in V, v\ne0\big\}.\]
Thus appealing to \eqref{prop-e} we know $\cA$  satisfies the conditions in  \Cref{lemma-integral}. It is worth to point out that for $t=0$, \eqref{int_formula} reduces to the Balakrishnan formula  \cite{balakrishnan,kato1961,bonito2015numerical}.

\section{Error analysis in space}
In this section we shall demonstrate the error bound for $\|u-u_h\|$, and to this end we first introduce some Sobolev spaces that are needed in later analysis.

\subsection{Characterization of \texorpdfstring{$D(\cA^{\alpha})$}{D(A\^alpha)}}
Let $(\cdot,\cdot)$ denote the inner product in complex $L^2(\Omega)$ space and we use $\langle\cdot,\cdot \rangle$ to represent the dual pair that is linear with respect to the first variable and is antilinear with respect to the second variable. Furthermore, we denote by $V_l^*$ the linear functional space on $V$. 

Corresponding to the operator $\cA$  defined at the beginning of this paper, we introduce its dual by the following way. For any $G_l\in V_l^*$, then by the Lax-Milgram theorem it follows that there exists a unique $T^*G_l\in V$ with $T^*:V_l^*\rightarrow V$ such that $A^*(v,T^*G_l)=G_l(v)$ holds for any $v\in V$. We then define  $\cA^*=(T^*)^{-1}$. 

Appealing to \eqref{prop-e} and \eqref{prop-c} one immediately gets that $\left[A_{\RE}(w,w)\right]^{1/2}$ provides an equivalent norm on $V$. We denote by $\cS$ the self-adjoint operator associate with $A_{\RE}$.  The self-adjoint operator $\cS$ naturally induces a Hilbert space $\dot{H}^\mu$, given by $\dot{H}^{\mu}=D(\cS^{\mu/2})$.  It is known that the eigenpairs of $\cS$ is countable, and we denote them by $\{\lambda_j,\phi_j\}_{j=1}^\infty$ with $\|\phi_j\|_{L^2(\Omega)} = 1$. In terms of the eigenfunction expansion we can characterize $\dot{H}^\mu$ by 
\begin{equation*}
    \dot{H}^\mu:=D(\cS^{\mu/2})=\Bigg\{\sum_{j=1}^{\infty}c_j\phi_j: \sum_{j=1}^{\infty}|c_j|^2\lambda_j^\mu<\infty\Bigg\}.
\end{equation*}
The corresponding inner product is 
\begin{equation*}
    (w,v)_{\mu}=\sum_{j=1}^{\infty}\lambda_j^\mu (w,\phi_j)\overline{(v,\phi_j)}.
\end{equation*}
Furthermore, it is well known that $\dot{H}^1=V$ and $\dot{H}^0=L^2(\Omega)$. The antilinear and linear functionals  on $\dot{H}^\mu$, say $H^{-\mu}_a$ and $H^{-\mu}_l$ are respectively given by
\begin{equation*}
    H^{-\mu}_a=\big\{\langle v,\cdot \rangle: v\in \dot{H}^{\mu}\big\}, \quad\text{and}\quad H^{-\mu}_l=\big\{\langle \cdot,v \rangle: v\in \dot{H}^{\mu}\big\}.
\end{equation*}
Obviously we have $V_a^*=H^{-1}_a$ and $V_l^*=H^{-1}_l$. Assume \eqref{prop-e} and \eqref{prop-c} hold, then it follows for $\mu\in[0,1)$ \cite{kato1961}
\begin{equation*}
    D(\cA^{\mu/2})=D((\cA^*)^{\mu/2})=D(\cS^{\mu/2}).
\end{equation*}

Let $H^s(\Omega)(s>0)$ denote the Sobolev space of order $s$. We introduce the following spaces equipped with their natural norms:
\begin{equation*}
    \widetilde{H}^s:=\left\{
    \begin{matrix}
        \dot{H}^s  & \mbox{ for }& s\in [0,1],\\
        {H}^s(\Omega)\cap V  & \mbox{ for }& s>1.\\
    \end{matrix}
    \right.
\end{equation*}

\begin{assumption}\label{equivalent-space}
	Assume that for $s\in[0,\gamma]$ with $\gamma\in[0,1]$,  $D(\cA^{\frac{s+1}{2}}){\subset}\widetilde{H}^{1+s}$.
\end{assumption}

The case of $s=0$ is contained in the famous Kato square root problem that has been intensively studied under different boundary conditions, we refer \cite{axelsson2006kato} and the references therein for the interested readers. For the case of $s\in(0,\gamma]$, it can be demonstrated that the inclusion relation holds provided that $\gamma$ is the  elliptic regularity index of $\cA$, that is, $T$ is a bounded map from ${H}_a^{-1+s}$ into $\widetilde{H}^{1+s}$ for $s\in(0, \gamma]$, see \cite[Theorem 4.4]{bonito2016numerical}.
\subsection{Spatial error}
To apply the finite element method, we triangulate the domain $\Omega$ into finitely many simplices and we use $\mathcal{T}_h$ to represent the mesh and $h$ to denote the mesh size. We further suppose that the mesh is qusi-uniform and any simplex of the mesh with dimension less than $d$ lies on $\Gamma$ is either contained in $\Gamma_D$ or $\Gamma_N$. Let $P_1$ be the piecewise linear finite element space defined on $\mathcal{T}_h$ and $V_h=P_1\cap V$. Let $\cA_h$ denote the discrete counterpart of $\cA$ given by
\begin{equation*}
    (\cA_h w_h, v_h):=A(w_h,v_h)\quad \forall w_h,\quad u_h \in V_h.
\end{equation*}
Let $u_h=(\cA_h^\alpha + b\cI)^{-1}f_h$ be the finite element approximation for \eqref{eqn-p1} with $f_h=\pi_h f$,  where $\pi_h$ represents the $L^2$-projection mapping from $L^2(\Omega)$ onto  $V_h$ given by \[ (\pi_h f, v_h) = (f, v_h), \forall v_h \in V_h, \] then by virtue of  \Cref{lemma-integral} we know
\begin{equation}\label{u_h}
    u_h=\frac{\sin\pi\alpha}{\pi}\int_{0}^{+\infty}(\rho+\cA_h)^{-1}f_h\frac{\rho^\alpha}{\rho^{2\alpha}+2b\cos\pi\alpha \rho^\alpha+b^2}d\rho.
\end{equation}
To give the error bound for $\|u-u_h\|$  we introduce two auxiliary functions $w_\rho$ and $w_{\rho,h}$ that satisfy
\begin{equation*}
    (\rho+\cA)w_{\rho}=f \quad \mbox{ and }\quad  (\rho+\cA_h)w_{\rho,h}=f_h,
\end{equation*}
respectively. 

\begin{lemma}\label{regu-est}
    Suppose $s\in[0,1]$ and $g\in D(\cA^q)$ with $s\ge q$, then we have 
    $$\|\cA^s(\rho+\cA)^{-1}g\|\le c \wrho^{s-q-1} \|\cA^{q}g\|,$$
    where $\wrho=\max\{1,\rho\}$ and $c$ is a positive constant independent of $\rho$. The analogous inequality holds with $\cA^{*}$ replacing $\cA$ above.
    
\end{lemma}
\begin{proof}
    Considering the auxiliary equation for $g\in D(\cA^q)$
    \begin{equation*}
        (\rho+\cA)y_\rho=\cA^q g:=\wg,
    \end{equation*}
    test it by $y_\rho$ and take the real part on both sides, then we get
    \begin{equation*}
        \rho\|y_\rho\|^2+\Re(A(y_\rho,y_\rho))=\Re((\wg,y_\rho)).
    \end{equation*}
    Appealing to \eqref{prop-e} it follows
    \begin{equation*}
        \rho\|y_\rho\|^2+c_0\|y_\rho\|^2_V\le \Re((\wg,y_\rho))\le \|\wg\| \|y_\rho\|
    \end{equation*}
    which implies 
    $$\|y_\rho\|\le \rho^{-1}\|\wg\|,\quad \mbox{and}\quad c_0\|y_\rho\|_V^2\le \|\wg\|\|y_\rho\|.$$ 
    So one can get
    \begin{equation*}
        \|y_\rho\|\le c\min\left\{1,\rho^{-1}\right\}\|\wg\|.
    \end{equation*}
    Note that 
    \begin{equation*}
        \|\cA y_\rho\|=\|\cA(\rho+\cA)^{-1}\wg\|=\|\wg-\rho(\rho+\cA)^{-1}\wg\|\le 2\|\wg\|,
    \end{equation*}
    which implies $y_\rho\in D(\cA)$. 
    Thanks to \cite[Theorem 4.29]{lunardi2009interpolation} we know the  purely imaginary power of $\cA$ is bounded from $L^2(\Omega)$ to $L^2(\Omega)$, i.e., $\|\cA^{it}\|\le e^{\pi|t|/2}$ for $t\in \mathbb{R}$. That is, we can connect the domain of the fractional powers of $\cA$ to the interpolation space obtained by the complex method, say
    \begin{equation*}
        \left[L^2(\Omega),D(\cA)\right]_{s}=D(\cA^s),\quad \mbox{for }s\in [0,1].
    \end{equation*}
    So applying the interpolation inequality we have for $\forall w\in D(\cA)$,
    \begin{equation*}
        \|\cA^s w\|\le c \|w\|_{[L^2(\Omega),D(A)]_s}\le c\|\cA w\|^s \|w\|^{1-s}, \quad s\in(0,1).
    \end{equation*}
    Combining the above estimates we obtain for $g\in D(\cA^q)$
    \begin{equation*}
        \|\cA^s(\rho+\cA)^{-1}g\|=\|\cA^{s-q}(\rho+\cA)^{-1}\cA^q g\|=\|\cA^{s-q}y_\rho\| \le c\min\left\{1,\rho^{s-q-1}\right\}\|\wg\|.
    \end{equation*}
\end{proof}

We also need the following standard estimate for the projection operator $\pi_h$:
\begin{lemma}\label{proj-err}
    For $s\in [0,1]$ and $s+\sigma\le 2$ with $\sigma\ge0$, there exists a constant $C=C(s,\sigma)$ independent of the mesh size $h$ such that
    \begin{equation*}
        \|v-\pi_h v\|_{\widetilde{H}^s} \le C h^\sigma \|v\|_{\widetilde{H}^{s+\sigma}}, \quad \forall v\in  \widetilde{H}^{s+\sigma}.
    \end{equation*}
\end{lemma}
\begin{proof}
	See \cite[Lemma 5.1]{bonito2016numerical} for the case of $\sigma>0$. Since $\pi_h$ is stable in both $L^2(\Omega)$ and $H^1(\Omega)$, the case of $\sigma=0$ follows from interpolation. 
\end{proof}
 Appealing to \eqref{prop-e} and \eqref{prop-c} and the Galerkin orthogonality it follows
\begin{equation}\label{H1-est0}
\begin{aligned}
    &c_0\|w_\rho-w_{\rho,h}\|^2_V+\rho\|w_\rho-w_{\rho,h}\|^2\\
    &\le \Re\left(A(w_\rho-w_{\rho,h},w_\rho-w_{\rho,h})+\rho(w_\rho-w_{\rho,h},w_\rho-w_{\rho,h})\right)\\
    &=\Re\left(A(w_\rho-w_{\rho,h},w_\rho-\pi_hw_{\rho})+\rho(w_\rho-w_{\rho,h},w_\rho- \pi_h w_{\rho})\right)\\
    &\le c_1\|w_\rho-w_{\rho,h}\|_V\|w_\rho-\pi_hw_{\rho}\|_V +\rho\|w_\rho-w_{\rho,h}\|\|w_\rho-\pi_h w_{\rho}\|,
\end{aligned}
\end{equation}
which implies
\begin{equation}\label{H1-est1}
	\begin{aligned}
		\|w_\rho-w_{\rho,h}\|_V+\sqrt{\rho}\|w_\rho-w_{\rho,h}\|
		\le c\|w_\rho-\pi_hw_{\rho}\|_V +c\sqrt{\rho}\|w_\rho-\pi_h w_{\rho}\|.
	\end{aligned}
\end{equation}
Appealing to  \Cref{proj-err}, and \Cref{regu-est} we get
\begin{equation}\label{H1-est2}
	\begin{aligned}
		&\|w_\rho-w_{\rho,h}\|_V+\sqrt{\rho}\|w_\rho-w_{\rho,h}\|\le ch^{\sigma}\|w_\rho\|_{\widetilde{H}^{\sigma+1}}+c\sqrt{\rho}h^{\sigma'}\|w_\rho\|_{\widetilde{H}^{\sigma'}}\\
		&\le ch^{\sigma}\|\cA^{\frac{\sigma+1}{2}}w_\rho\|+c\sqrt{\rho}h^{\sigma'}\|\cA^{\frac{\sigma'}{2}}w_\rho\|\\
		&\le c h^{\sigma}\wrho^{\max(\frac{\sigma-\delta-1}{2},-1)}\|\cA^{\frac{\delta}{2}}f\|+c h^{\sigma'}\wrho^{\max(\frac{\sigma'-\delta-1}{2},-\frac{1}{2})}\|\cA^{\frac{\delta}{2}}f\|
	\end{aligned}
\end{equation}
where $\sigma\in [0,\gamma]$ and $\sigma'\in[0,\gamma+1]$, and the second line was obtained under Assumption \ref{equivalent-space}.
Now we turn to the $L^2$-estimate for $w_\rho-w_{\rho,h}$. By the Aubin-Nitsche trick we consider the following auxiliary equation
\begin{equation}\label{eqn-z-rho}
    \cA^*z_\rho +\rho z_{\rho}=w_\rho-w_{\rho,h}.
\end{equation}
Taking the duality and testing the obtained  equation by $w_\rho-w_{\rho,h}$ and integrating by parts we get
\begin{equation*}
    \|w_\rho-w_{\rho,h}\|^2=\left(w_\rho-w_{\rho,h},\cA^*z_\rho +\rho z_{\rho}\right)=A(w_\rho-w_{\rho,h},z_\rho)+\rho(w_\rho-w_{\rho,h},z_\rho).
\end{equation*}
By virtue of the Galerkin orthogonality and the Cauchy-Schwarz inequality, we get
\begin{equation}\label{L2-est}
\begin{aligned}
    &\|w_\rho-w_{\rho,h}\|^2
    =A(w_\rho-w_{\rho,h},z_\rho-\pi_h z_\rho)+\rho(w_\rho-w_{\rho,h},z_\rho-\pi_h z_\rho)\\
    &\le c\left(\|w_\rho-w_{\rho,h}\|_V+\sqrt{\rho}\|w_\rho-w_{\rho,h}\| \right)  \left(\|z_\rho-\pi_h z_\rho \|_V+\sqrt{\rho }\| z_\rho-\pi_h z_{\rho}\|\right).
\end{aligned}
\end{equation}
Using  \Cref{proj-err}, \Cref{regu-est} and recalling the definition of $z_\rho$ we have for $\wsigma\in [0,\gamma]$
\begin{equation}\label{error-z}
	\|z_\rho-\pi_h z_\rho\|_V+\sqrt{\rho}\|z_\rho-\pi_h z_\rho\| \le c h^{\wsigma}\wrho^{\frac{\wsigma-1}{2}}\|w_\rho - w_{\rho,h}\|,
\end{equation}
where we implicitly used Assumption \ref{equivalent-space}. Insert \eqref{H1-est2} and \eqref{error-z} into \eqref{L2-est} then for $\sigma,\wsigma\in[0,\gamma]$, $\sigma'\in[0,\gamma+1]$ we have
\begin{equation}\label{L2-est-f}
	\|w_\rho-w_{\rho,h}\|\le c\left(h^{\sigma+\wsigma}\wrho^{\max(\frac{\sigma+\wsigma-\delta}{2}-1,\frac{\wsigma-1}{2}-1)}+h^{\sigma'+\wsigma}\wrho^{\max(\frac{\sigma'+\wsigma-\delta}{2}-1,\frac{\wsigma}{2}-1)}\right)\|\cA^{\frac{\delta}{2}}f\|.
\end{equation}

Let $w_\rho^\varepsilon$ and $w_{\rho,h}^\varepsilon$ denote the solutions of the following equations
\begin{equation*}
	(\rho+\varepsilon\cA)w^\varepsilon_{\rho}=f \quad \mbox{ and }\quad  (\rho+\varepsilon\cA_h)w^\varepsilon_{\rho,h}=f_h,
\end{equation*}
then  dividing both sides of the two equations by $\varepsilon$ and utilizing \eqref{L2-est-f} we can get the following estimate:
\begin{lemma}\label{L2-est-ep}
	Suppose Assumption \ref{equivalent-space} holds and $f\in D(\cA^{\frac{\delta}{2}}) $ with $\delta\in [0,2]$, then we have
	\begin{equation}\label{est-w-rho-ep}
	\|w^\varepsilon_\rho-w^\varepsilon_{\rho,h}\|\le c \varepsilon^{-1} \left(h^{\sigma+\wsigma}\wrho^{\max(\frac{\sigma+\wsigma-\delta}{2}-1,\frac{\wsigma-1}{2}-1)}+h^{\sigma'+\wsigma}\wrho^{\max(\frac{\sigma'+\wsigma-\delta}{2}-1,\frac{\wsigma}{2}-1)}\right)\|\cA^{\frac{\delta}{2}}f\|,
	\end{equation}
	where $\sigma,\wsigma\in[0,\gamma]$, $\sigma'\in[0,\gamma+1]$ and $\wrho=\max\{1,\rho/\varepsilon\}$.
\end{lemma}
Now we are at the position to present the estimate for $u-u_h$.
\begin{theorem}\label{theorem-space}
        Suppose Assumption \ref{equivalent-space} holds and $f\in D(\cA^{\frac{\delta}{2}})$ with $\delta\in[0,2]$. Denote $\varepsilon=b^{-1/\alpha}$ then we have
        \begin{equation*}
           \|u-u_h\|\le  c_\alpha
           \left\{
           \begin{matrix}
           	 C(\varepsilon)c_h  h^{\min(2\alpha+\delta,2\gamma)}\|\cA^{\frac{\delta}{2}}f\|, & \quad \varepsilon\ge h^2(b\le h^{-2\alpha}),\\
           	\varepsilon^{\alpha}h^{\min(\delta,2\gamma)}\|\cA^{\frac{\delta}{2}}f\|, & \quad \varepsilon< h^2(b> h^{-2\alpha}),
           \end{matrix}
           \right.
       \end{equation*}
   where 
   \begin{equation*}
   	c_h, C(\varepsilon)=\left\{
\begin{matrix}
	   	1,& \varepsilon^{\alpha+\min(0,\frac{\delta}{2}-\gamma)}, & 2\alpha+\delta> 2\gamma,\\
	1+|\ln h|,& 1, & 2\alpha+\delta\le 2\gamma,
\end{matrix}
   	\right.
   \end{equation*}  
and $c_\alpha$ is a constant depending on $\alpha$ and is uniformly bounded for $\alpha\in (0,1)$.
\end{theorem}
\begin{proof}
    Let $\varepsilon=b^{-1/\alpha}$, then by change of variables we have
    \begin{equation*}
    \begin{aligned}
        \|u-u_h\|\le  \frac{\varepsilon^\alpha\sin\pi\alpha}{\pi}\int_{0}^\infty\frac{\rho^\alpha \|w^\varepsilon_{\rho}-w^\varepsilon_{\rho,h}\|}{\rho^{2\alpha}+2 \rho^\alpha \cos\pi\alpha +1}d\rho.
    \end{aligned}
    \end{equation*}
We split the integral into four parts, say $\int_0^\varepsilon+\int_\varepsilon^{1}+\int_1^{h^{-2}}+\int_{h^{-2}}^\infty$, denoted by $I_1,I_2,I_3$ and $I_4$,  respectively. Recalling the estimate for $\|w^{\varepsilon}_{\rho}-w^{\varepsilon}_{\rho,h}\|$ in  \Cref{L2-est-ep}, we further split $\{I_{k}\}_{k=2}^4$ into $\{I_{k1},I_{k2}\}_{k=2}^4$ corresponding to the two terms on the right hand side of \eqref{est-w-rho-ep}.

We first focus on the case of $\alpha\in (0,1/2]$. Use \Cref{L2-est-ep} and take $\sigma=\wsigma=\sigma'=\gamma$, then we obtain
\begin{equation}\label{I1-a}
    I_1\le c h^{2\gamma}\varepsilon^{\alpha-1}\frac{\sin \pi\alpha}{\pi}\int_0^\varepsilon\rho^\alpha d\rho
    = c\varepsilon^{2\alpha}h^{2\gamma}\frac{\sin \pi\alpha}{(1+\alpha)\pi}.
\end{equation} 
For $I_2$, we have 
\begin{equation*}
\begin{aligned}
    I_{21}&\le  c \varepsilon^{2\alpha}\frac{\sin \pi\alpha}{\pi}h^{\sigma+\wsigma}\int_{\varepsilon}^{1}\left(\frac{\rho}{\varepsilon}\right)^{\alpha+\max(\frac{\sigma+\wsigma-\delta}{2}-1,\frac{\wsigma-1}{2}-1)}d\frac{\rho}{\varepsilon}\\
    &\le c \frac{\sin \pi\alpha}{\pi}h^{\sigma+\wsigma}
    \left\{
    \begin{matrix}
    	\frac{\varepsilon^{2\alpha}+\varepsilon^{\alpha+\frac{1}{2}\min(\delta-\sigma-\wsigma,1-\wsigma)}}{\alpha+\max(\frac{\sigma+\wsigma-\delta}{2},\frac{\wsigma-1}{2})}, & \alpha+\max(\frac{\sigma+\wsigma-\delta}{2},\frac{\wsigma-1}{2})\neq 0, \\
    	\varepsilon^{2\alpha} |\ln\varepsilon|, & \mbox{otherwise},
    \end{matrix}
    \right.    
\end{aligned}
\end{equation*}
and 
\begin{equation*}
	\begin{aligned}
		I_{22}&\le  c \varepsilon^{2\alpha}\frac{\sin \pi\alpha}{\pi}h^{\sigma'+\wsigma}\int_{\varepsilon}^{1}\left(\frac{\rho}{\varepsilon}\right)^{\alpha+\max(\frac{\sigma'+\wsigma-\delta}{2}-1,\frac{\wsigma}{2}-1)}d\frac{\rho}{\varepsilon}\\
		&\le c \frac{\sin \pi\alpha}{\pi}h^{\sigma'+\wsigma}
		\left\{
		\begin{matrix}
			\frac{\varepsilon^{2\alpha}+\varepsilon^{\alpha+\frac{1}{2}\min(\delta-\sigma'-\wsigma,-\wsigma)}}{\alpha+\max(\frac{\sigma'+\wsigma-\delta}{2},\frac{\wsigma}{2})}, & \alpha+\max(\frac{\sigma'+\wsigma-\delta}{2},\frac{\wsigma}{2})\neq 0, \\
			\varepsilon^{2\alpha}|\ln \varepsilon|, & \mbox{otherwise}.
		\end{matrix}
		\right.    
	\end{aligned}
\end{equation*}
For $I_3$, replace the denominator by $\rho^{2\alpha}$ then we get
\begin{equation*}
	\begin{aligned}
		&I_{31}\le  c \frac{\sin \pi\alpha}{\pi}h^{\sigma+\wsigma}\int_{1}^{h^{-2}}\left(\frac{\rho}{\varepsilon}\right)^{-\alpha+\max(\frac{\sigma+\wsigma-\delta}{2},\frac{\wsigma-1}{2})-1}d\frac{\rho}{\varepsilon}\\
		&\le c \frac{\sin \pi\alpha}{\pi}h^{\sigma+\wsigma}
		\left\{
		\begin{matrix}
			\frac{\varepsilon^{\alpha+\frac{1}{2}\min(\delta-\sigma-\wsigma,1-\wsigma)}}{\max(\frac{\sigma+\wsigma-\delta}{2},\frac{\wsigma-1}{2})-\alpha}(1+h^{2\alpha+\min(\delta-\sigma-\wsigma,1-\widetilde{\sigma})}), & \max(\frac{\sigma+\wsigma-\delta}{2},\frac{\wsigma-1}{2})\neq \alpha, \\
			2|\ln h|, & \mbox{otherwise},
		\end{matrix}
		\right.    
	\end{aligned}
\end{equation*}
and
\begin{equation*}
	\begin{aligned}
		&I_{32}\le  c \frac{\sin \pi\alpha}{\pi}h^{\sigma'+\wsigma}\int_{1}^{h^{-2}}\left(\frac{\rho}{\varepsilon}\right)^{-\alpha+\max(\frac{\sigma'+\wsigma-\delta}{2},\frac{\wsigma}{2})-1}d\frac{\rho}{\varepsilon}\\
		&\le c \frac{\sin \pi\alpha}{\pi}h^{\sigma'+\wsigma}
		\left\{
		\begin{matrix}
			\frac{\varepsilon^{\alpha+\frac{1}{2}\min(\delta-\sigma'-\wsigma,-\wsigma)}}{\max(\frac{\sigma'+\wsigma-\delta}{2},\frac{\wsigma}{2})-\alpha}(1+h^{2\alpha+\min(\delta-\sigma'-\wsigma,-\widetilde{\sigma})}), & \max(\frac{\sigma'+\wsigma-\delta}{2},\frac{\wsigma}{2})\neq \alpha, \\
			2|\ln h|, & \mbox{otherwise}.
		\end{matrix}
		\right.    
	\end{aligned}
\end{equation*}
For the last term $I_4$, take $\wsigma =0$ then we have for $\sigma\neq2\alpha+\delta$
\begin{equation*}
	\begin{aligned}
		I_{41}&\le  c \frac{\sin \pi\alpha}{\pi}h^{\sigma+\wsigma}\int_{h^{-2}}^{+\infty}\left(\frac{\rho}{\varepsilon}\right)^{-\alpha+\max(\frac{\sigma+\wsigma-\delta}{2},\frac{\wsigma-1}{2})-1}d\frac{\rho}{\varepsilon}\\
		&=  \frac{c\sin \pi\alpha}{\pi|\frac{1}{2}\max(\sigma-\delta,-1)-\alpha|}\varepsilon^{\alpha+\frac{1}{2}\min(\delta-\sigma,1)} h^{2\alpha+\min(\delta,1+\sigma)},
	\end{aligned}
\end{equation*}
and for $\sigma'\neq 2\alpha+\delta$
\begin{equation*}
	\begin{aligned}
		I_{42}&\le  c \frac{\sin \pi\alpha}{\pi}h^{\sigma'+\wsigma}\int_{h^{-2}}^{+\infty}\left(\frac{\rho}{\varepsilon}\right)^{-\alpha+\max(\frac{\sigma'+\wsigma-\delta}{2},\frac{\wsigma}{2})-1}d\frac{\rho}{\varepsilon}\\
		&=  \frac{c\sin \pi\alpha}{\pi|\frac{1}{2}\max(\sigma'-\delta,0)-\alpha|}\varepsilon^{\alpha+\frac{1}{2}\min(\delta-\sigma',0)} h^{2\alpha+\min(\delta,\sigma')}.
	\end{aligned}
\end{equation*}
Next we turn to the situation of $\alpha\in (1/2,1)$. In this case we still take $\sigma=\wsigma=\sigma'=\gamma$ then it follows
\begin{equation*}
    \begin{aligned}
       I_1 &\le \varepsilon^{\alpha-1}h^{2\gamma}\frac{\sin \pi\alpha}{\pi}\int_0^\varepsilon\frac{\rho^\alpha}{(\rho^\alpha+\cos\pi\alpha)^2+\sin^2\pi\alpha}d\rho\\
        &= \varepsilon^{\alpha-1}h^{2\gamma}\frac{\sin \pi\alpha}{\pi\alpha}\int_0^{\varepsilon^\alpha}\frac{\rho^{1/\alpha}}{(\rho+\cos\pi\alpha)^2+\sin^2\pi\alpha}d\rho\\
        &\le  \varepsilon^{\alpha}h^{2\gamma}\frac{1}{\pi\alpha}\left.\arctan\frac{\rho+\cos\pi\alpha}{\sin\pi\alpha}\right|_{0}^{\varepsilon^\alpha}\le \frac{\varepsilon^{\alpha}}{\alpha}h^{2\gamma}.
    \end{aligned}
\end{equation*}
Consider $I_2$, by  change of variable $\rho^\alpha\rightarrow \rho$ and the Cauchy-Schwarz inequality we get 
\begin{equation*}
\begin{aligned}
	    I_{21}&\le  c h^{\sigma+\wsigma}\varepsilon^{\alpha+\frac{1}{2}\min(\delta-\sigma-\wsigma,1-\wsigma)}\frac{\sin \pi\alpha}{\alpha\pi}\int_{\varepsilon^\alpha}^{1}\frac{\rho^{\frac{1}{2\alpha}\max(\sigma+\wsigma-\delta,\wsigma-1)}}{(\rho+\cos\pi\alpha)^2+\sin^2\pi\alpha}d\rho\\
	    &\le ch^{\sigma+\wsigma}\varepsilon^{\alpha+\frac{1}{2}\min(\delta-\sigma-\wsigma,1-\wsigma)}\frac{1}{\pi\alpha}\max\left\{1,\varepsilon^{\frac{1}{2}\max(\sigma+\wsigma-\delta,\wsigma-1)}\right\}\left.\arctan\frac{\rho+\cos\pi\alpha}{\sin\pi\alpha}\right|_{\varepsilon^\alpha}^{1}\\
	    &\le \frac{ch^{\sigma+\wsigma}}{\alpha}\max\left\{\varepsilon^{\alpha+\frac{1}{2}\min(\delta-\sigma-\wsigma,1-\wsigma)},\varepsilon^{\alpha}\right\}
\end{aligned}
\end{equation*}
and similarly 
\begin{equation*}
	\begin{aligned}
		I_{22}&\le \frac{ch^{\sigma'+\wsigma}}{\alpha}\max\left\{\varepsilon^{\alpha+\frac{1}{2}\min(\delta-\sigma'-\wsigma,-\wsigma)},\varepsilon^{\alpha}\right\}.
	\end{aligned}
\end{equation*}
Next we turn to $I_3$. Let $l=|2\cos\pi\alpha|^{1/\alpha}$ then for $\rho\ge l$ it follows $\rho^\alpha+\cos\pi\alpha\ge \frac{1}{2}\rho^\alpha$, so if $l\le 1$, or equivalently  $\alpha\in(\frac{1}{2},\frac{2}{3}]$, by virtue of the following inequality
\begin{equation}\label{den-est}
	\rho^{2\alpha}+2\rho^\alpha\cos\pi\alpha+1=(\rho^\alpha+\cos\pi\alpha)^2+\sin^2\pi\alpha\ge \frac{1}{4}\rho^{2\alpha}
\end{equation}
we can get the same estimates for $I_{31}$ and $I_{32}$ as the case $\alpha\le \frac{1}{2}$. 
Otherwise, we split $I_{3k}(k=1,2)$ into two terms $\int_{1}^{l}+\int_{l}^{h^{-2}}$ denoted by $I_{3k}^1$ and $I_{3k}^2$, respectively. By change of variable $\rho^\alpha\rightarrow \rho$ for $I_{3k}^1$ and repeating the procedures as for  $I_{2}$ we have 
\begin{equation*}
	\begin{aligned}
		I_{31}^1&\le  c h^{\sigma+\wsigma}\varepsilon^{\alpha+\frac{1}{2}\min(\delta-\sigma-\wsigma,1-\wsigma)}\frac{\sin \pi\alpha}{\alpha\pi}\int_{1}^{l^\alpha}\frac{\rho^{\frac{1}{2\alpha}\max(\sigma+\wsigma-\delta,\wsigma-1)}}{(\rho+\cos\pi\alpha)^2+\sin^2\pi\alpha}d\rho\\
		&\le ch^{\sigma+\wsigma}\varepsilon^{\alpha+\frac{1}{2}\min(\delta-\sigma-\wsigma,1-\wsigma)}\frac{1}{\pi\alpha}\left.\arctan\frac{\rho+\cos\pi\alpha}{\sin\pi\alpha}\right|_{\varepsilon^\alpha}^{1}\\
		&\le \frac{ch^{\sigma+\wsigma}}{\alpha}\varepsilon^{\alpha+\frac{1}{2}\min(\delta-\sigma-\wsigma,1-\wsigma)}
	\end{aligned}
\end{equation*}
and similarly
\begin{equation*}
	\begin{aligned}
		I_{32}^1&\le \frac{ch^{\sigma'+\wsigma}}{\alpha}\varepsilon^{\alpha+\frac{1}{2}\min(\delta-\sigma'-\wsigma,-\wsigma)}.
	\end{aligned}
\end{equation*}
For $I_{3k}^2$ and $I_4$, utilizing \eqref{den-est} and repeating the same procedures as for $I_3$ and $I_4$ for the case $\alpha\le \frac{1}{2}$ then the same estimates follow. 

If $\varepsilon\ge h^{2}$, then take $\sigma=\wsigma=\min\{\alpha+\frac{\delta}{2},\gamma\}$ for $I_{21}$, $I_{31}$  and $\wsigma=0$, $\sigma'=\min\{2\gamma,2\alpha+\delta\}$ for $I_{22}$, $I_{32}$ it follows 
\begin{equation}\label{I2-F}
	I_2\le c_\alpha\left\{ 
	\begin{matrix}
		  \varepsilon^{\alpha+\min(0,\frac{\delta}{2}-\gamma)}h^{2\gamma}, &\quad 2\alpha+\delta> 2\gamma, \\
		h^{2\alpha+\delta}, &\quad 2\alpha+\delta\le 2\gamma,\\
	\end{matrix}
	\right.
\end{equation}
and 
\begin{equation}\label{I3-F}
	I_3\le c_\alpha\left\{ 
	\begin{matrix}
		\varepsilon^{\alpha+\min(\frac{\delta}{2}-\gamma,0)} h^{2\gamma}, &\quad 2\alpha+\delta> 2\gamma, \\
		|\ln h|h^{2\alpha+\delta}, &\quad 2\alpha+\delta\le 2\gamma.\\
	\end{matrix}
	\right.
\end{equation}
 If $\varepsilon\le h^{2}$, take $\sigma=\wsigma=\min(\gamma,\frac{\delta}{2})$ in $I_{21}$, $I_{31}$ and $\widetilde{\sigma}=0$, $\sigma'=\min(\delta,2\gamma)$ in $I_{22}$ and $I_{32}$ then
\begin{equation}\label{I2-F1}
	I_2,I_3\le c_\alpha \varepsilon^{\alpha} h^{\min(\delta,2\gamma)}.
\end{equation} 
Lastly for $I_4$, we take $\sigma=\min(\max(0,\delta-1),\gamma)$ and $\sigma'=\min(\delta,\gamma+1)$ then it follows
\begin{equation}\label{I4-F}
	I_4\le c_\alpha \varepsilon^{\alpha}h^{2\alpha+\min(\delta,1+\gamma)}. 
\end{equation}
The boundedness of $c_\alpha$ with respect to $\alpha$ is due to the factor $\sin\pi\alpha$.
\end{proof}
\begin{remark}\label{remark-space}
	For $\gamma=1$,  \Cref{theorem-space} can be simplified into the following form
        \begin{equation*}
	\|u-u_h\|\le  c_\alpha
	\left\{
	\begin{matrix}
		 C'(\varepsilon)c_h  h^{\min(2\alpha+\delta,2)}\|\cA^{\frac{\delta}{2}}f\|, & \quad b\le h^{-2\alpha},\\
		b^{-1}h^{\delta}\|\cA^{\frac{\delta}{2}}f\|, & \quad b> h^{-2\alpha},
	\end{matrix}
	\right.
\end{equation*}
where $C'(\varepsilon)=\varepsilon^{\max(\alpha+\frac{\delta}{2}-1,0)}$ and $c_h=1+|\ln h|$ if $2\alpha+\delta\le 2$,  otherwise $c_h=1$.
\end{remark}

\section{Error analysis for quadrature}
In this section we focus on the quadrature error. Appealing to \eqref{u_h}, we transplant the domain from $(0,+\infty)$ to $(-\infty,+\infty)$ by letting $\rho^\alpha=e^s$ then 
\begin{equation}\label{uh}
    \begin{aligned}
    u_h=	\frac{\sin\pi\alpha}{\alpha\pi}\int_{-\infty}^\infty\frac{(1+e^{-s/\alpha}\cA_h)^{-1}f_h}{e^{s}+2b\cos\pi\alpha +b^2e^{-s}}ds.
\end{aligned}
\end{equation}
To evaluate the integral we apply the trapezoidal rule, say
\begin{equation}\label{buh}
    u_h\approx u_{h,\tau}=\frac{\sin \pi \alpha}{\alpha\pi}\tau\sum_{n=-\infty}^{\infty}\frac{(1+e^{-n\tau/\alpha}\cA_h)^{-1}f_h}{e^{n\tau}+2b\cos\pi\alpha +b^2e^{-n\tau}}.
\end{equation}
The exponentially decaying property of the integrand allows us to cut the summation  to $n=-M,-M+1,\cdots, N-1,N$ with desired accuracy. That is, in practice our scheme is
\begin{equation}\label{Uh}
    U_{h,\tau}^{M,N}=\frac{\sin \pi \alpha}{\alpha\pi}\tau\sum_{n=-M}^{N}\frac{(1+e^{-n\tau/\alpha}\cA_h)^{-1}f_h}{e^{n\tau}+2b\cos\pi\alpha +b^2e^{-n\tau}}.
\end{equation}
Next we shall give the error between $u_h$ and $u_{h,\tau}$ then $u_h-U_{h,\tau}^{M,N}$ follows from the triangle inequality. To this end we need the resolvent estimate for the discrete operator $\cA_h$.
\begin{lemma}\label{resol-est}
    Let $\cA$ be the regularly accretive operator with index $\beta$ and $\cA_h$ be the discrete counterpart. The resolvent of $\cA_h$ satisfies the inequality
    \begin{equation}\label{resol-est-dis-1}
        \left\|(z\cI -\cA_h)^{-1}\right\|\le \left\{
        \begin{matrix}
            &\left[|z|\sin(\left|\arg z\right|-\theta)\right]^{-1},&\theta<\left|\arg z\right|\le \frac{\pi}{2}+\theta, \\
            & |z|^{-1}, &\left|\arg z\right| >\frac{\pi}{2}+\theta,
        \end{matrix}	
        \right.
    \end{equation}
where $\theta=\arctan \beta<\frac{\pi}{2}$. And for $\Re{z}\le \frac{c_0}{2}$ it follows
    \begin{equation}\label{resol-est-dis-2}
    \left\|(z\cI -\cA_h)^{-1}\right\|\le \frac{2}{c_0}.
    \end{equation}
\end{lemma}

\begin{proof}
	    Note that $|\Im\left(\cA v,v\right)|\le \beta\Re\left(\cA v,v\right)$ for all $v\in V$ implies the numerical range of $\cA$ is a subset of the sector $\Sigma'_\theta:=\{z\in \mathbb{C}; |\arg z|\le \theta\}$ with $\theta=\arctan \beta<\frac{\pi}{2}$. 
	Since all the eigenvalues of $\cA_h$ lie in the numerical range of $\cA_h$ which is a subset of the numerical range of $\cA$,  the distance between $z$ and the the numerical range of $\cA_h$ is bigger than the distance of $z$ to the sector $\Sigma'_{\theta}$, denoted by $d$, say
	\begin{equation*}
		\left|z-\frac{(\cA_h v,v)}{\|v\|^2}\right|\ge d,\quad \forall v\in V_h, v\ne 0.
	\end{equation*}
The desired estimates follows from repeating the  arguments in the proof of \cite[Theorem 2.2]{kato1961}.
\end{proof}

Denote 
\begin{equation*}
    \kappa(x,s)=\frac{(1+e^{-s/\alpha}\cA_h)^{-1}f_h}{e^{s}+2b\cos\pi\alpha +b^2e^{-s}},
\end{equation*}
then we know the poles of $\kappa(\cdot,s)$ are those that make $(1+e^{-s/\alpha}\cA_h)^{-1}$ not invertible and those which are the zeros of $e^{s}+2b\cos\pi\alpha +b^2e^{-s}$. After manipulations we obtain the poles of the integrand are
$$s_1(k)=\alpha \ln |\sigma(\cA_h)|+\alpha(\arg \sigma(-\cA_h)+2k\pi)\i\quad \mbox{and} \quad s_2(k)= \ln b + \left((2k-1)\pi\pm\alpha\pi\right)\i$$
where $\sigma(\cA_h)$ denotes the spectrum of $\cA_h$ and $k=0,\pm 1,\pm 2\cdots$. Calculations show that 
\begin{equation}\label{m-pole1}
    \min_{k}\left|\Im\left(s_1(k)\right)\right|\ge \left|\Im\left(s_1(0)\right)\right|\ge\alpha(\pi-\arctan\beta):=\kappa_1
\end{equation}
and
\begin{equation}\label{m-pole2}
    \min_{k}\left|\Im\left(s_2(k)\right)\right|\ge (1-\alpha)\pi:=\kappa_2.
\end{equation}
Let $\kappa'$ be any number in $(0, \min\{\kappa_1,\kappa_2\}]$, then thanks to \cite[Theorem 5.1]{trefethen2014exponentially} we know
\begin{equation}\label{est-L2-err}
    \|u_h-u_{h,\tau}\|\le \frac{2\sin\alpha\pi}{\alpha\pi}\frac{\sup_{p\in(-\kappa',\kappa')}M(p)}{e^{2\pi \kappa'/\tau}-1},
\end{equation}
with 
\begin{equation*}
    M(p)=\int_{-\infty}^{+\infty}\|\kappa(x,s+p\i)\|ds<\infty,\quad  p\in (-\kappa',\kappa').
\end{equation*}
Obviously, a bigger $\kappa'$ leads to faster convergence if $M(p)$ is independent of $\kappa'$. However, it is easy to observe that $\|\kappa(x,s+p\i)\|$ is unbounded 
when $\kappa'=\min\{\kappa_1,\kappa_2\}$ and $|p|$ gets close to $\kappa'$,
which may lead to $M(p)$ going to infinity. Next we shall give the bound of $M(p)$ with respect to $p$.

Appealing to Lemma \ref{resol-est} we have for $\pi-\frac{|p|}{\alpha}>\theta$
\begin{equation*}
	\left\|\big(-e^{(s+p\i)/\alpha}-\cA_h\big)^{-1}\right\|\le
		e^{-s/\alpha}\sin^{-1}\left(\min\Big\{\pi-\frac{|p|}{\alpha}-\theta,\frac{\pi}{2}\Big\}\right),
\end{equation*}
which implies for $|p|<\kappa_1$
\begin{equation*}
	\left\|\big(1+e^{-(s+p\i)/\alpha}\cA_h\big)^{-1}\right\|\le 
		\sin^{-1}\left(\min\big\{\pi-\frac{|p|}{\alpha}-\theta,\frac{\pi}{2}\big\}\right)
\end{equation*}
with $\theta=\arctan \beta<\frac{\pi}{2}$. So  for $|p|<\min\{\kappa_1,\kappa_2\}$  we have 
\begin{equation}\label{est-num}
	\left\|\big(1+e^{-(s+p\i)/\alpha}\cA_h\big)^{-1}f_h\right\|\le 
		\sin^{-1}\left(\min\Big\{\frac{\kappa_1-|p|}{\alpha},\frac{\pi}{2}\Big\}\right)\;\|f_h\|.
\end{equation} 
 \begin{remark}
	In fact, Lemma \ref{resol-est} implies for $s\le \alpha\ln\frac{c_0}{2}$
	\begin{equation*}
		\left\|\big(1+e^{-(s+p\i)/\alpha}\cA_h\big)^{-1}f_h\right\|\le e^{s/\alpha-\ln\frac{c_0}{2}} \; \|f_h\|,
	\end{equation*}
	which ensures that the integrand $\|\kappa(x,s+p\i)\|$ still decays exponentially for $s\rightarrow -\infty$ regardless of whether $b=0$ or not.
\end{remark}

Let $d(z)=e^{z}+2b\cos\pi\alpha +b^2e^{-z}$ and denote $z_0=\ln b \pm (1-\alpha)\pi \i$ that are the zeros of $d(z)$ with minimum $|\Im z|$. Since $d(z)$ is analytic on the whole $z$-plane, a Taylor expansion on $z=z_0$ gives
\begin{equation*}
\begin{aligned}
    d(z)&=\sum_{n=1}^{+\infty}\frac{e^{z_0}+(-1)^nb^2e^{-z_0}}{n!}(z-z_0)^n\\
    &=2b\i\sin(\pm\alpha\pi)\sum_{n=1}^{\infty}\frac{(z-z_0)^{2n-1}}{(2n-1)!}-2b\cos(\alpha\pi)\sum_{n=1}^{\infty}\frac{(z-z_0)^{2n}}{(2n)!}.
\end{aligned}
\end{equation*}
So for $|z-z_0|\le \frac{1}{2}\min\{1,|\tan\alpha\pi|\}:=s_\alpha$ with $\alpha\in (0,1)$, it follows
\begin{equation}\label{est-den}
\begin{aligned}
    |d(z)|&\ge 2b\sin(\alpha\pi)|z-z_0|-2b\sin(\alpha\pi)\sum_{n=1}^{\infty}\frac{|z-z_0|^{2n+1}}{(2n+1)!}-2b|\cos(\alpha\pi)|\sum_{n=1}^{\infty}\frac{|z-z_0|^{2n}}{(2n)!}\\
    &\ge 2b\sin(\alpha\pi)|z-z_0|-\frac{b}{3}\sin(\alpha\pi)\frac{|z-z_0|^3}{1-|z-z_0|^2}-b|\cos(\alpha\pi)|\frac{|z-z_0|^2}{1-|z-z_0|^2}\\
    &\ge \frac{11}{9}b\sin(\alpha\pi)|z-z_0|>b\sin(\alpha\pi) |z-z_0|.
\end{aligned}
\end{equation}

Let $\epsilon=\kappa_2-|p|$, then for $\epsilon\le s_\alpha$ we know there exists some $s^*\ge 0$ such that $|s^*+\ln b+p\i-z_0|=s_\alpha$. In this case we split the integral $\int_{-\infty}^{\infty}\|\kappa(x,s+p \i)\| ds$ into three parts, say, $\int_{-\infty}^{\ln b-s^*}\cdot \; ds$, $\int_{\ln b-s^*}^{\ln b+s^*}\cdot \; ds$ and $\int_{\ln b+s^*}^{+\infty}\cdot \; ds$. Then for the second part, appealing to \eqref{est-num} and \eqref{est-den} we have
\begin{equation}\label{est-kappa-2}
\begin{aligned}
    &\int_{\ln b-s^*}^{\ln b+s^*}\|\kappa(x,s+p \i)\| ds\le \frac{2 \sin^{-1}\left(\min\big\{\frac{\kappa_1-|p|}{\alpha},\frac{\pi}{2}\big\}\right)}{b\sin\alpha\pi} \int_{0}^{s^*}\frac{\|f_h\|}{\sqrt{\epsilon^2+s^2}} ds\\
    &= \frac{2\sin^{-1}\left(\min\big\{\frac{\kappa_1-|p|}{\alpha},\frac{\pi}{2}\big\}\right)}{b\sin\alpha\pi}\left(\ln\left(s^*+\sqrt{\epsilon^2+s^*}\right)-\ln \epsilon\right)\|f_h\|.
\end{aligned}
\end{equation}
Now, we consider the first and third terms. In fact, trivial calculations show that for $\forall z=s+p\i$
\begin{equation}\label{derivative-dz}
    \begin{aligned}
        \frac{\partial}{\partial s}|d(z)|^2&\,=2(e^s+b^2 e^{-s}+2b\cos\alpha\pi\cos p)(e^s-b^2 e^{-s})\\
        &\hskip-8pt\left\{
            \begin{aligned}
            &\ge 2(e^s+b^2 e^{-s}-2b)(e^s-b^2 e^{-s}), &\quad \text{if } s > \ln b, \\
            & =0,& \quad \text{if } s = \ln b, \\
            &\le 2(e^s+b^2 e^{-s}-2b)(e^s-b^2 e^{-s}), &\quad \text{if } s < \ln b,
            \end{aligned}
        \right.
    \end{aligned}
\end{equation}
which implies $\partial_s |d(s +p\i)| > 0$ for $s > \ln b$ and $\partial_s |d(s+p\i)|< 0$ for $s<\ln b$. That is  $|d(s+p\i)|$ increases as $|s-\ln b|$ increases. Denote 
\begin{equation*}
    \underline{d}(s)=\left(e^{2s}+b^4e^{-2s}-4be^{s}-4b^3e^{-s}+(6+s_\alpha^2\sin^2(\alpha\pi))b^2\right)^{1/2}
\end{equation*}
then it is easy to check that $\underline{d}^2(s)\le |d(z)|^2$ for $s\in (-\infty,+\infty)$. So
\begin{equation}\label{est-kappa-1-3}
\begin{aligned}
    &\int_{-\infty}^{\ln b-s^*} \|\kappa(x,s+p \i)\|\; ds+ \int_{\ln b+s^*}^{\infty} \|\kappa(x,s+p \i)\|\; ds \\
    &\le 2 \int_{s^*}^\infty\frac{\sin^{-1}\left(\min\Big\{\frac{\kappa_1-|p|}{\alpha},\frac{\pi}{2}\Big\}\right)\|f_h\|}{b(e^{2s}+e^{-2s}-4e^{s}-4e^{-s}+6+s_\alpha^2\sin^2(\alpha\pi))^{1/2}}ds\\
    &\le \frac{c(s^*)}{b}\sin^{-1}\left(\min\Big\{\frac{\kappa_1-|p|}{\alpha},\frac{\pi}{2}\Big\}\right)\|f_h\|
\end{aligned}
\end{equation}
with $c(s^*)$ depending only on $s^*$ and $\alpha$. 

For $\epsilon>s_\alpha$, say, $|p|<(1-\alpha)\pi-s_\alpha$, noticing that 
\begin{equation*}
\begin{aligned}
        |d(s+p\i)|\ge |d(\ln b+p\i)|&=2b|\cos p+\cos\pi\alpha|=4b\left|\cos\left(\frac{|p|+\pi\alpha}{2}\right)\cos\left(\frac{|p|-\pi\alpha}{2}\right)\right|
\end{aligned}
\end{equation*}
and
\begin{equation*}
    \frac{\pi\alpha}{2}\le\frac{|p|+\pi\alpha}{2}\le \frac{\pi-s_\alpha}{2}, \quad -\frac{\pi\alpha}{2}\le \frac{|p|-\pi\alpha}{2}\le \frac{\pi}{2}-\frac{s_\alpha}{2}-\pi\alpha,
\end{equation*}
we have for $|p|\le \kappa'$
\begin{equation}\label{est-den-2}
    |d(s+p\i)|\ge 4b\min\left\{\cos\frac{\pi\alpha}{2}, \sin\frac{s_\alpha}{2}\right\}
                  \min\left\{\cos\frac{\pi\alpha}{2}, \sin\Big(\frac{s_\alpha}{2} + \pi\alpha\Big)\right\}
\end{equation}
which is independent of $p$. So taking $s^*=0$ and repeating the procedures of $\epsilon\le s_\alpha$  it follows
\begin{equation}\label{est-kappa}
    \int_{-\infty}^{\infty}\|\kappa(x,s+p \i)\| ds\le  \frac{c_\alpha}{b}\sin^{-1}\left(\min\Big\{\frac{\kappa_1-|p|}{\alpha},\frac{\pi}{2}\Big\}\right)\|f_h\|
\end{equation}
with $c_\alpha$ depending only on $\alpha$. 

The analysis above implies for $p=\kappa_2-\epsilon$ with $\forall \epsilon \in (0,\kappa_2)$, 
\begin{equation}\label{est-kappa-F}
    \int_{-\infty}^{\infty}\|\kappa(x,s+p \i)\| ds\le  \frac{\tilde{c}_\alpha}{b} (1+|\ln\epsilon|)\sin^{-1}\left(\min\Big\{\frac{\kappa_1-|p|}{\alpha},\frac{\pi}{2}\Big\}\right)\|f_h\|
\end{equation}
with $\tilde{c}_\alpha$ depending only on $\alpha$. This gives the following theorem:

\begin{theorem}\label{theorem-quadrature}
    Let $A(\cdot,\cdot)$ denote the sesquilinear form given in \eqref{eqn-sesquilinear} that satisfies \eqref{prop-e} and \eqref{prop-c}, and denote by $\cA$ the regular accretive operator associated with it and $\cA_h$ the corresponding discrete counterpart. Let $\kappa_1=\alpha(\pi-\arctan\beta)$, $\kappa_2=(1-\alpha)\pi$, then for $\tau<\pi-\arctan\beta$ we have
    \begin{equation*}
        \|u_h-u_{h,\tau}\|\le \frac{C(\alpha)}{b} C(\tau) e^{-2\pi\min(\kappa_1,\kappa_2)/\tau}\|f_h\|,
    \end{equation*}
where 
\begin{equation*}
	C(\tau)=\left\{
	\begin{matrix}
		\sin^{-1}\left(\min\Big\{\frac{\kappa_1-\kappa_2}{\alpha}+{\frac{1-\alpha}{\alpha}}\tau,\frac{\pi}{2}\Big\}\right) (1+|\ln \tau|),&\quad \kappa_1>\kappa_2,\\
		\sin^{-1}\left(\min\Big\{\tau,\frac{\pi}{2}\Big\}\right) (1+|\ln (\kappa_2-\kappa_1+\alpha\tau)|),&\quad \kappa_1\le\kappa_2,\\
	\end{matrix}
	\right.
\end{equation*}
and $C(\alpha)$ is a constant depends only on $\alpha$.
\end{theorem}
\begin{proof}
    For $\kappa_1> \kappa_2$, say, $\alpha > \frac{\pi}{2\pi-\theta}$ with $\theta=\arctan \beta$, appealing to  \eqref{est-kappa-2}, \eqref{est-kappa-1-3} and \eqref{est-kappa}, take $\kappa'=\kappa_2-(1-\alpha)\tau $ we have
    \begin{equation*}
        \sup_{p\in(-\kappa',\kappa')}  M(p) \le \frac{C(\alpha)}{b} \sin^{-1}\left(\min\Big\{\frac{\kappa_1-\kappa_2}{\alpha}+\frac{1-\alpha}{\alpha}\tau,\frac{\pi}{2}\Big\}\right) (1+|\ln \tau|)\|f_h\|,
    \end{equation*}
and for $\kappa_1\le \kappa_2$, taking $\kappa'= \kappa_1 - \alpha\tau$ it follows
\begin{equation*}
\begin{aligned}
    \sup_{p\in(-\kappa',\kappa')}  M(p)  \le \frac{C(\alpha)}{b} \sin^{-1}\left(\min\Big\{\tau,\frac{\pi}{2}\Big\}\right) (1+|\ln (\kappa_2-\kappa_1+\alpha\tau)|)\|f_h\|
\end{aligned}
\end{equation*}
that yields the desired estimate.
\end{proof}
\begin{remark}
	One can observe that in the worst case scenario, $C(\tau)=\mathcal{O}(\tau^{-1}(1+|\ln \tau|))$ as $\tau \rightarrow 0$, which is negligible compared with the exponentially decaying term $e^{-2\pi\min(\kappa_1,\kappa_2)/\tau}$. In fact, for $\alpha$ close to $1$, $C(\tau)=\mathcal{O}(1+|\ln \tau|).$
\end{remark}
 Applying \Cref{resol-est} we have 
\begin{equation*}
	\|(1+e^{-n\tau/\alpha}\cA_h)^{-1} f_h\|\le 
	\left\{ 
	\begin{matrix}
		\frac{2}{c_0} e^{n\tau/\alpha}\|f_h\|, & n\le 0,\\
		\|f_h\|, & n>0,
	\end{matrix}
\right.	
\end{equation*}
so
\begin{equation*}
	\|u_{h,\tau}-U_{h,\tau}^{M,N}\|\le c e^{-N\tau}\|f_h\|+c' b^{-2}e^{-(1+\alpha^{-1})M\tau}\|f_h\|.
\end{equation*}
Recalling  \Cref{theorem-quadrature} and using triangle inequality  it follows
\begin{equation}\label{est-uh-Uh}
	\|u_h-U_{h,\tau}^{M,N}\|\le \left(\frac{C(\alpha)}{b}C(\tau) e^{-\frac{2\pi\min(\kappa_1,\kappa_2)}{\tau}}+ce^{-N\tau}+c'b^{-2}e^{-(1+\alpha^{-1})M\tau}\right)\|f_h\|.
\end{equation} 
In practice the optimal way to choose $\tau,M$ and $N$ is to balance the three terms on the right hand side of \eqref{est-uh-Uh}. Thus, ignoring the negligible coefficients $C(\alpha)$, $C(\tau)$ and $c,c'$ we can choose 
\begin{equation}\label{choice-para}
	M=\max\left\{\frac{\alpha}{\alpha+1}\left(\frac{2\pi\min(\kappa_1,\kappa_2)}{\tau^2}-\frac{\ln b}{\tau}\right),0\right\},\quad N=\frac{2\pi\min(\kappa_1,\kappa_2)}{\tau^2}+\frac{\ln b}{\tau},
\end{equation}
then for $M>0$, in terms of $(1+\alpha^{-1})M+N$ we  arrive at 
\begin{equation}\label{root-exp-1}
	\|u_h-U_{h,\tau}^{M,N}\|\le \frac{C'(\alpha)C(\tau)}{b} e^{-\sqrt{\pi\min(\kappa_1,\kappa_2)\left[(1+\alpha^{-1})M+N\right]}}.
\end{equation}
That is, the error decays root-exponentially with respect to the number of solves.

\section{Numerical examples}
In this section, some numerical tests on two dimensional  domain $\Omega = (0, 1) \times (0, 1)$ 
for various $\cA$ and $f$ under homogeneous boundary conditions are presented to verify our theoretical
estimates in \Cref{theorem-space} and \Cref{theorem-quadrature}.
As application, the space-fractional Allen-Cahn equation is solved at the end of this section.
All the computation is performed by Firedrake \cite{Rathgeber2017}, 
an automated system for the solution of partial differential equations using the finite element method.

We will adopt the following three operators to do all the tests.
\begin{enumerate}
	\item[$\cA_1$]: Laplace operator
	\begin{align*}
		C(x) = 1, \quad a(x) = 0, \quad r = 0.
	\end{align*}
	\item[$\cA_2$]: General real elliptic operator
	\begin{align*}
		C(x) = 
		\begin{pmatrix}
			1 + 0.5\sin \pi x & 0.5\cos \pi x, \\
			0.5\sin \pi y & 1 + 0.5\cos \pi y
		\end{pmatrix},\  a(x) = (0.5 + y, 0.5 + x)^T,\  r = 0.
	\end{align*}
	\item[$\cA_3$]: General complex elliptic operator
	\begin{align*}
		C(x)  = 
		\begin{pmatrix}
			0.5 + 5x\i + y & x - y  \\
			-xy\i & 0.5 + x + 5y \i
		\end{pmatrix},
		\quad a(x) = 0, \quad r = 0.
	\end{align*}
\end{enumerate}
The elliptic regularity index $\gamma$ for the three operators are all supposed to be $1$. For each operator, we test the convergence under the following three source terms:
\begin{enumerate}
	\item[\bf  A .] $\widetilde{H}^{2}$ source term $f(x, y)=xy(1-x)(1-y):=f_1;$
	\item[\bf  B .] $\widetilde{H}^{1}$ source term $f(x, y)=(xy)^{0.51}\big((1-x)(1-y)\big)^{0.51}:=f_2;$
	\item[\bf  C .] $\widetilde{H}^{0.5-\epsilon}$ source term $f(x, y)=1:=f_3.$
\end{enumerate}
 Furthermore, in those places without declaration, the value of $b$ we take is $1$.  
\subsection{Numerical results for space convergence}
In order to verify the theoretical estimates in  \Cref{theorem-space}, we use the quadrature scheme \eqref{Uh} to produce $u_h$ in which we choose $N=M=200$ and $\tau=\frac{3}{20}$ such that the quadrature error is negligible. We compute $u_h$ under various uniform meshes whose mesh sizes vary from $h=2^{-3}$ to $h=2^{-7}$ in steps of $\frac{1}{2}$, and then calculate the errors $\|u_h-u_{\rm ref}\|$. The reference solution $u_{\rm ref}$ is obtained by means of the same approach as $u_h$ with finer mesh size $h=2^{-10}$.  

The numerical results of the three operators $\cA_1$, $\cA_2$ and $\cA_3$ are separately shown in \Cref{tab:space-2d-laplace}, \Cref{tab:space-2d-gro} and \Cref{tab:space-2d-gco}. Each table includes three parts that correspond to the three  different source terms $f_1,f_2$ and $f_3$ we presented at the beginning of this subsection. One can see that the numerical convergence orders fit well with the theoretical values.

For the variation of errors with respect to $b$, we choose $\cA_1$ to do the test. We fix $h=2^{-7}$ and compute the numerical solution $u_h(x,b)$ under different $b$. The plots in \Cref{fig:rangeb} give the results of $\|u_h(x,b)-u_{\rm ref}(x,b)\|$ as functions of $b$, in which different column corresponds to different source term. The reference solutions $u_{\rm ref}(x,b)$ under various $b$ are obtained on a finer mesh with mesh size $h=2^{-10}$. Recalling \Cref{remark-space} we know for $f=f_1$, $\|u-u_h\|\le c b^{-1}h^2\|\Delta f\|$ which is exactly verified by \Cref{fig:rangeb}(a) and \Cref{fig:rangeb}(d). The results for $f_2$ and $f_3$ are shown in \Cref{fig:rangeb}(b), \Cref{fig:rangeb}(e) and \Cref{fig:rangeb}(c) \Cref{fig:rangeb}(f), respectively. The solid lines are numerical results and the dashed lines represent predicted decaying rates. One can observe that the numerical results coincide with our theoretical predictions. Corresponding to \Cref{fig:rangeb}, we also give the results for $b < 1$ in \Cref{fig:rangeb_all}. One can see that the proposed scheme is robust with respect to  $b \in (0, \infty)$.

\begin{table}
\centering\scriptsize
\caption{Spatial errors with respect to $h$ for $\cA_1$}\label{tab:space-2d-laplace}
\begin{tabular}{rrrrrrrrrrrr}
\toprule
            &&\multicolumn{2}{c}{$\alpha=0.10$}
            &\multicolumn{2}{c}{$\alpha=0.30$}
            &\multicolumn{2}{c}{$\alpha=0.50$}
            &\multicolumn{2}{c}{$\alpha=0.70$}
            &\multicolumn{2}{c}{$\alpha=0.90$}
\\\cmidrule(lr){3-4}\cmidrule(lr){5-6}\cmidrule(lr){7-8}\cmidrule(lr){9-10}\cmidrule(lr){11-12}
      &  $N$ & $E_{L^2}$ &  conv. & $E_{L^2}$ &  conv. & $E_{L^2}$ &  conv. & $E_{L^2}$ &  conv. & $E_{L^2}$ &  conv.\\
\midrule\multirow{6}{*}{$f_1$}
    &      8 &  2.36e-04 &     -  &  1.78e-04 &     -  &  1.42e-04 &     -  &  1.10e-04 &     -  &  7.87e-05 &     - \\
    &     16 &  5.39e-05 &   2.13 &  4.25e-05 &   2.07 &  3.52e-05 &   2.02 &  2.75e-05 &   1.99 &  1.99e-05 &   1.98\\
    &     32 &  1.28e-05 &   2.07 &  1.04e-05 &   2.03 &  8.77e-06 &   2.01 &  6.88e-06 &   2.00 &  5.00e-06 &   2.00\\
    &     64 &  3.13e-06 &   2.04 &  2.59e-06 &   2.01 &  2.18e-06 &   2.00 &  1.72e-06 &   2.00 &  1.25e-06 &   2.00\\
    &    128 &  7.72e-07 &   2.02 &  6.42e-07 &   2.01 &  5.42e-07 &   2.01 &  4.25e-07 &   2.01 &  3.09e-07 &   2.01\\
\cmidrule{2-12}
&\multicolumn{2}{c}{theor. conv.}&   2.0 &           &   2.0 &           &   2.0 &           &   2.0 &           &   2.0\\
\midrule\multirow{6}{*}{$f_2$}
    &      8 &  5.49e-03 &     -  &  2.17e-03 &     -  &  9.56e-04 &     -  &  5.68e-04 &     -  &  3.80e-04 &     - \\
    &     16 &  2.40e-03 &   1.20 &  7.10e-04 &   1.61 &  2.49e-04 &   1.94 &  1.42e-04 &   2.00 &  9.59e-05 &   1.99\\
    &     32 &  1.06e-03 &   1.17 &  2.34e-04 &   1.60 &  6.44e-05 &   1.95 &  3.55e-05 &   2.00 &  2.40e-05 &   2.00\\
    &     64 &  4.73e-04 &   1.17 &  7.74e-05 &   1.60 &  1.66e-05 &   1.96 &  8.83e-06 &   2.01 &  5.99e-06 &   2.00\\
    &    128 &  2.10e-04 &   1.17 &  2.55e-05 &   1.60 &  4.24e-06 &   1.97 &  2.18e-06 &   2.02 &  1.48e-06 &   2.01\\
\cmidrule{2-12}
& \multicolumn{2}{c}{theor. conv.} &   1.2 &           &   1.6 &           &   2.0 &           &   2.0 &           &   2.0\\
\midrule\multirow{6}{*}{$f_3$}
    &      8 &  1.20e-01 &     -  &  3.96e-02 &     -  &  1.26e-02 &     -  &  4.73e-03 &     -  &  2.36e-03 &     - \\
    &     16 &  7.73e-02 &   0.63 &  1.95e-02 &   1.03 &  4.60e-03 &   1.45 &  1.35e-03 &   1.81 &  6.09e-04 &   1.95\\
    &     32 &  4.94e-02 &   0.64 &  9.38e-03 &   1.05 &  1.65e-03 &   1.48 &  3.77e-04 &   1.84 &  1.54e-04 &   1.98\\
    &     64 &  3.13e-02 &   0.66 &  4.46e-03 &   1.07 &  5.87e-04 &   1.49 &  1.04e-04 &   1.86 &  3.87e-05 &   1.99\\
    &    128 &  1.95e-02 &   0.68 &  2.09e-03 &   1.09 &  2.07e-04 &   1.51 &  2.82e-05 &   1.88 &  9.62e-06 &   2.01\\
\cmidrule{2-12}
&\multicolumn{2}{c}{theor. conv.}&   0.7 &           &   1.1 &           &   1.5 &        &   1.9 &          &   2.0\\
\bottomrule
\end{tabular}
\end{table}
\begin{table}
\centering\scriptsize
\caption{Spatial errors with respect to $h$ for $\cA_2$}\label{tab:space-2d-gro}
\begin{tabular}{rrrrrrrrrrrr}
\toprule
           &&\multicolumn{2}{c}{$\alpha=0.10$}
            &\multicolumn{2}{c}{$\alpha=0.30$}
            &\multicolumn{2}{c}{$\alpha=0.50$}
            &\multicolumn{2}{c}{$\alpha=0.70$}
            &\multicolumn{2}{c}{$\alpha=0.90$}
\\\cmidrule(lr){3-4}\cmidrule(lr){5-6}\cmidrule(lr){7-8}\cmidrule(lr){9-10}\cmidrule(lr){11-12}
       & $N$ & $E_{L^2}$ &  conv. & $E_{L^2}$ &  conv. & $E_{L^2}$ &  conv. & $E_{L^2}$ &  conv. & $E_{L^2}$ &  conv.\\
\midrule\multirow{6}{*}{$f_1$}
    &      8 &  2.46e-04 &     -  &  2.10e-04 &     -  &  1.78e-04 &     -  &  1.37e-04 &     -  &  9.61e-05 &     - \\
    &     16 &  5.66e-05 &   2.12 &  5.13e-05 &   2.03 &  4.51e-05 &   1.98 &  3.52e-05 &   1.96 &  2.50e-05 &   1.94\\
    &     32 &  1.35e-05 &   2.07 &  1.27e-05 &   2.01 &  1.13e-05 &   1.99 &  8.90e-06 &   1.99 &  6.32e-06 &   1.98\\
    &     64 &  3.29e-06 &   2.04 &  3.16e-06 &   2.01 &  2.83e-06 &   2.00 &  2.23e-06 &   2.00 &  1.58e-06 &   2.00\\
    &    128 &  8.11e-07 &   2.02 &  7.84e-07 &   2.01 &  7.02e-07 &   2.01 &  5.51e-07 &   2.01 &  3.92e-07 &   2.01\\
\cmidrule{2-12}
& \multicolumn{2}{c}{theor. conv.} &   2.0 &           &   2.0 &           &   2.0 &           &   2.0 &           &   2.0\\
\midrule\multirow{6}{*}{$f_2$}
    &      8 &  5.58e-03 &     -  &  2.38e-03 &     -  &  1.19e-03 &     -  &  7.31e-04 &     -  &  4.77e-04 &     - \\
    &     16 &  2.43e-03 &   1.20 &  7.76e-04 &   1.62 &  3.14e-04 &   1.92 &  1.87e-04 &   1.97 &  1.23e-04 &   1.95\\
    &     32 &  1.07e-03 &   1.18 &  2.53e-04 &   1.62 &  8.14e-05 &   1.95 &  4.71e-05 &   1.99 &  3.11e-05 &   1.99\\
    &     64 &  4.78e-04 &   1.17 &  8.27e-05 &   1.61 &  2.09e-05 &   1.96 &  1.18e-05 &   2.00 &  7.77e-06 &   2.00\\
    &    128 &  2.12e-04 &   1.17 &  2.70e-05 &   1.62 &  5.32e-06 &   1.98 &  2.91e-06 &   2.01 &  1.92e-06 &   2.01\\
\cmidrule{2-12}
&  \multicolumn{2}{c}{theor. conv.}&   1.2 &           &   1.6 &           &   2.0 &           &   2.0 &        &   2.0\\
\midrule\multirow{6}{*}{$f_3$}
    &      8 &  1.21e-01 &     -  &  4.15e-02 &     -  &  1.43e-02 &     -  &  5.92e-03 &     -  &  3.03e-03 &     - \\
    &     16 &  7.78e-02 &   0.63 &  2.03e-02 &   1.03 &  5.19e-03 &   1.46 &  1.70e-03 &   1.80 &  7.95e-04 &   1.93\\
    &     32 &  4.98e-02 &   0.64 &  9.77e-03 &   1.06 &  1.84e-03 &   1.49 &  4.71e-04 &   1.85 &  2.02e-04 &   1.97\\
    &     64 &  3.15e-02 &   0.66 &  4.64e-03 &   1.07 &  6.49e-04 &   1.51 &  1.29e-04 &   1.87 &  5.09e-05 &   1.99\\
    &    128 &  1.96e-02 &   0.68 &  2.17e-03 &   1.09 &  2.27e-04 &   1.52 &  3.46e-05 &   1.89 &  1.26e-05 &   2.01\\
\cmidrule{2-12}
& \multicolumn{2}{c}{theor. conv.} &   0.7 &           &   1.1 &           &   1.5 &           &   1.9 &           &   2.0\\
\bottomrule
\end{tabular}
\end{table}

\begin{table}
\centering\scriptsize
\caption{Spatial errors with respect to $h$ for $\cA_3$}\label{tab:space-2d-gco}
\begin{tabular}{rrrrrrrrrrrr}
    \toprule
             &&\multicolumn{2}{c}{$\alpha=0.10$}
             &\multicolumn{2}{c}{$\alpha=0.30$}
             &\multicolumn{2}{c}{$\alpha=0.50$}
             &\multicolumn{2}{c}{$\alpha=0.70$}
             &\multicolumn{2}{c}{$\alpha=0.90$}
\\\cmidrule(lr){3-4}\cmidrule(lr){5-6}\cmidrule(lr){7-8}\cmidrule(lr){9-10}\cmidrule(lr){11-12}
       & $N$ & $E_{L^2}$ &  conv. & $E_{L^2}$ &  conv. & $E_{L^2}$ &  conv. & $E_{L^2}$ &  conv. & $E_{L^2}$ &  conv.\\
\midrule\multirow{6}{*}{$f_1$}
        &    8 &  2.29e-04 &     -  &  1.71e-04 &     -  &  1.28e-04 &     -  &  8.50e-05 &     -  &  5.10e-05 &     - \\
        &   16 &  5.23e-05 &   2.13 &  4.09e-05 &   2.07 &  3.13e-05 &   2.03 &  2.10e-05 &   2.02 &  1.27e-05 &   2.01\\
        &   32 &  1.25e-05 &   2.07 &  1.01e-05 &   2.02 &  7.76e-06 &   2.01 &  5.22e-06 &   2.01 &  3.16e-06 &   2.00\\
        &   64 &  3.04e-06 &   2.03 &  2.50e-06 &   2.01 &  1.93e-06 &   2.01 &  1.30e-06 &   2.00 &  7.88e-07 &   2.00\\
        &  128 &  7.52e-07 &   2.02 &  6.20e-07 &   2.01 &  4.79e-07 &   2.01 &  3.22e-07 &   2.01 &  1.95e-07 &   2.01\\
\cmidrule{2-12}
    &\multicolumn{2}{c}{theor. conv.}&   2.0 &           &   2.0 &           &   2.0 &           &   2.0 &           &   2.0\\
\midrule\multirow{6}{*}{$f_2$}
    &        8 &  5.26e-03 &     -  &  2.00e-03 &     -  &  9.24e-04 &     -  &  5.07e-04 &     -  &  2.81e-04 &     - \\
    &       16 &  2.28e-03 &   1.21 &  6.30e-04 &   1.66 &  2.34e-04 &   1.98 &  1.24e-04 &   2.03 &  6.91e-05 &   2.03\\
    &       32 &  1.01e-03 &   1.18 &  2.03e-04 &   1.64 &  5.94e-05 &   1.98 &  3.06e-05 &   2.02 &  1.72e-05 &   2.01\\
    &       64 &  4.48e-04 &   1.17 &  6.60e-05 &   1.62 &  1.51e-05 &   1.97 &  7.61e-06 &   2.01 &  4.28e-06 &   2.01\\
    &      128 &  1.98e-04 &   1.17 &  2.16e-05 &   1.61 &  3.84e-06 &   1.98 &  1.88e-06 &   2.02 &  1.06e-06 &   2.01\\
\cmidrule{2-12}
    &\multicolumn{2}{c}{theor. conv.}&   1.2 &           &   1.6 &           &   2.0 &           &   2.0 &           &   2.0\\
\midrule\multirow{6}{*}{$f_3$}
    &        8 &  1.14e-01 &     -  &  3.47e-02 &     -  &  1.13e-02 &     -  &  4.48e-03 &     -  &  2.07e-03 &     - \\
    &       16 &  7.34e-02 &   0.64 &  1.68e-02 &   1.05 &  3.97e-03 &   1.51 &  1.25e-03 &   1.84 &  5.27e-04 &   1.97\\
    &       32 &  4.69e-02 &   0.65 &  8.00e-03 &   1.07 &  1.38e-03 &   1.52 &  3.41e-04 &   1.87 &  1.33e-04 &   1.99\\
    &       64 &  2.96e-02 &   0.66 &  3.79e-03 &   1.08 &  4.84e-04 &   1.52 &  9.24e-05 &   1.89 &  3.34e-05 &   1.99\\
    &      128 &  1.84e-02 &   0.69 &  1.77e-03 &   1.10 &  1.69e-04 &   1.52 &  2.48e-05 &   1.90 &  8.29e-06 &   2.01\\
\cmidrule{2-12}
    &\multicolumn{2}{c}{theor. conv.}&   0.7 &           &   1.1 &           &   1.5 &           &   1.9 &           &   2.0\\
    \bottomrule
\end{tabular}
\end{table}

\begin{figure}
    \centering
    \subfigure[\bf $f=f_1$]{\includegraphics[width=.32\textwidth]{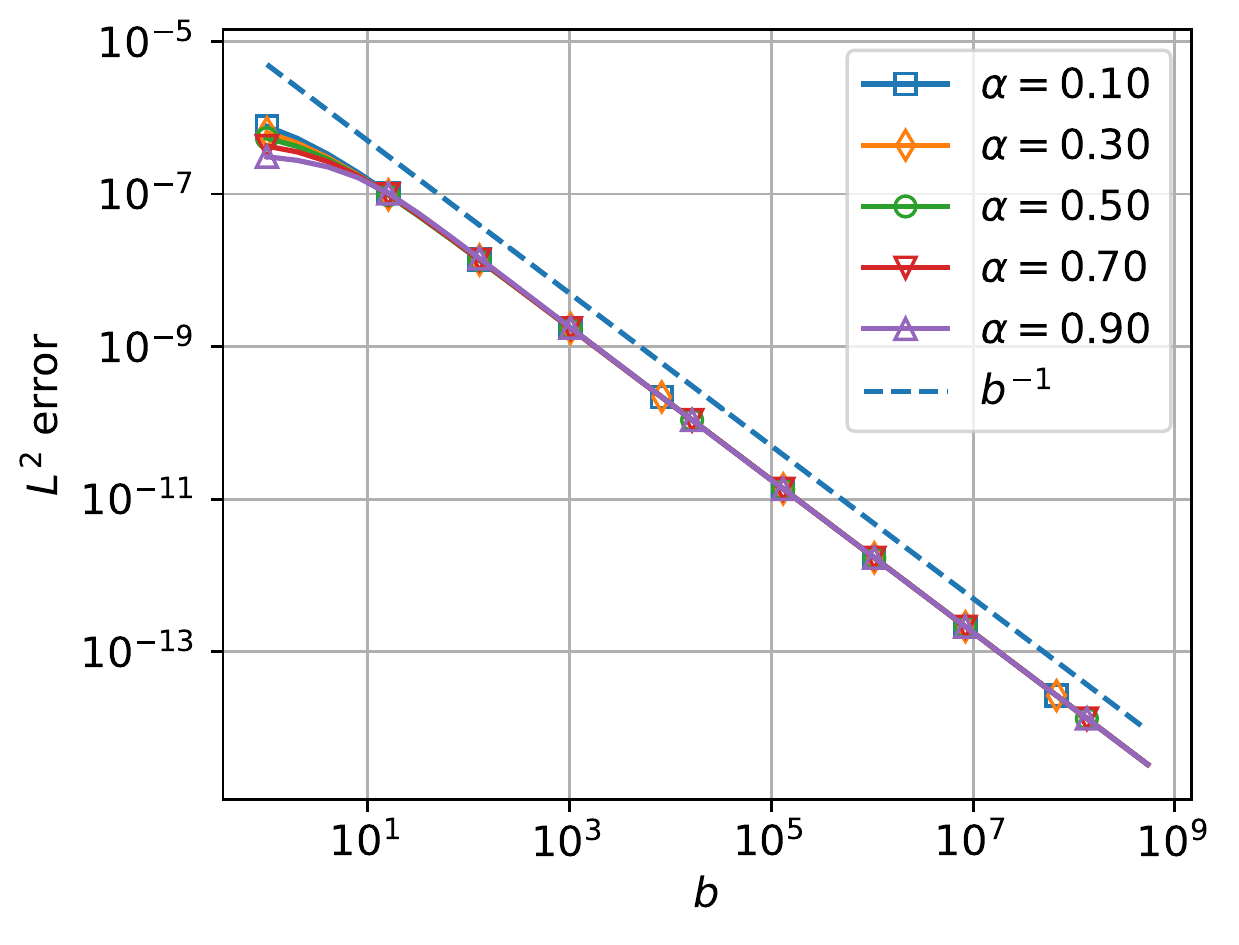}}
    \subfigure[\bf $f=f_2$]{\includegraphics[width=.32\textwidth]{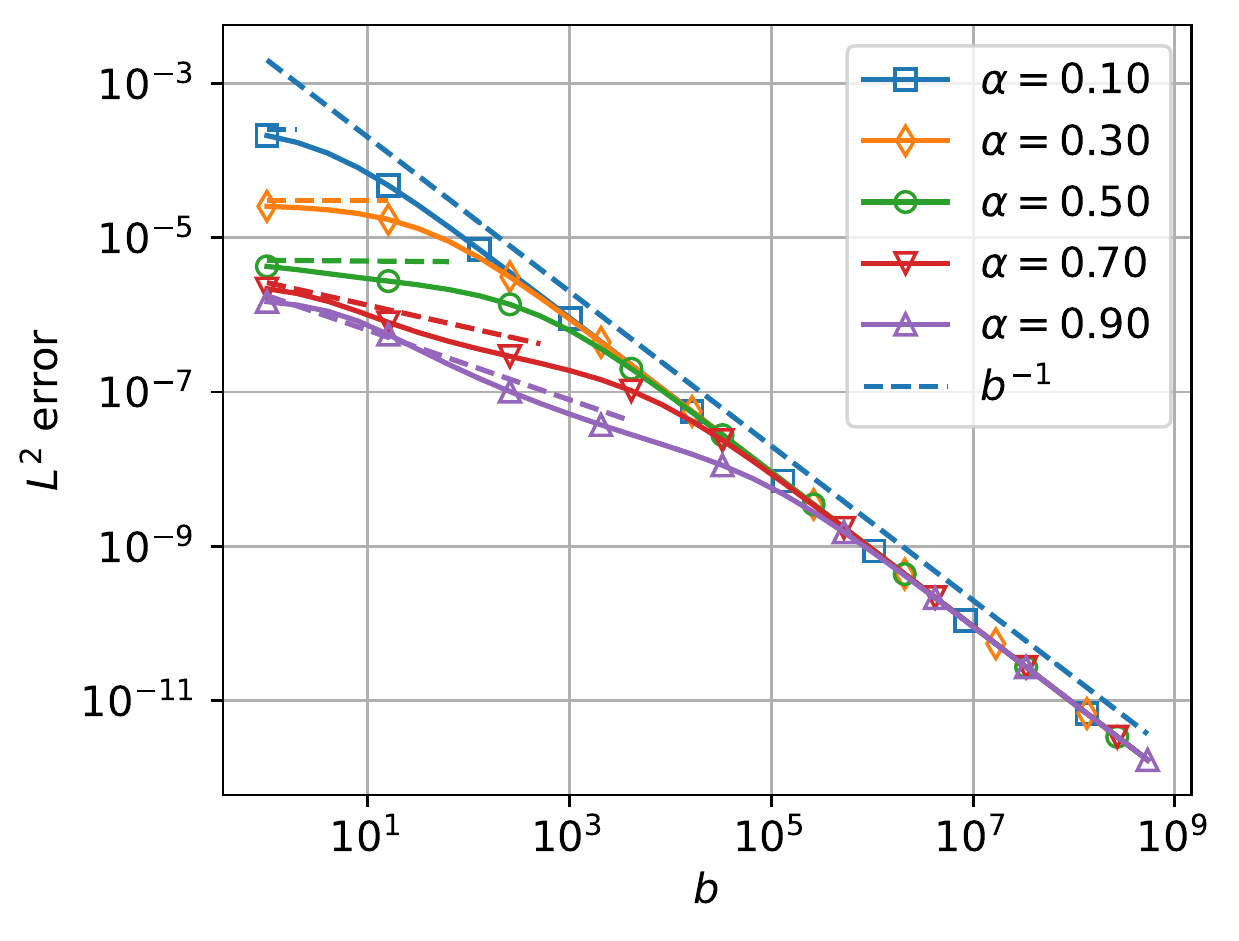}}
    \subfigure[\bf $f=f_3$]{\includegraphics[width=.32\textwidth]{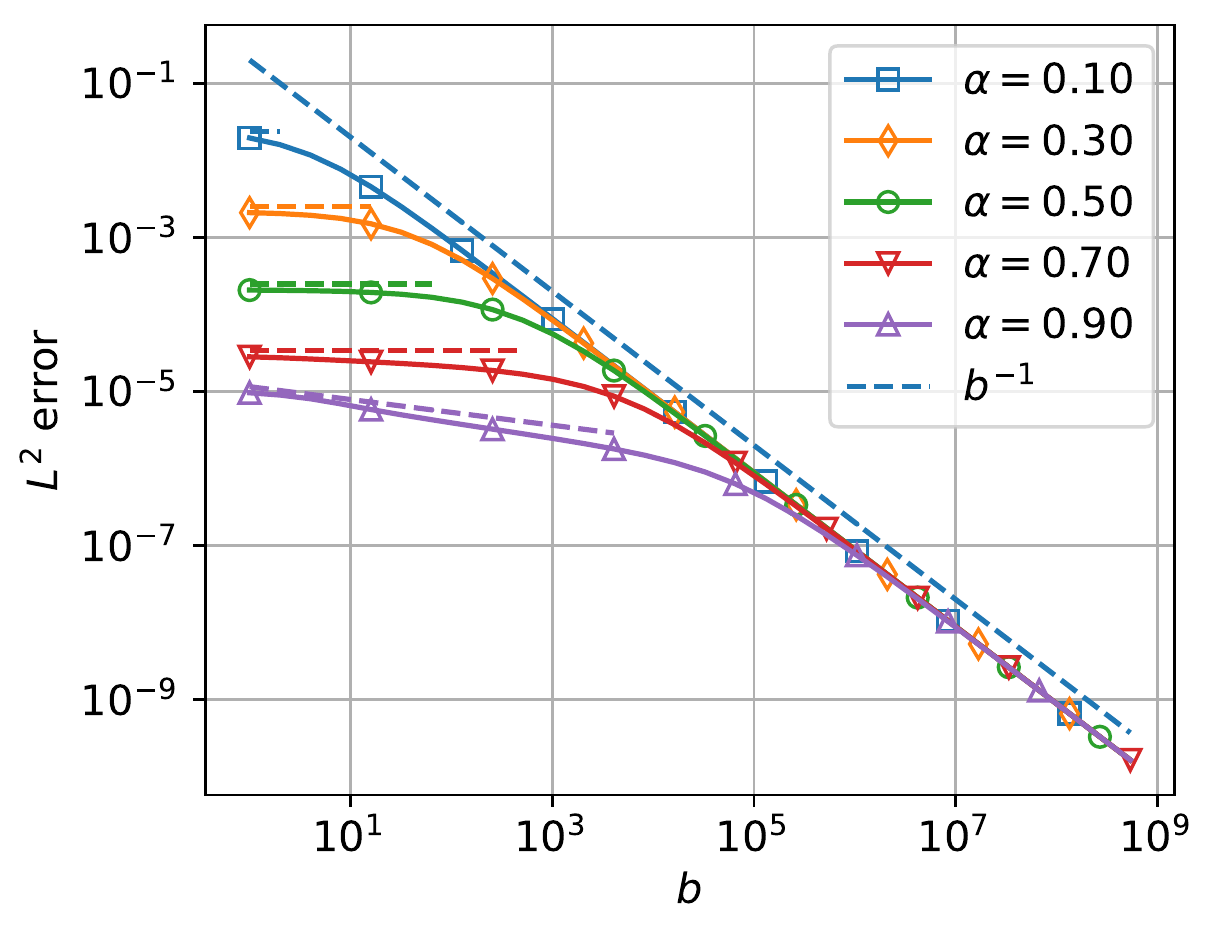}} \\
    \subfigure[\bf $f=f_1$]{\includegraphics[width=.32\textwidth]{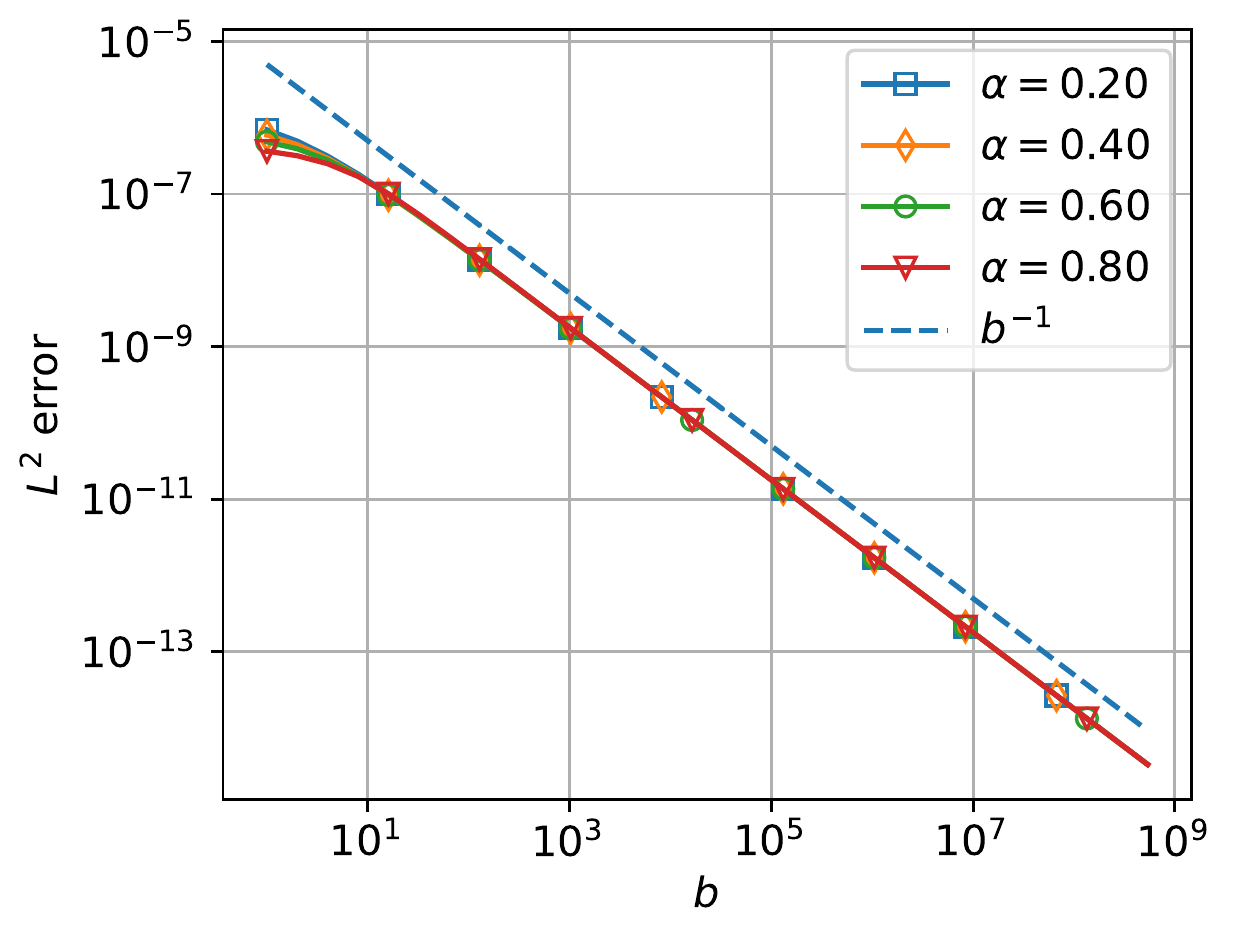}}
    \subfigure[\bf $f=f_2$]{\includegraphics[width=.32\textwidth]{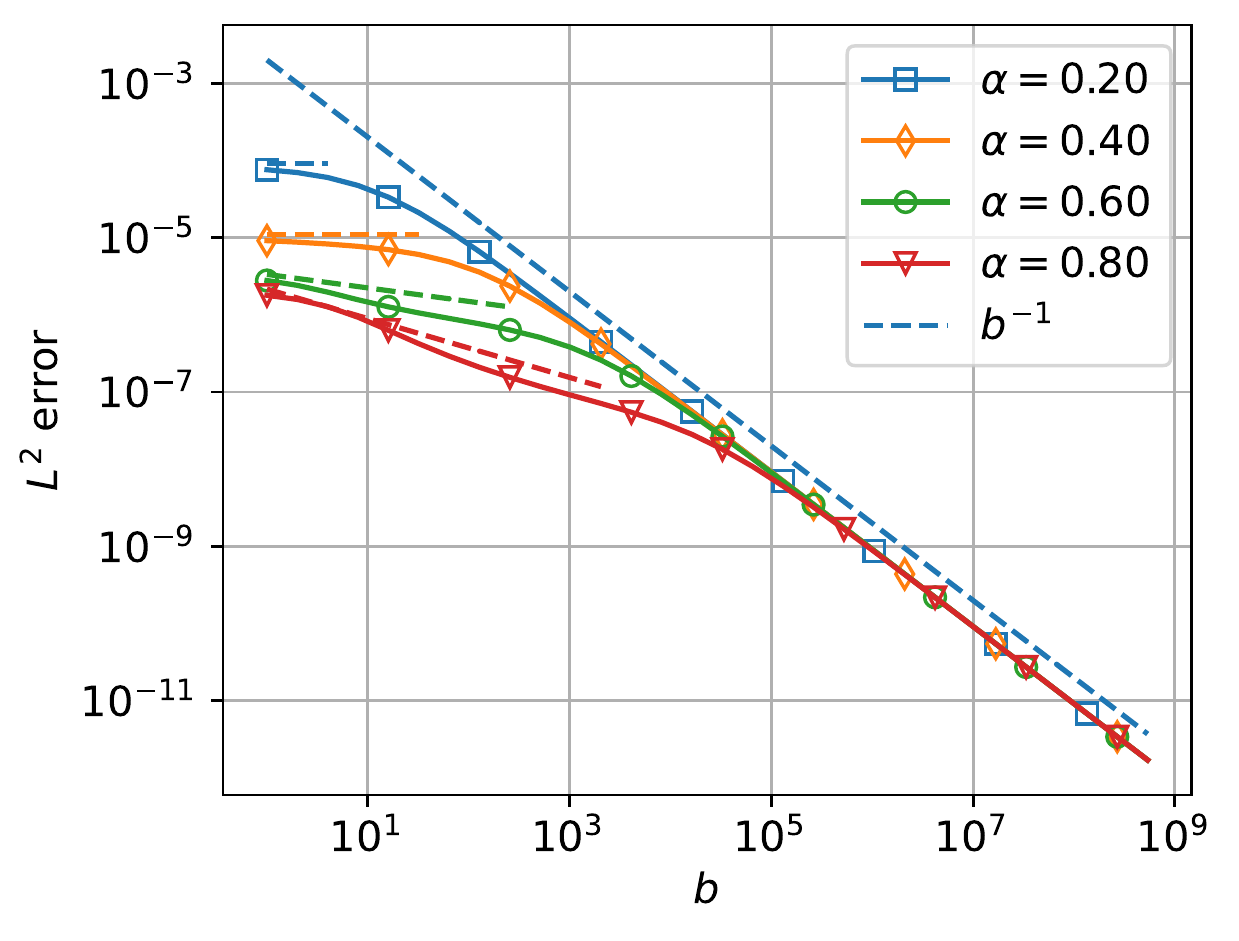}}
    \subfigure[\bf $f=f_3$]{\includegraphics[width=.32\textwidth]{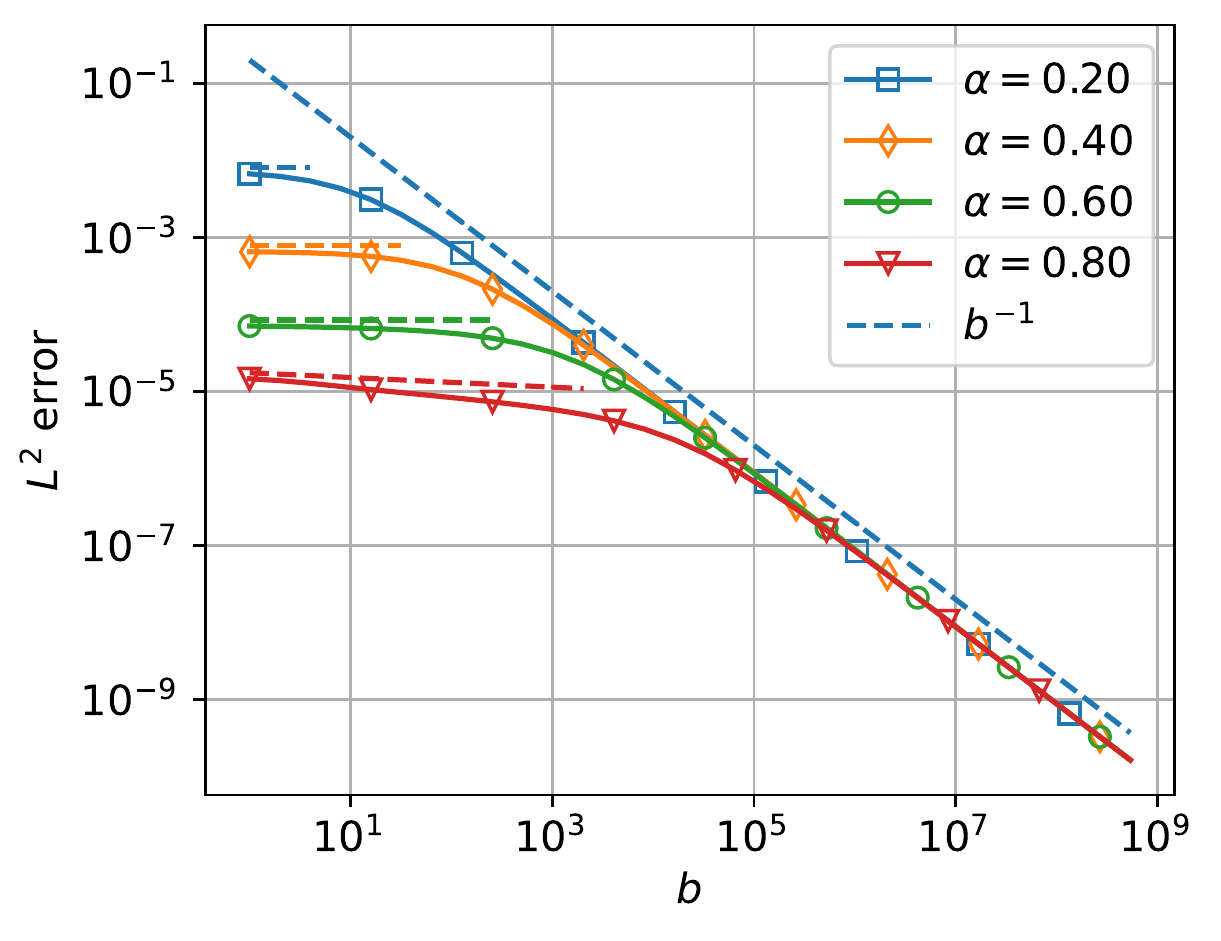}} \\
    \caption{The solid lines are spatial errors with respect to $b$ under different $f$. The dashed lines are theoretical predictions up to a constant multiplier.}\label{fig:rangeb}
\end{figure}

\begin{figure}
    \centering
    \subfigure[\bf $f=f_1$]{\includegraphics[width=.32\textwidth]{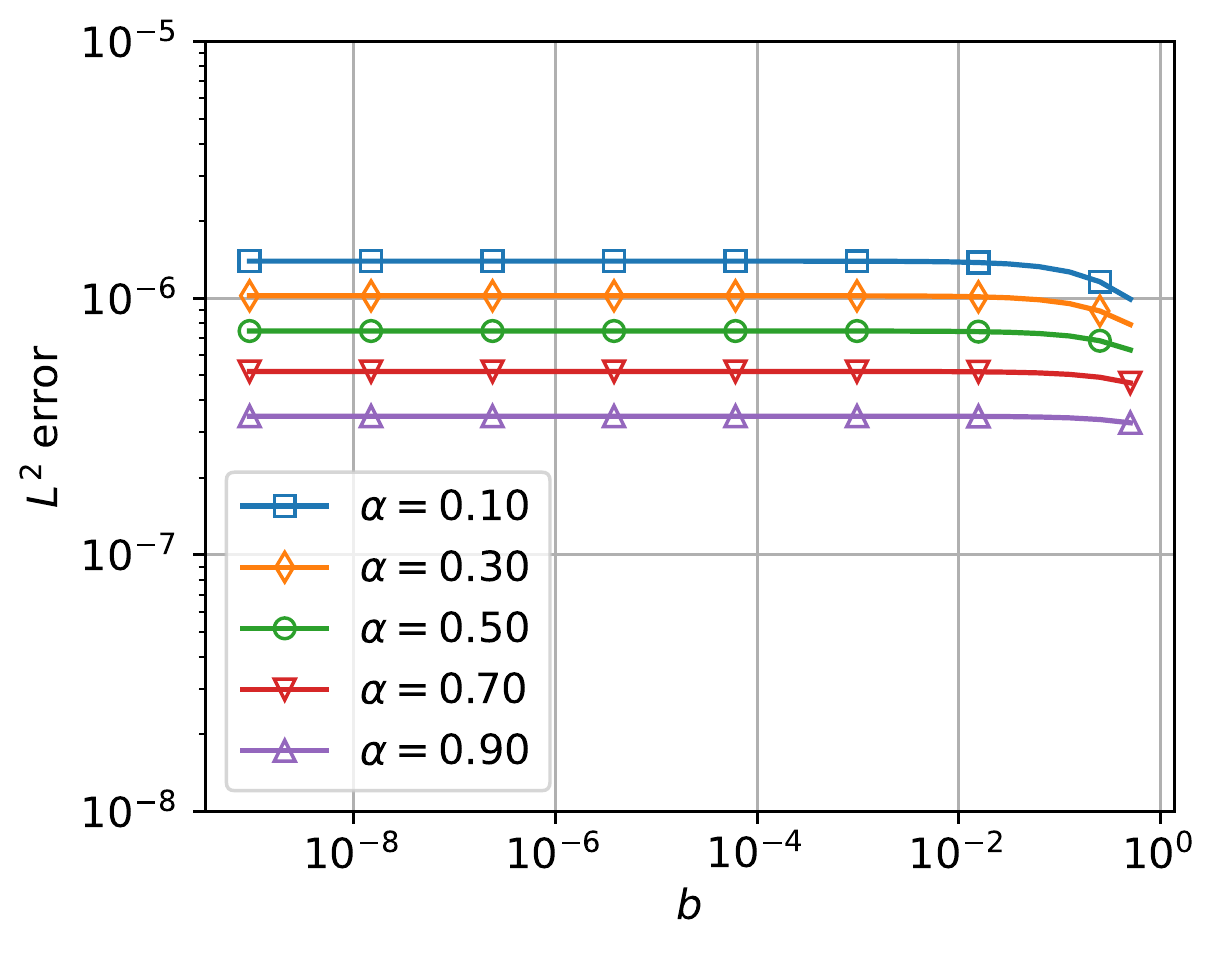}}
    \subfigure[\bf $f=f_2$]{\includegraphics[width=.32\textwidth]{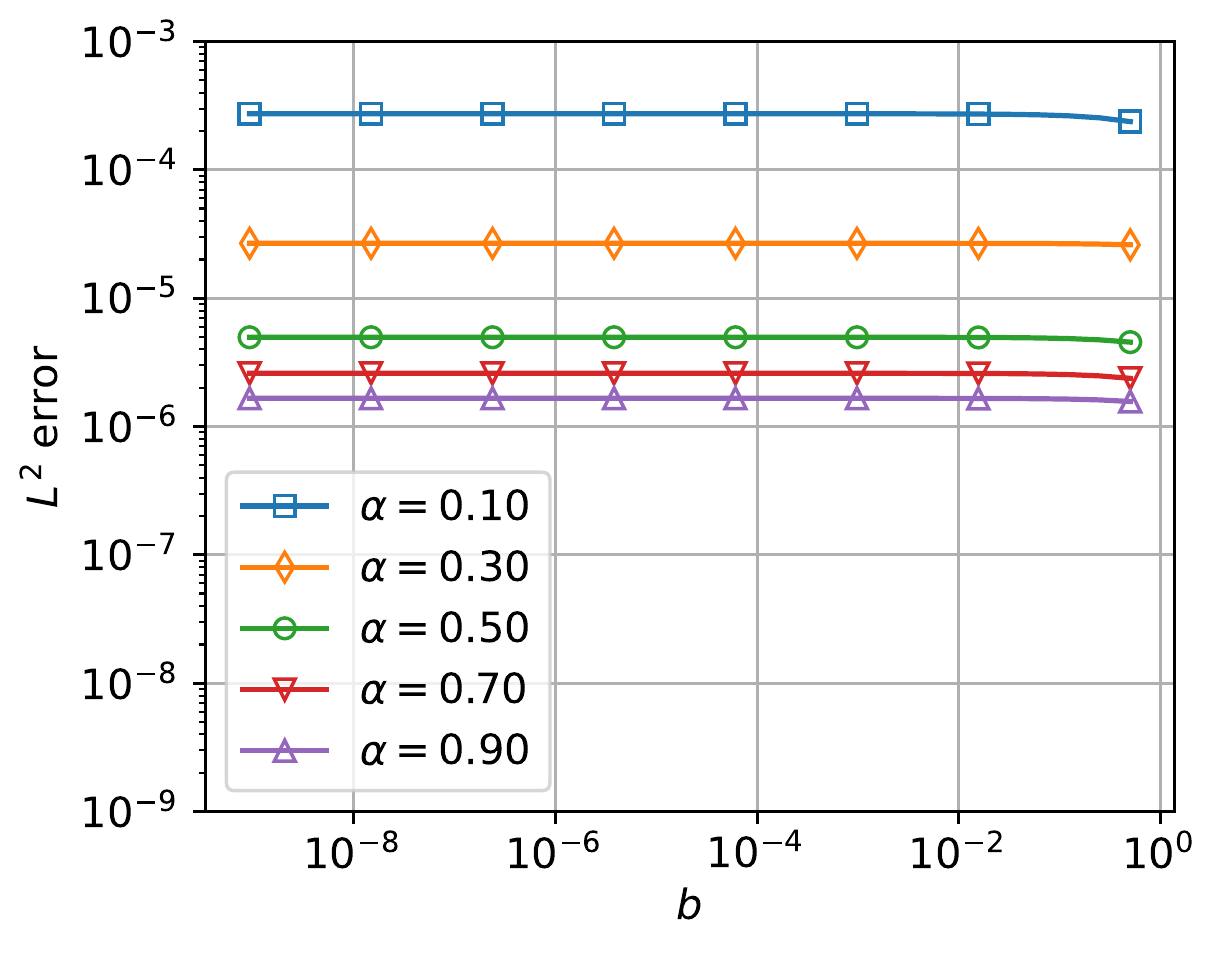}}
    \subfigure[\bf $f=f_3$]{\includegraphics[width=.32\textwidth]{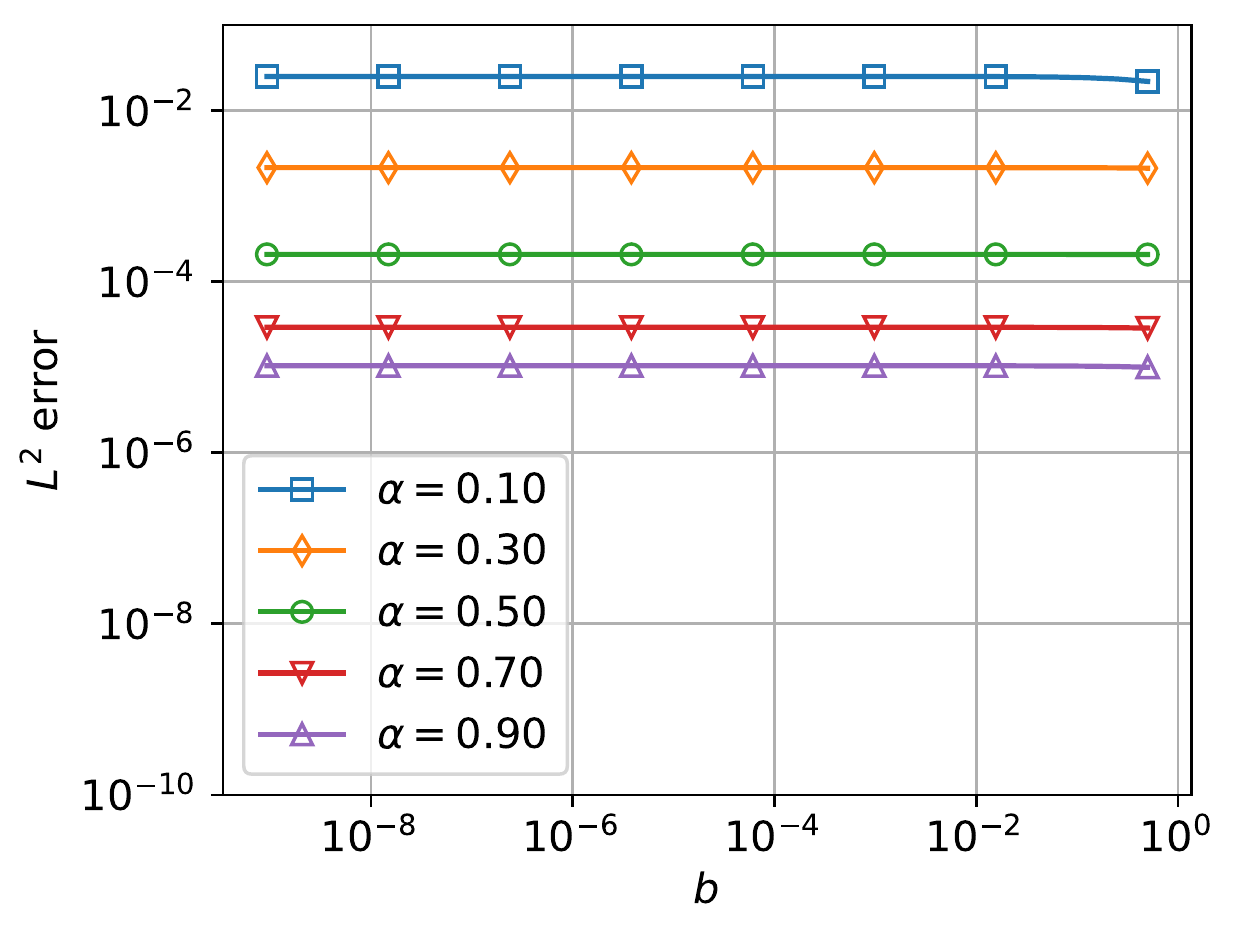}} \\
    \subfigure[\bf $f=f_1$]{\includegraphics[width=.32\textwidth]{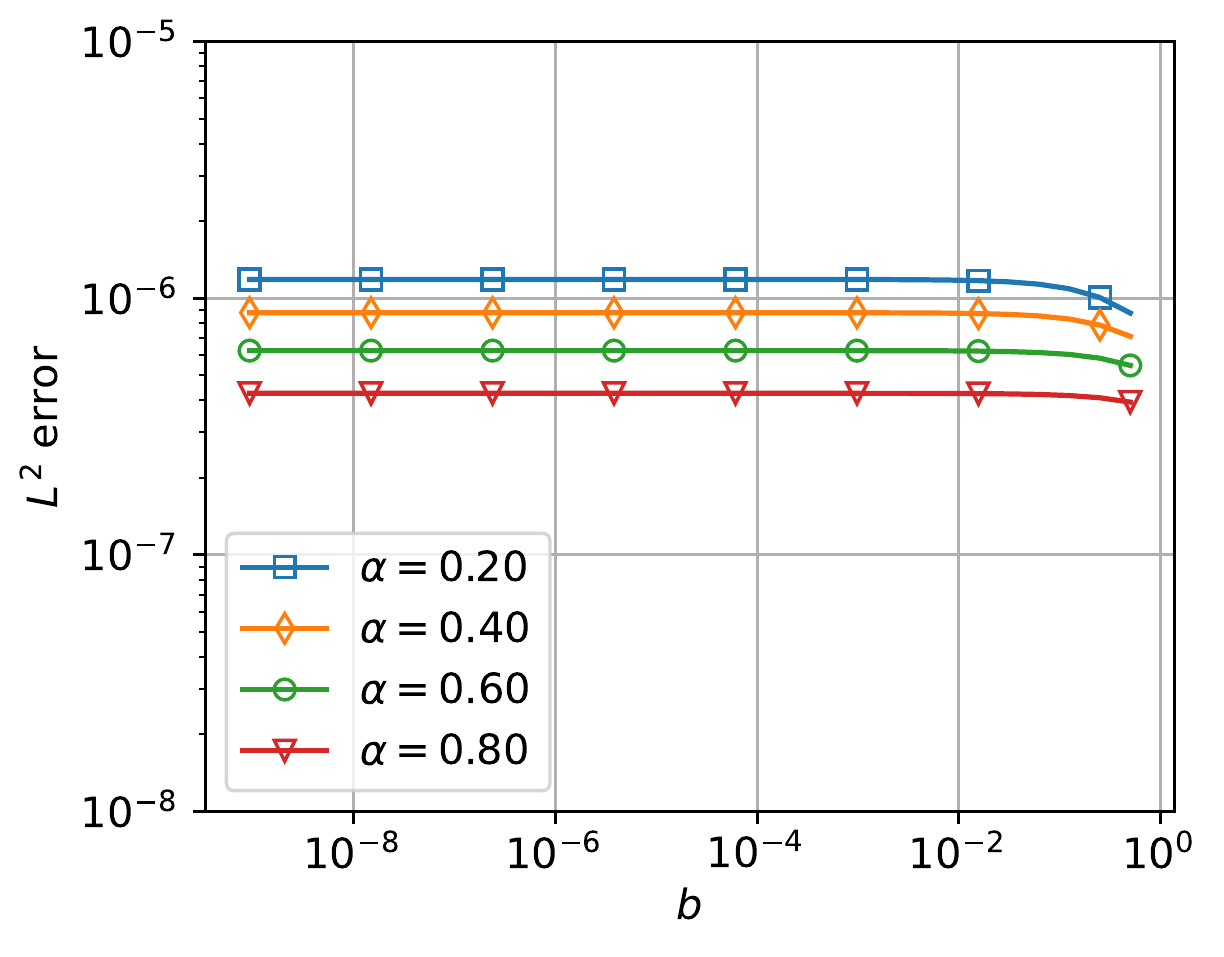}}
    \subfigure[\bf $f=f_2$]{\includegraphics[width=.32\textwidth]{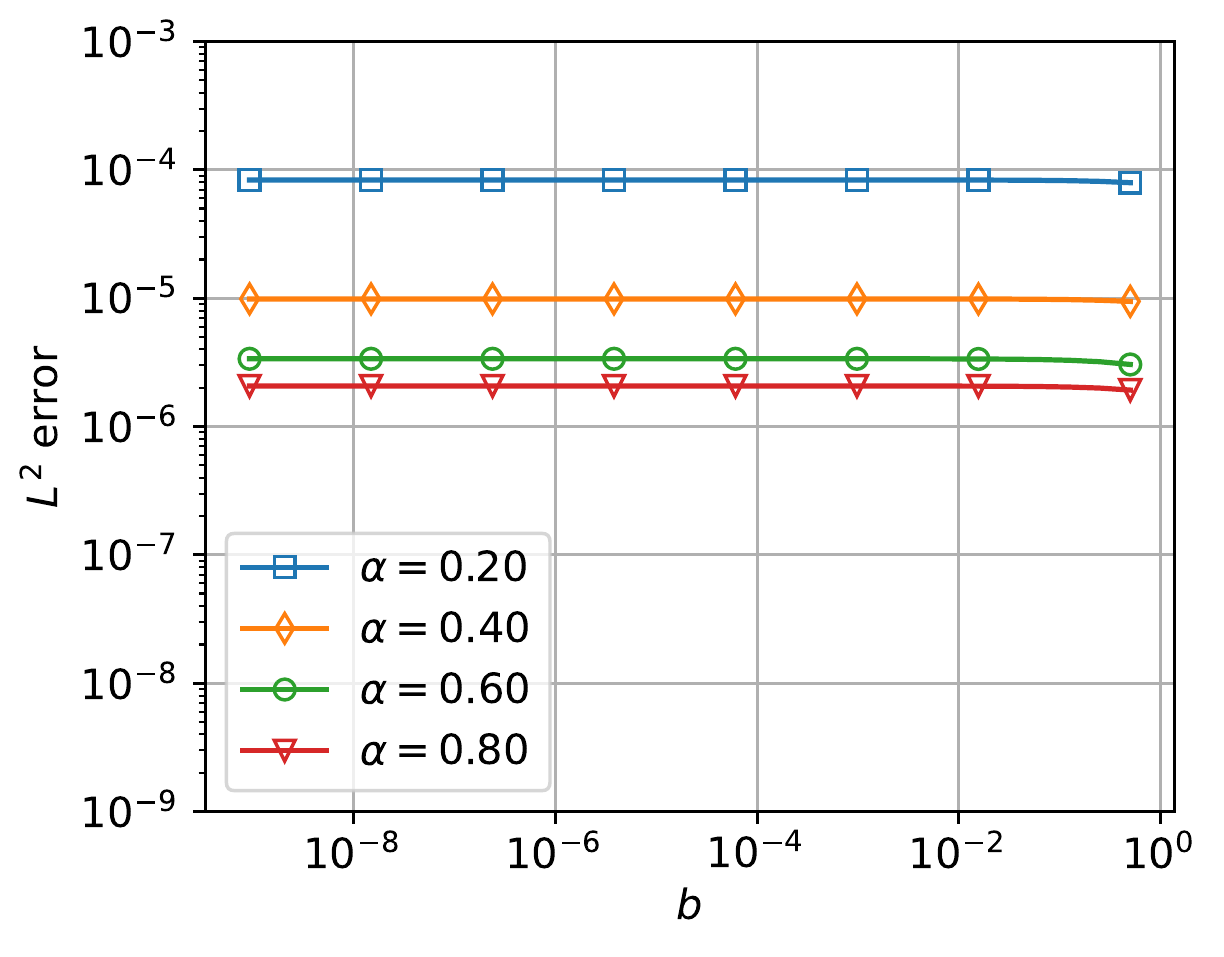}}
    \subfigure[\bf $f=f_3$]{\includegraphics[width=.32\textwidth]{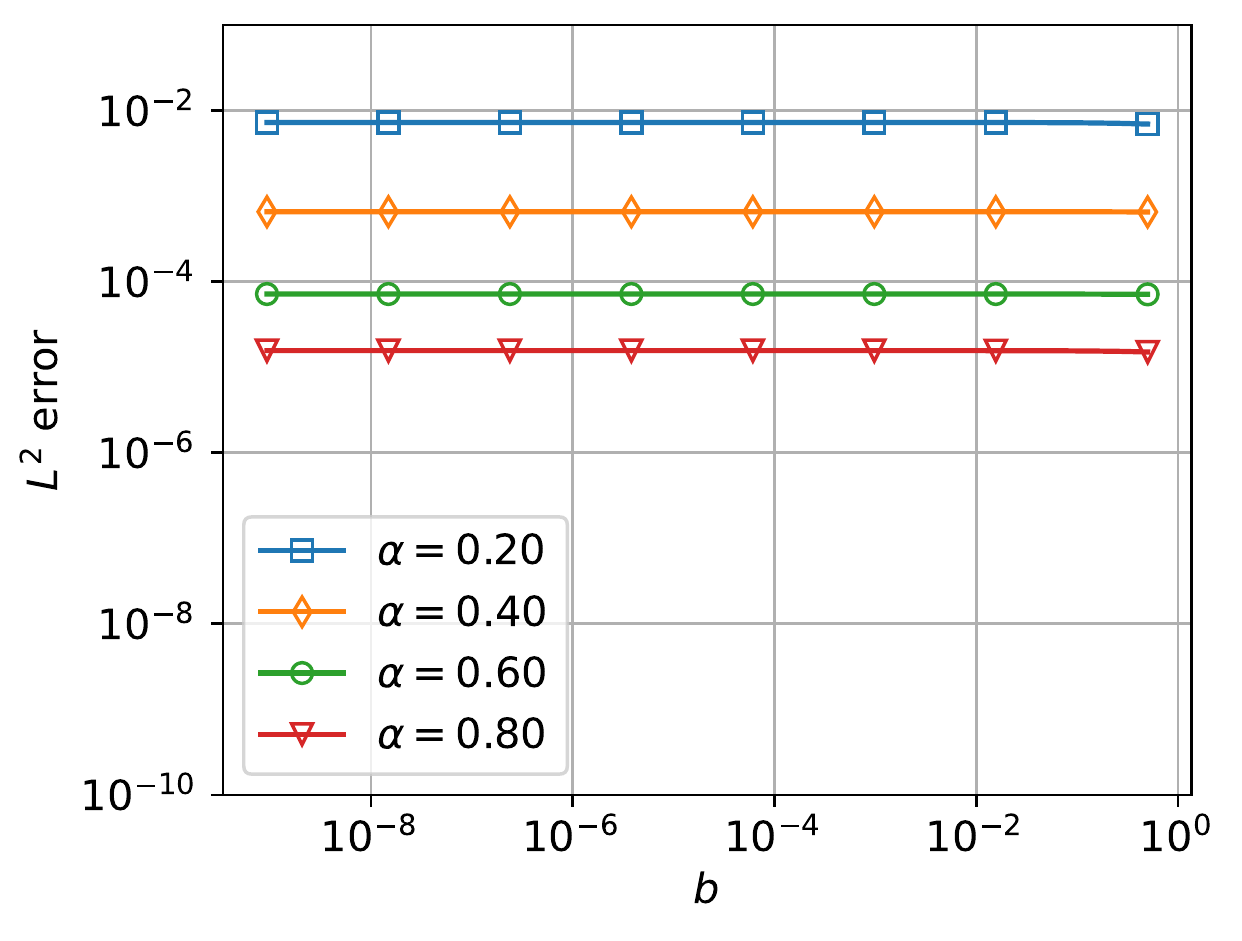}} \\
    \caption{The solid lines are spatial errors with respect to $b \in [2^{-30}, 2^{-1}]$ under different $f$.}\label{fig:rangeb_all}
\end{figure}

\subsection{Numerical results for the quadrature}
In this subsection, we shall verify \Cref{theorem-quadrature}, say, the error $\|u_h-u_{h,\tau}\|$ with respect to the quadrature step size $\tau$. To this end, we take $\cA_1$ and $\cA_3$ with different source terms $f_1,f_2$ and $f_3$ to do the test. The `semidiscrete' solution $u_h$ in all cases are approximated by the numerical solution $U_{h,\tau'}^{M',N'}$ with  $M'=N'=400$, $\tau'=\frac{3}{40}$. Numerical solutions are obtained by fixing  $M\tau=N\tau=30$ and varying $\tau$. For $\cA_1$ we clearly know $\beta=0$ and for $\cA_3$, computations show  $\beta\approx 1.778$ so $\arctan\beta\approx 1.059$. \Cref{fig:quad-laplace-split} and \Cref{fig:quad-gco-split} plot the quadrature errors for different $\alpha$ as functions of $N(M)$ where the solid lines are numerical results and the dashed lines are $\mathcal{O}(e^{-2\pi\min(\kappa_1,\kappa_2)/\tau})$. One can observe that, for a given value of $\alpha$, our theoretical estimates are sharp up to a constant $C(\tau)$ which is insignificant compared with the exponentially decaying term even in the worst case scenario $C(\tau)=\mathcal{O}(\tau^{-1}(1+|\ln \tau|))$.

By examining \Cref{fig:quad-laplace-split} and \Cref{fig:quad-gco-split}, we can see that for $\cA_1$ the quadrature error decays fastest when $\alpha=0.5$; while for $\cA_3$ the fastest decaying lines are corresponding to $\alpha=0.6$. This is consistent  with our theoretical analysis. In fact, appealing to \Cref{theorem-quadrature} we know the quadrature scheme achieves the fastest covergence rate when $\kappa_1=\kappa_2$ regardless of the smoothness of $f$, which gives $\alpha=0.5$ for $\cA_1$ and $\alpha\approx 0.60$ for $\cA_3$.

\begin{figure}
    \centering
    \subfigure[\bf $f=f_1$]{\includegraphics[width=.32\textwidth]{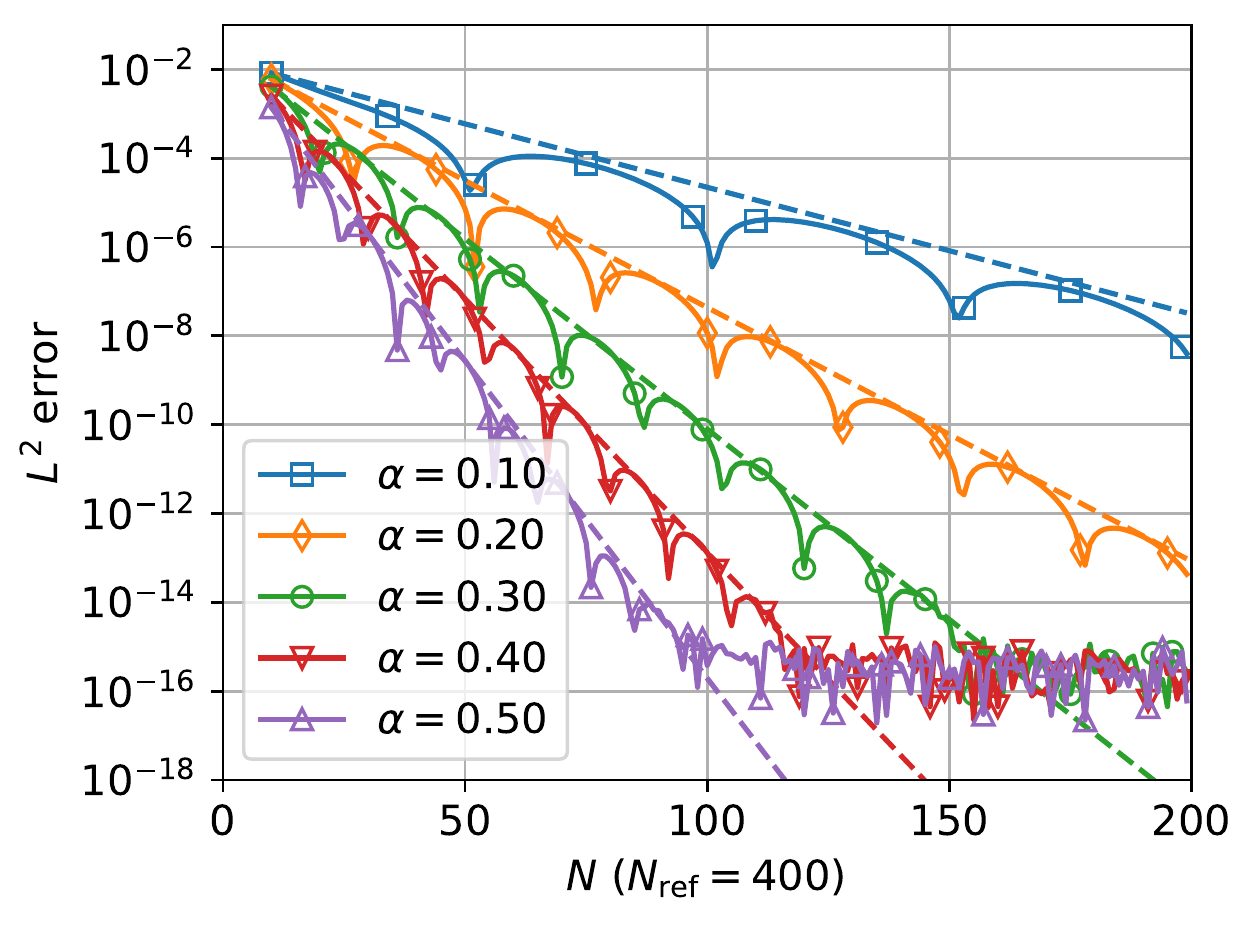}}
    \subfigure[\bf $f=f_2$]{\includegraphics[width=.32\textwidth]{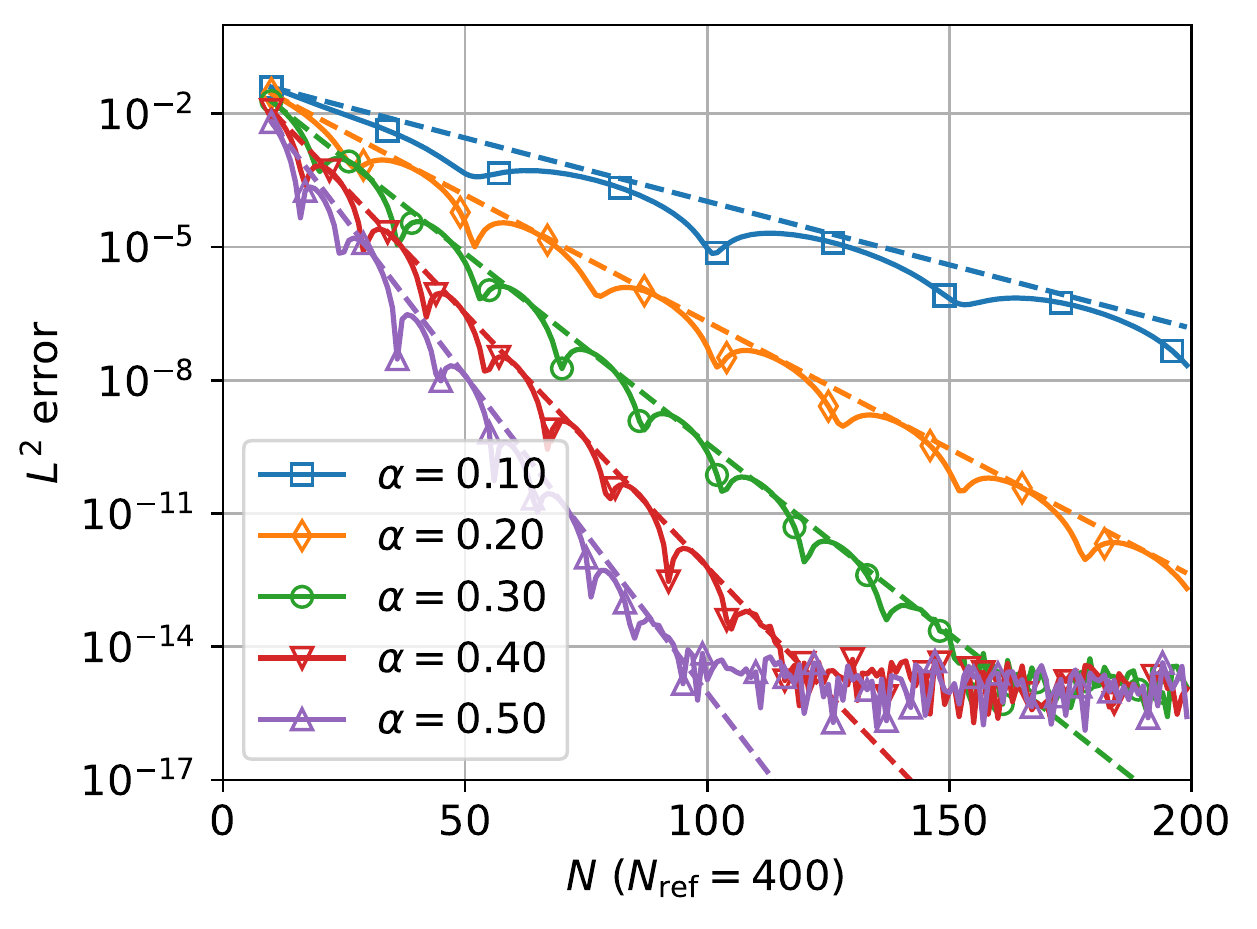}}
    \subfigure[\bf $f=f_3$]{\includegraphics[width=.32\textwidth]{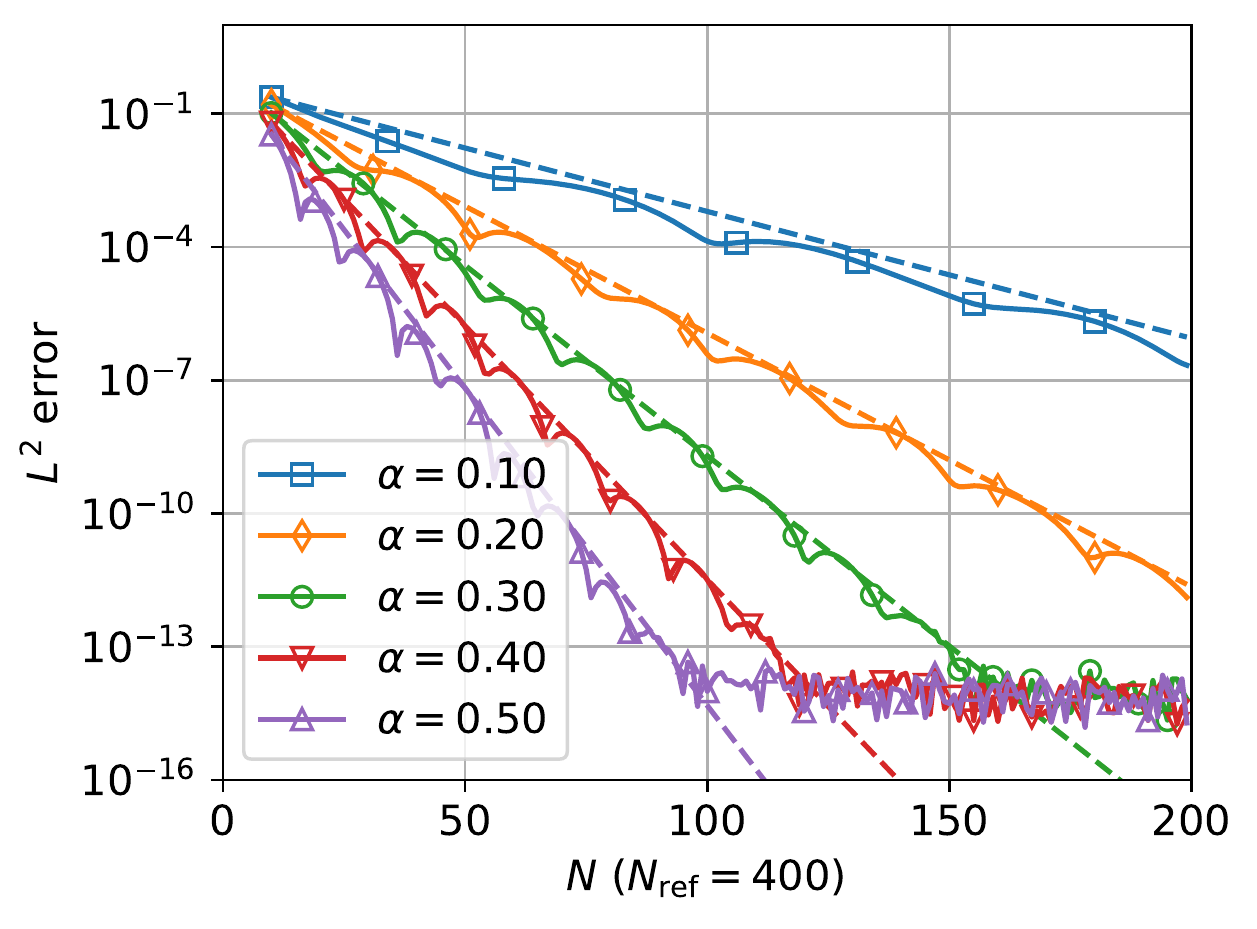}} \\
    \subfigure[\bf $f=f_1$]{\includegraphics[width=.32\textwidth]{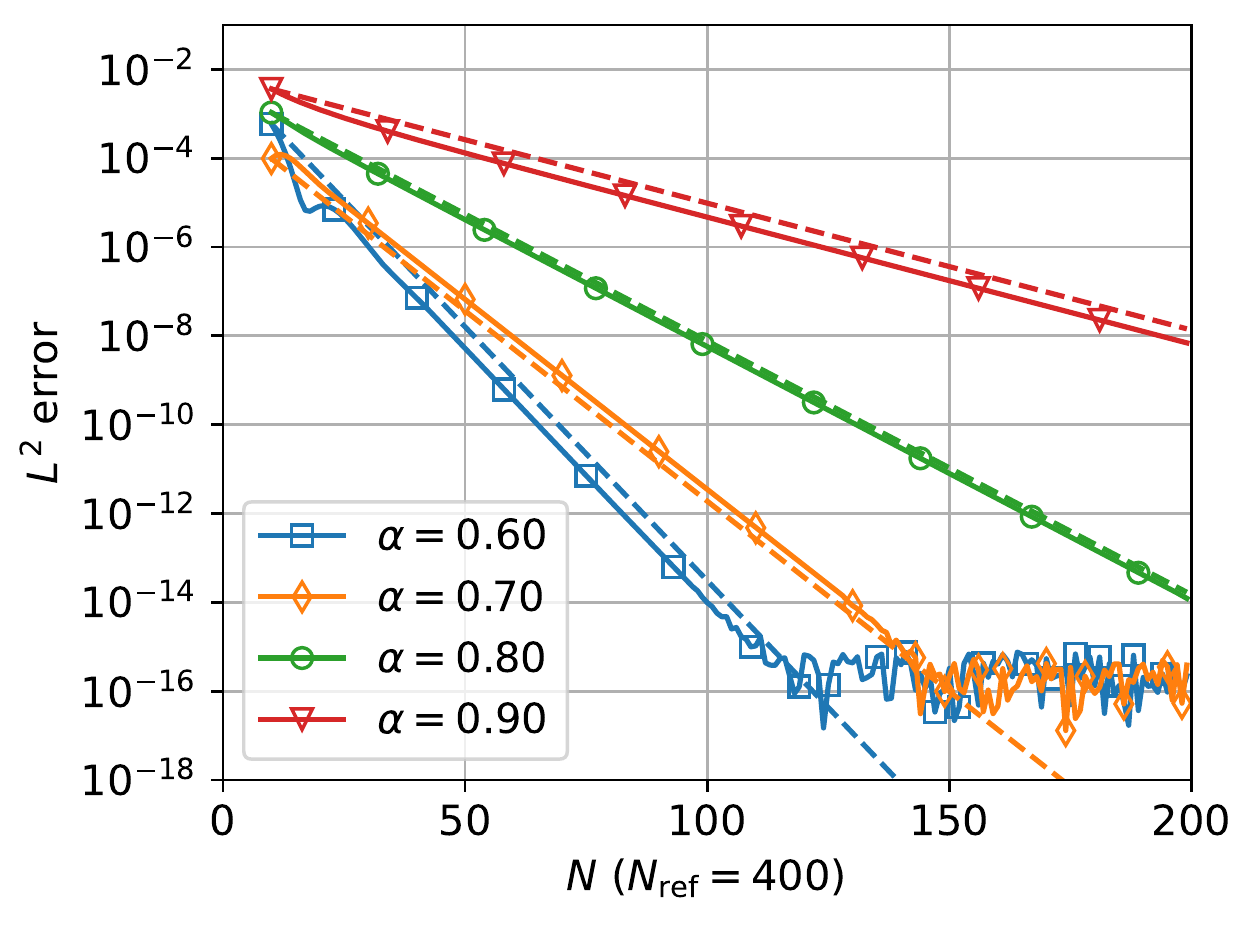}}
    \subfigure[\bf $f=f_2$]{\includegraphics[width=.32\textwidth]{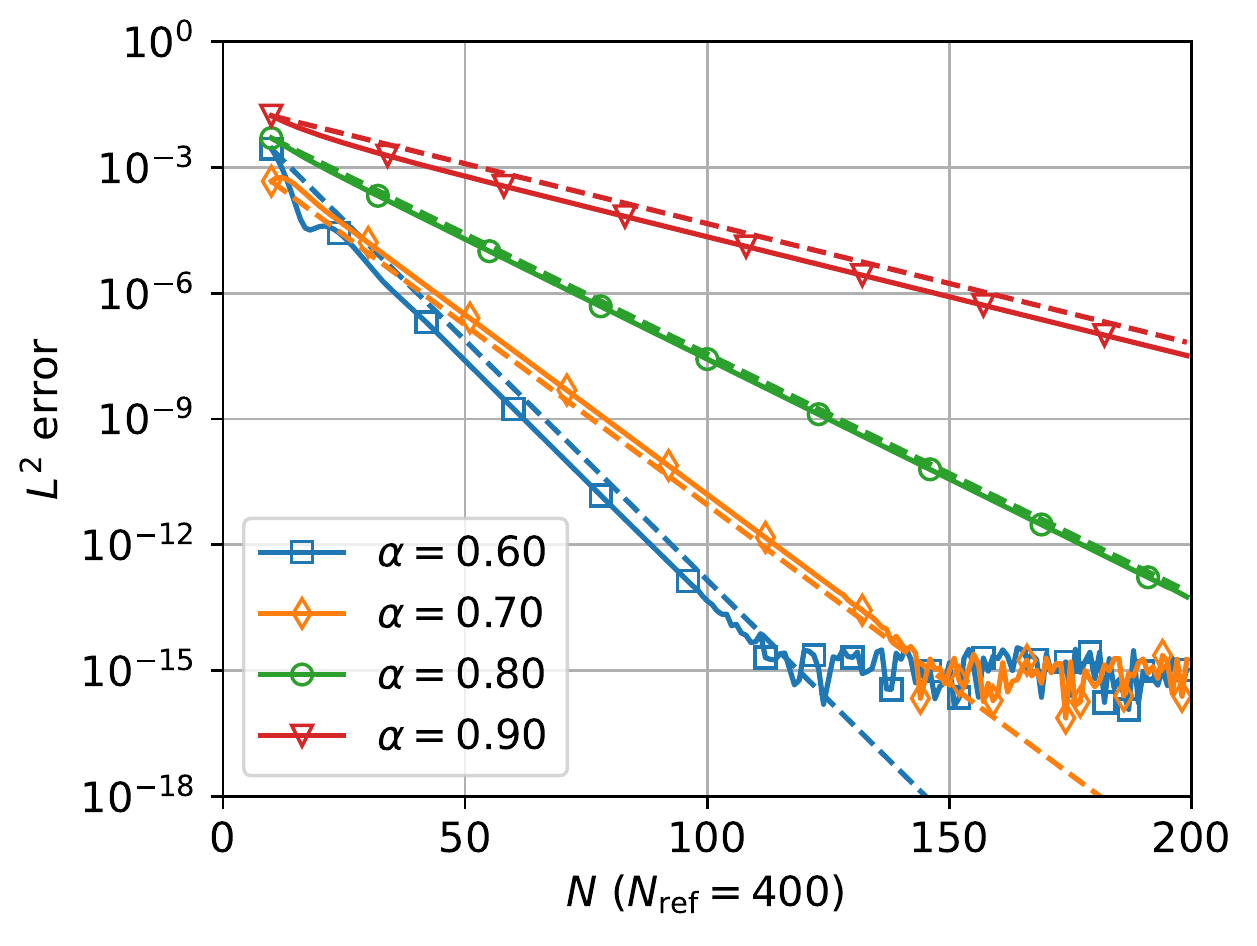}}
    \subfigure[\bf $f=f_3$]{\includegraphics[width=.32\textwidth]{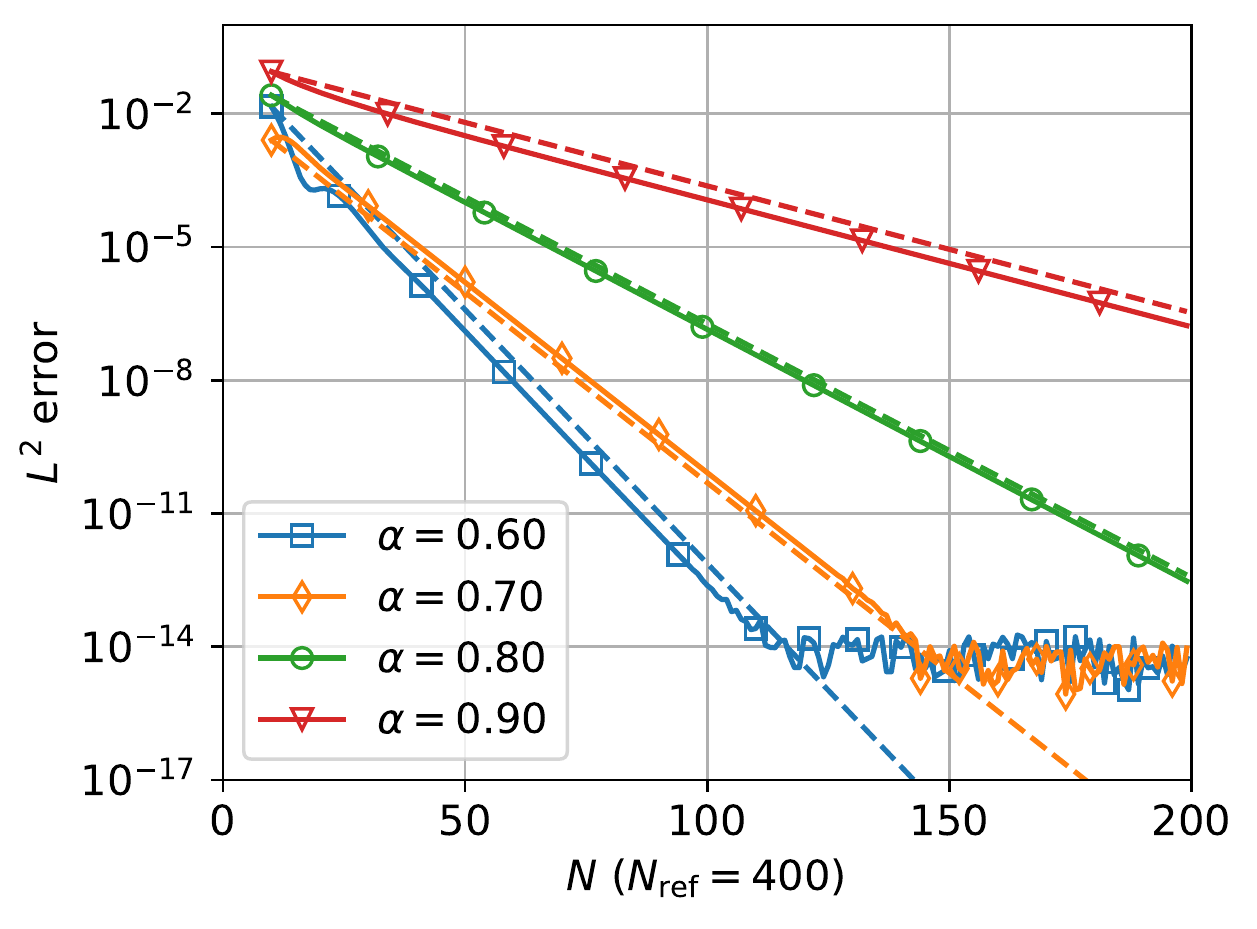}}
    \caption{Quadrature errors for $\cA_1$ under different $f$. The dashed lines are $\mathcal{O}(e^{-\frac{2\pi\min(\kappa_1,\kappa_2)}{\tau}})$ with $\kappa_1=\alpha\pi$, $\kappa_2=(1-\alpha)\pi$.}\label{fig:quad-laplace-split}
\end{figure}
\begin{figure}
    \centering
    \subfigure[\bf $f=f_1$]{\includegraphics[width=.32\textwidth]{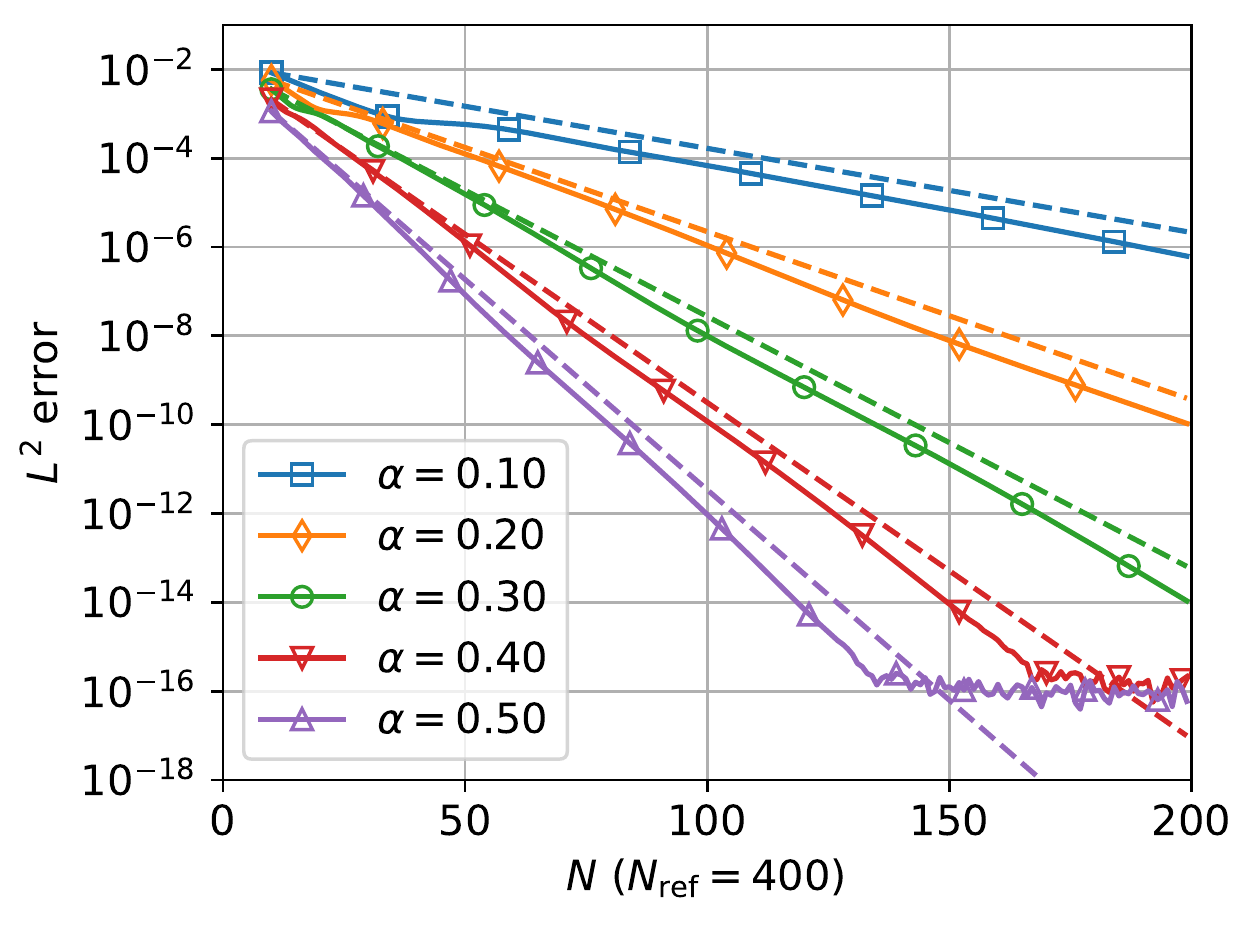}}
    \subfigure[\bf $f=f_2$]{\includegraphics[width=.32\textwidth]{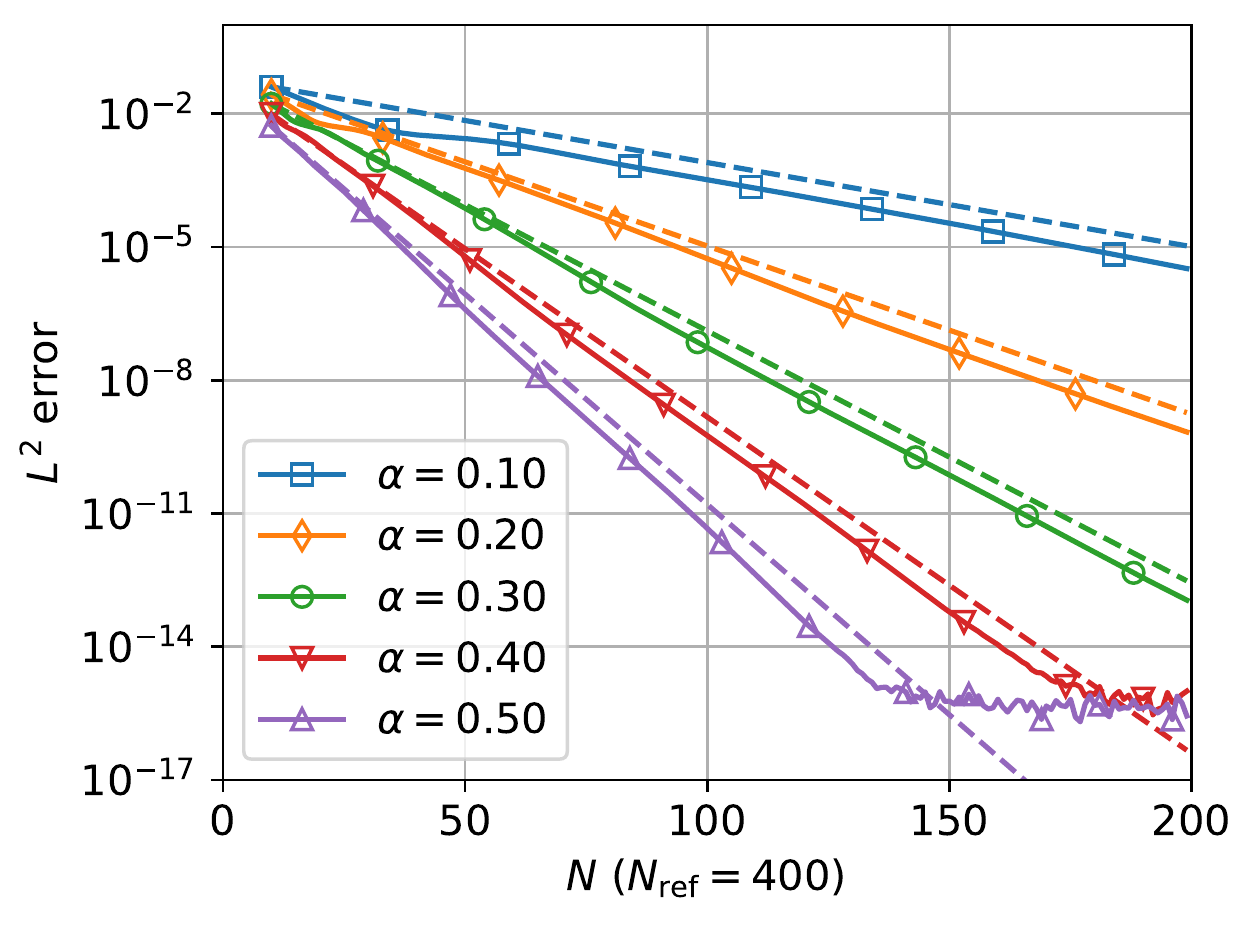}}
    \subfigure[\bf $f=f_3$]{\includegraphics[width=.32\textwidth]{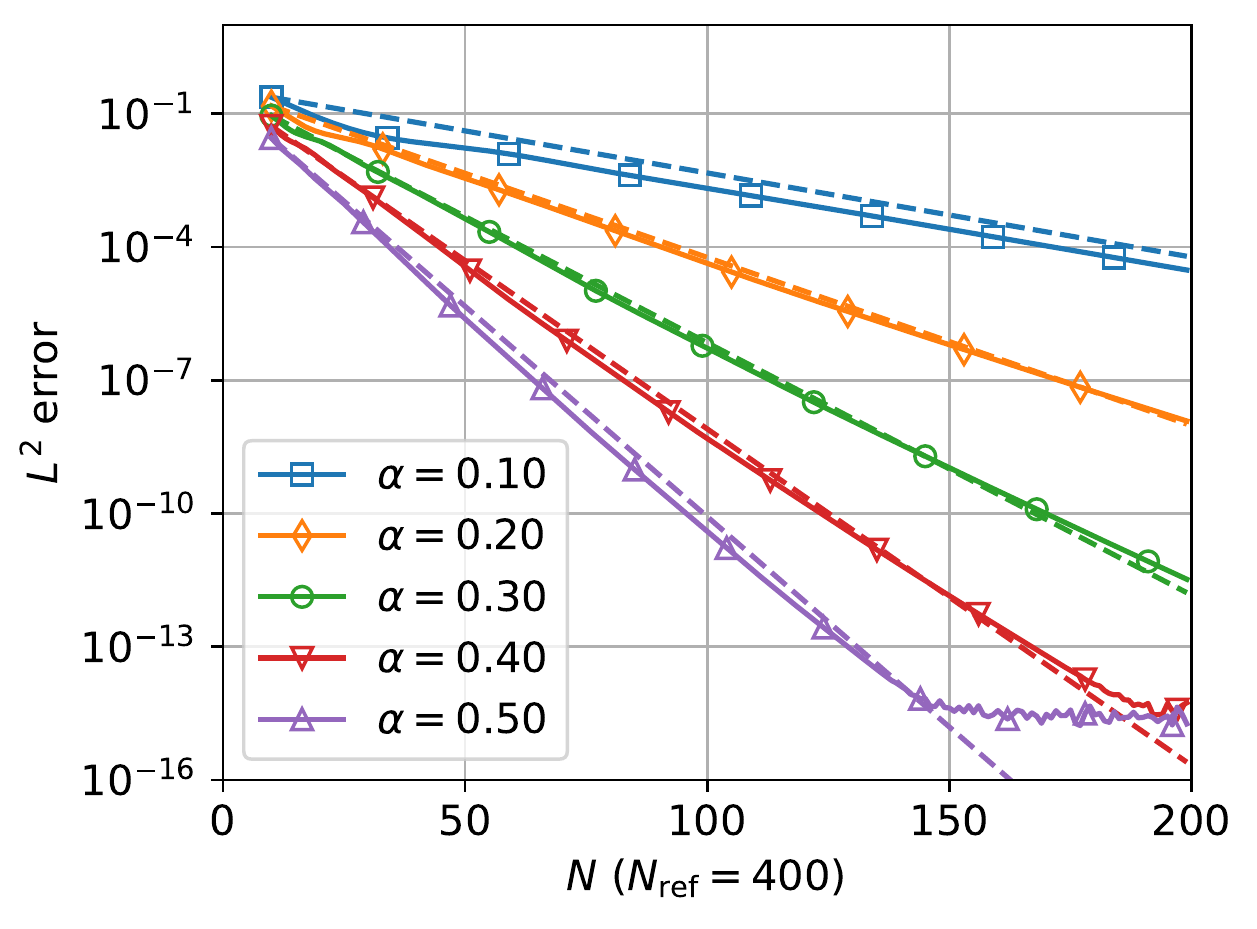}}  \\
    \subfigure[\bf $f=f_1$]{\includegraphics[width=.32\textwidth]{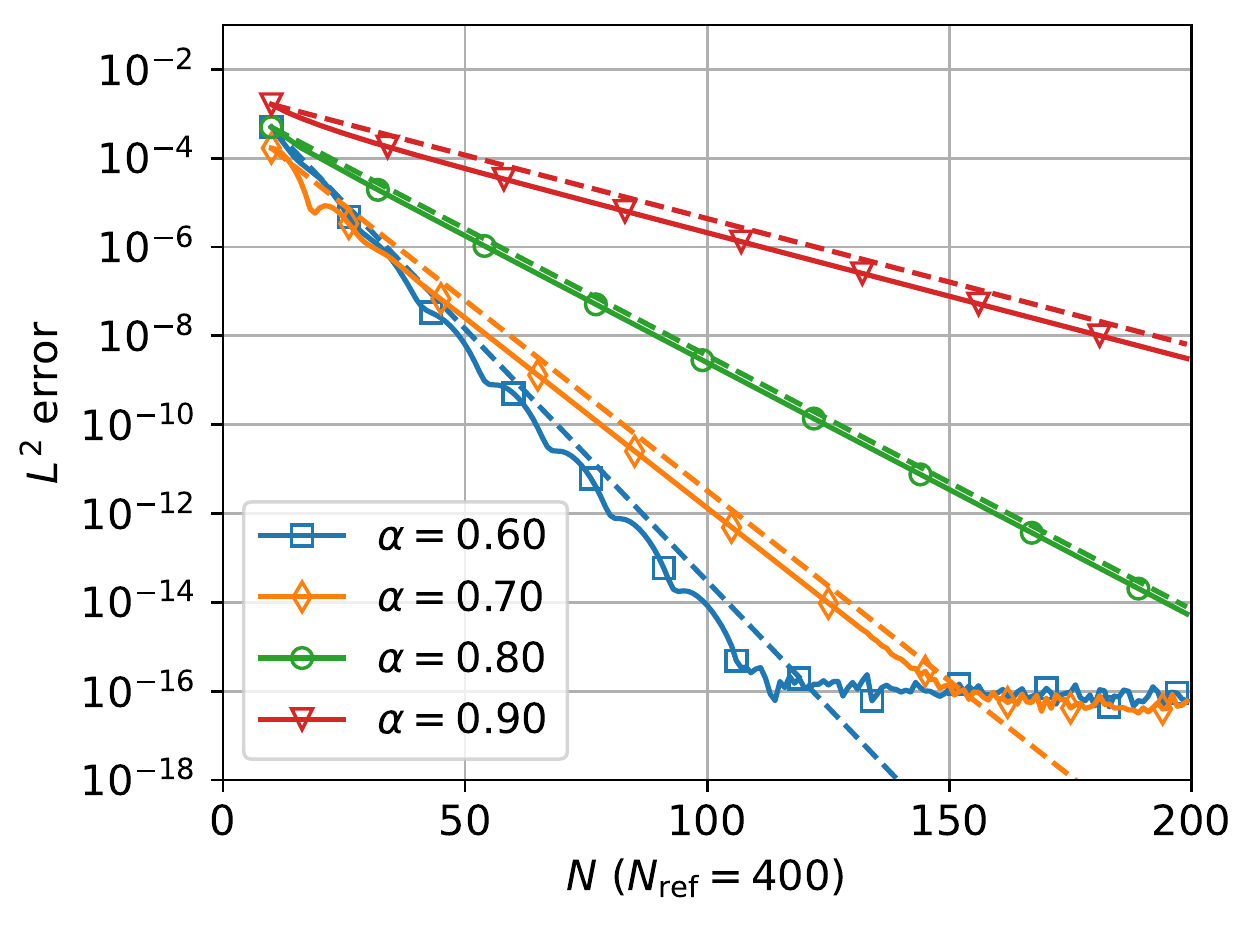}} 
    \subfigure[\bf $f=f_2$]{\includegraphics[width=.32\textwidth]{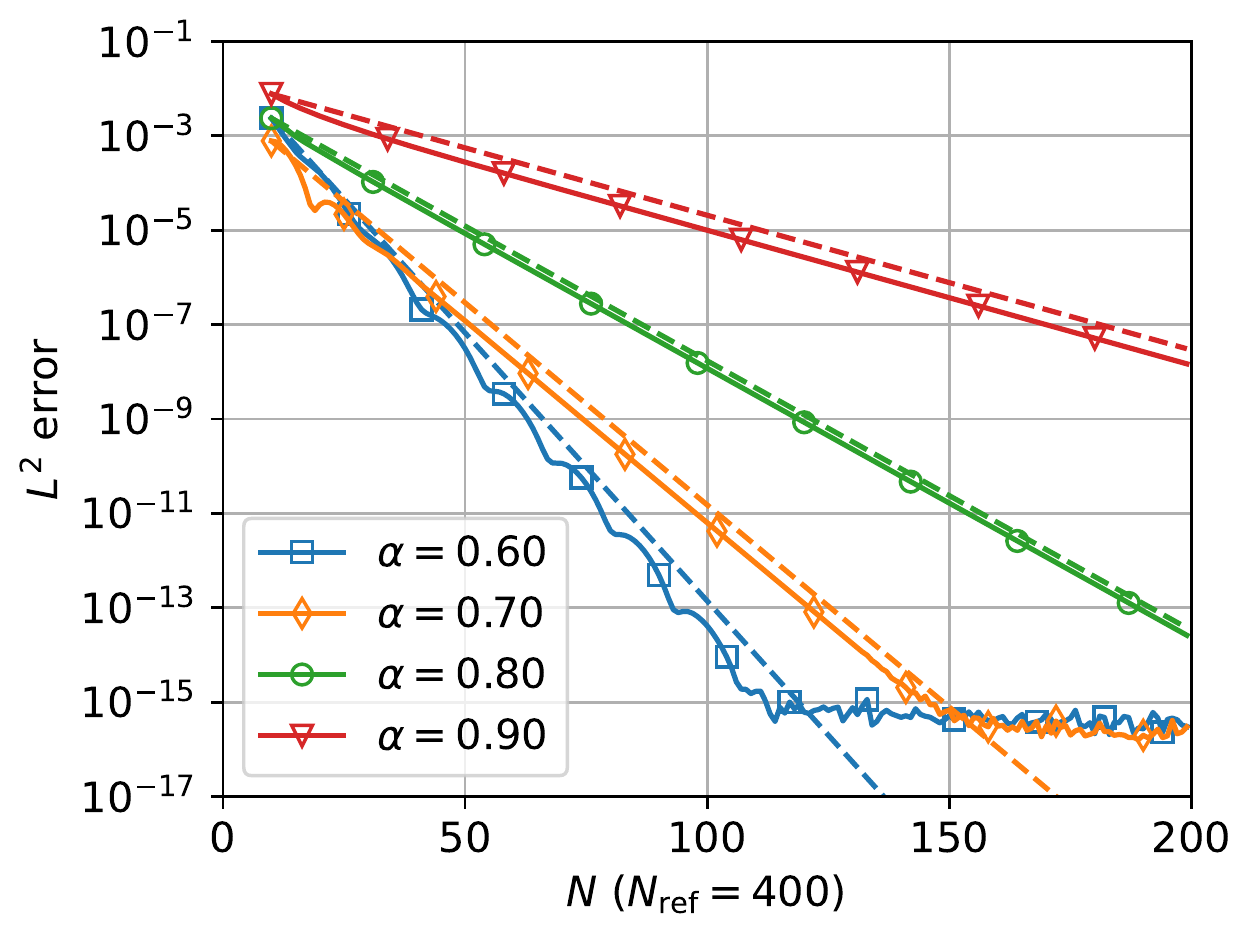}}
    \subfigure[\bf $f=f_3$]{\includegraphics[width=.32\textwidth]{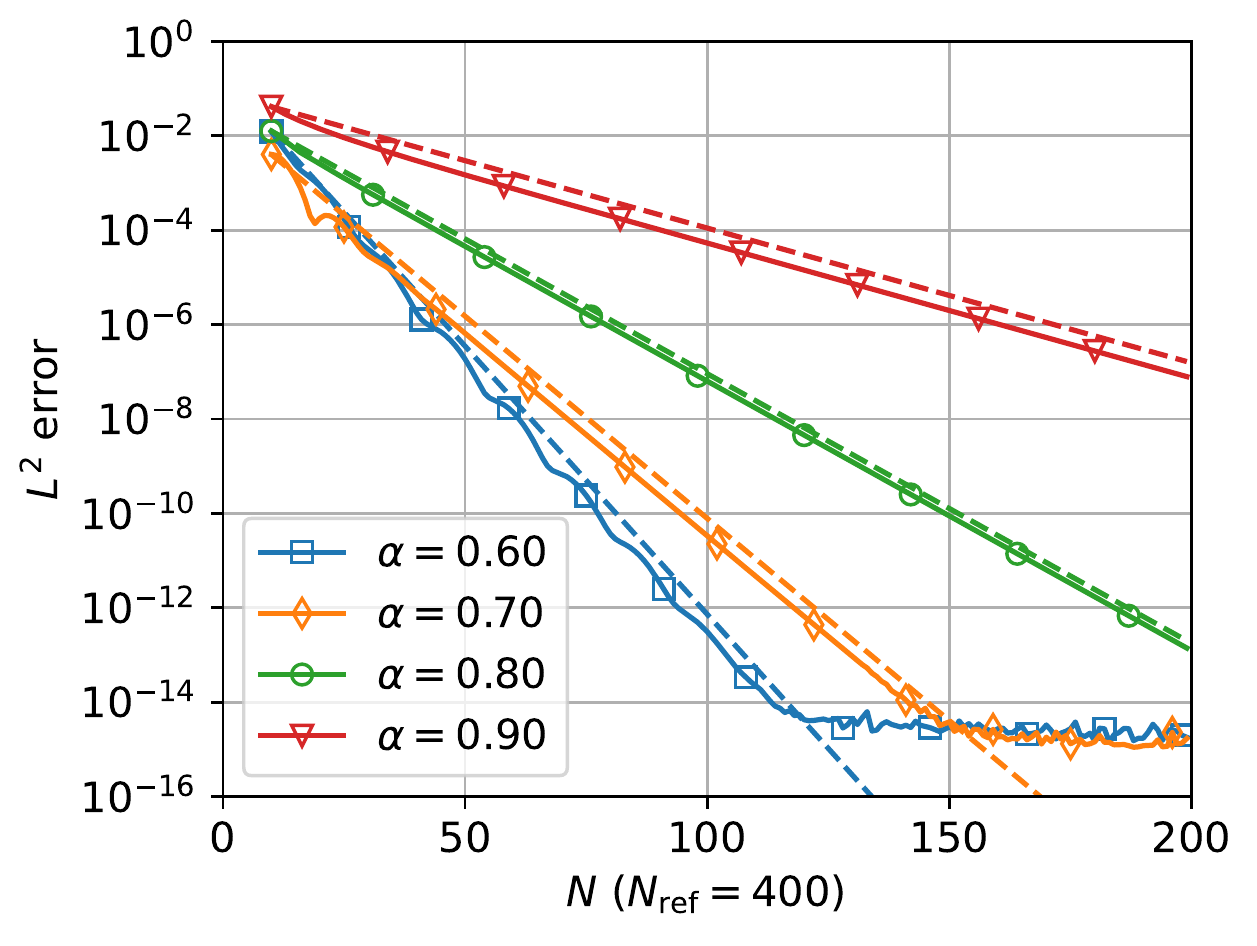}}
    \caption{Quadrature errors for $\cA_3$ under different $f$.  The dashed lines are $\mathcal{O}(e^{-\frac{2\pi\min(\kappa_1,\kappa_2)}{\tau}})$ with $\kappa_1\approx\alpha(\pi-1.059)$, $\kappa_2=(1-\alpha)\pi$.}\label{fig:quad-gco-split}
\end{figure}

\subsection{Application in solving time dependent problem}
As the last example, we consider the following space-fractional Allen-Cahn equation(see e.g., \cite{burrage2012efficient,NEC20083237,Song2016})
\begin{equation}
    \frac{\partial u}{\partial t} + \epsilon^2 (-\Delta)^\alpha u = F'(u), \quad \mbox{with } F(u) = \frac{1}{4}(1-u^2)^2.
\end{equation}
The equation above can be taken as the $L^2$-gradient flow under the free energy functional  $$\mathcal{E}[u]=\int_{\Omega}\frac{\epsilon^2}{2}|(-\Delta)^{\alpha/2}u|^2+F(u)dx.$$
We solve the problem in the periodic domain $(0, 2\pi)\times(0, 2\pi)$ under $\epsilon=0.1$ with the random initial 
data $u(x_1, x_2, 0) = 0.1\times{\texttt{rand}}(x_1, x_2) - 0.05$,  where the function $\texttt{rand}$ returns a random scalar in $[0, 1)$ sampled from the uniform distribution. The space and time step sizes are set to $h=\frac{2\pi}{128}$ and $\Delta t = \frac{1}{128}$, respectively. 
We determine the values of the parameters
$\tau$, $M$ and $N$ by \cref{choice-para} and the following tolerance
\begin{equation}
\frac{C(\tau)}{b} e^{-\sqrt{\pi\min(\kappa_1,\kappa_2)\left[(1+\alpha^{-1})M+N\right]}} =: {\rm tol} < 10^{-10},
\end{equation}
where $b=\frac{1}{\Delta t \epsilon^2}$ and $\kappa_1 = \alpha\pi, \kappa_2 = (1-\alpha)\pi$.
For $\alpha = 0.6$, we choose $(\tau, M, N) = (0.5, 5, 51)$,
and for $\alpha = 0.8$, we choose $(\tau, M, N) = (0.25, 12, 101)$.
One can observe in \cref{fig:allen-cahn-comp} that reducing the fractional power leads to slower merging process and thinner interface. 
This is consistent with the results reported in the previous literature \cite{burrage2012efficient}.

\begin{figure}[h]
    \centering
    \begin{tabular}{ccccc}
     \includegraphics[trim={320, 80, 320, 80}, clip, width=.2\textwidth]{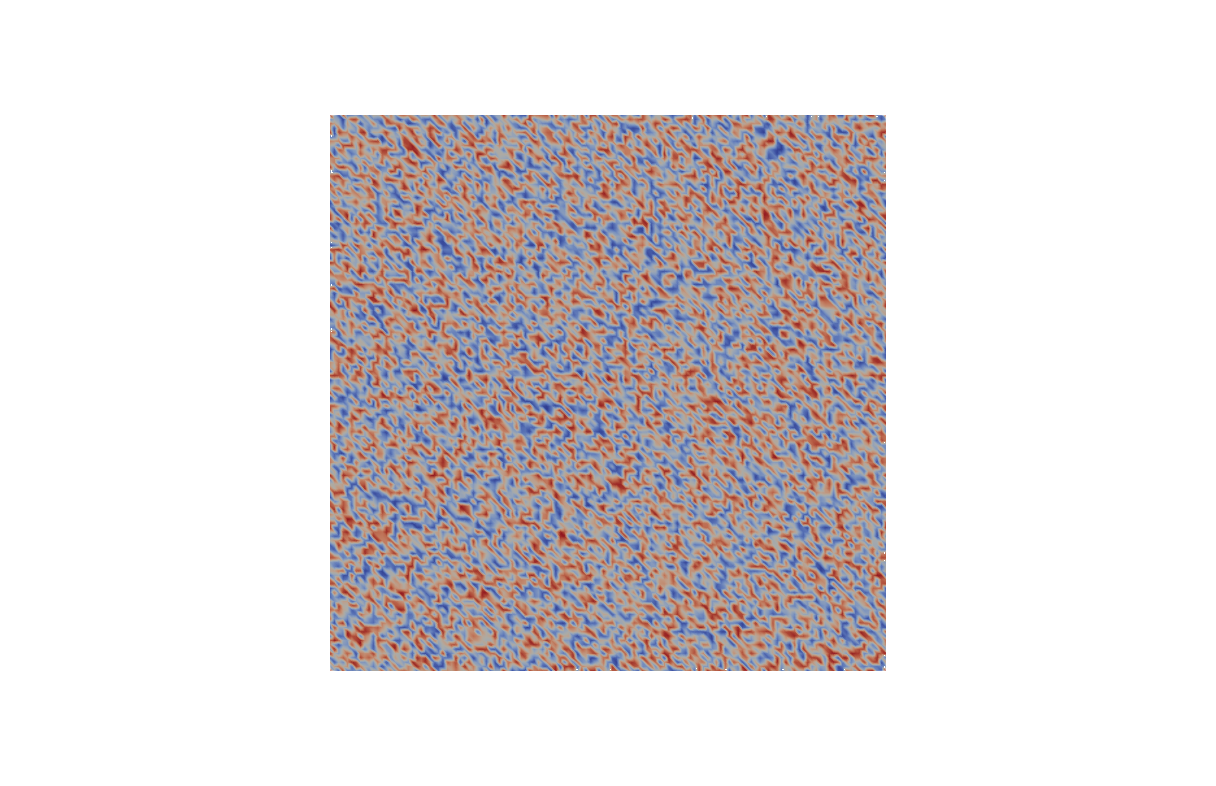}
    &\includegraphics[trim={320, 80, 320, 80}, clip, width=.2\textwidth]{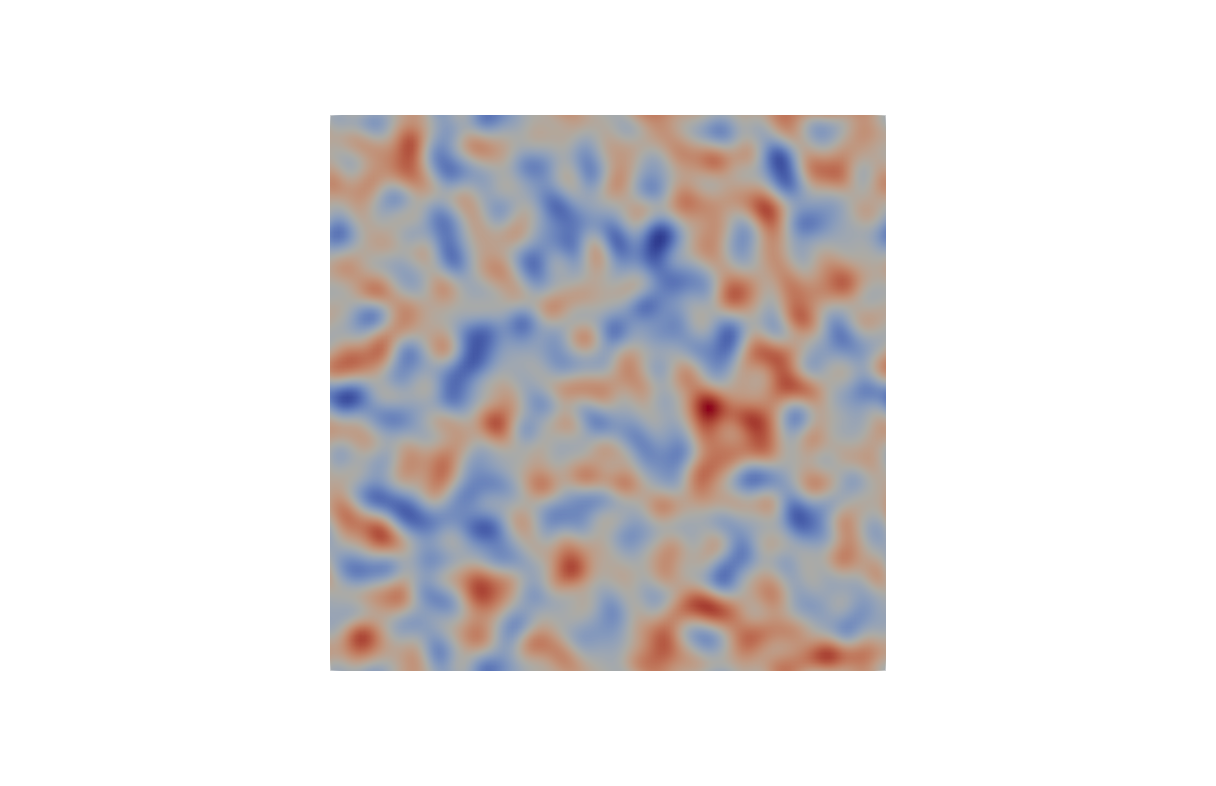}
    &\includegraphics[trim={320, 80, 320, 80}, clip, width=.2\textwidth]{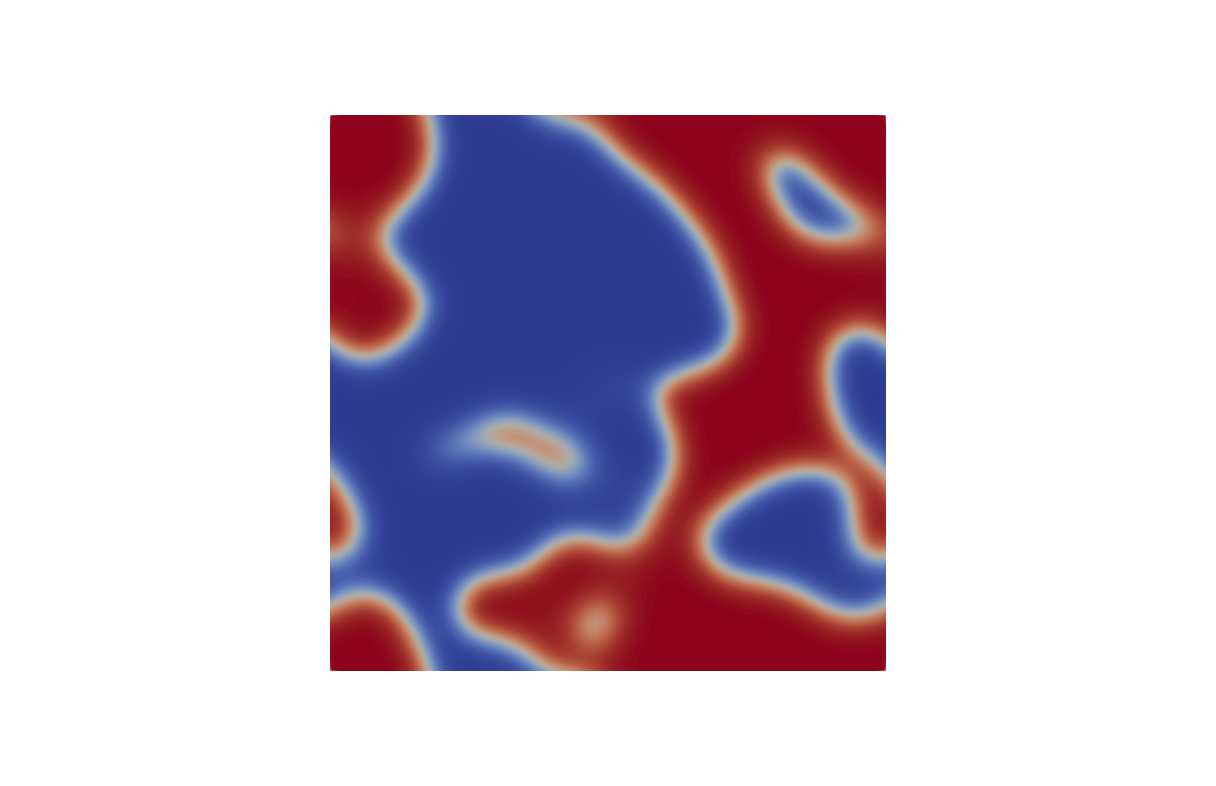}
    &\includegraphics[trim={320, 80, 320, 80}, clip, width=.2\textwidth]{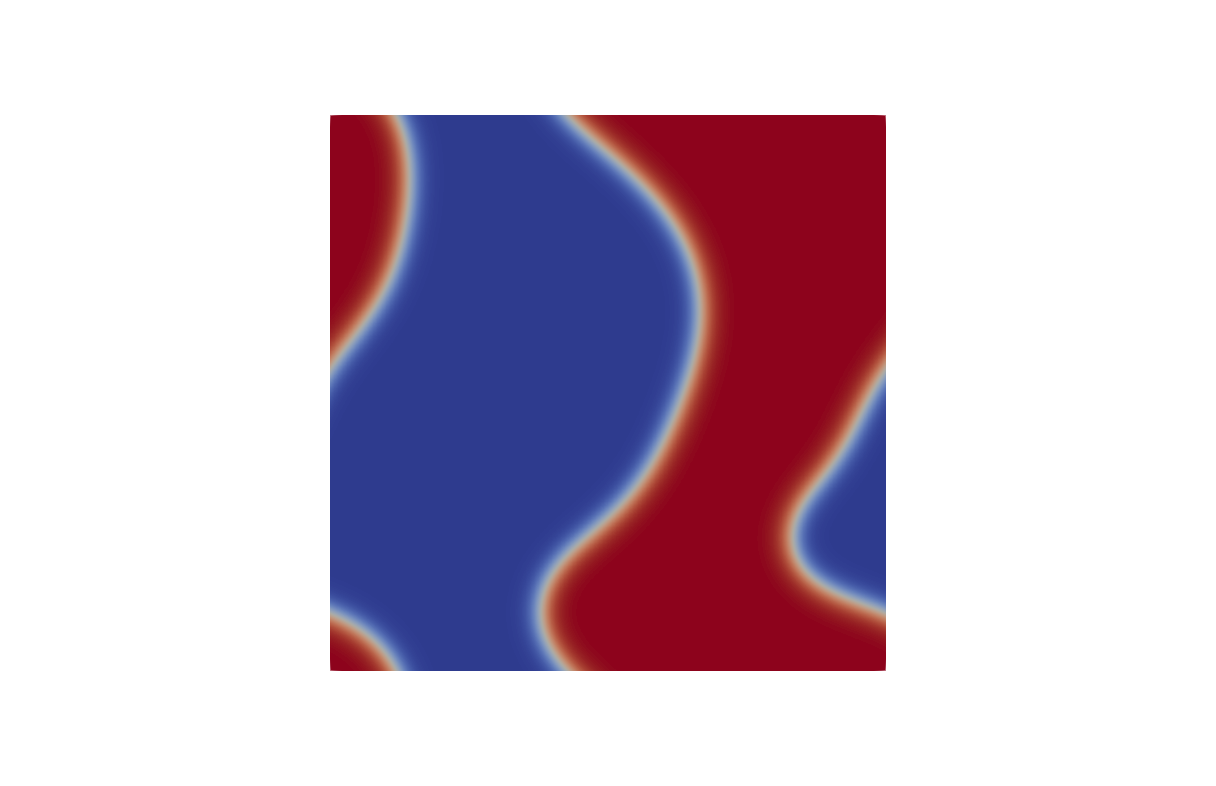} \\
     \includegraphics[trim={320, 80, 320, 80}, clip, width=.2\textwidth]{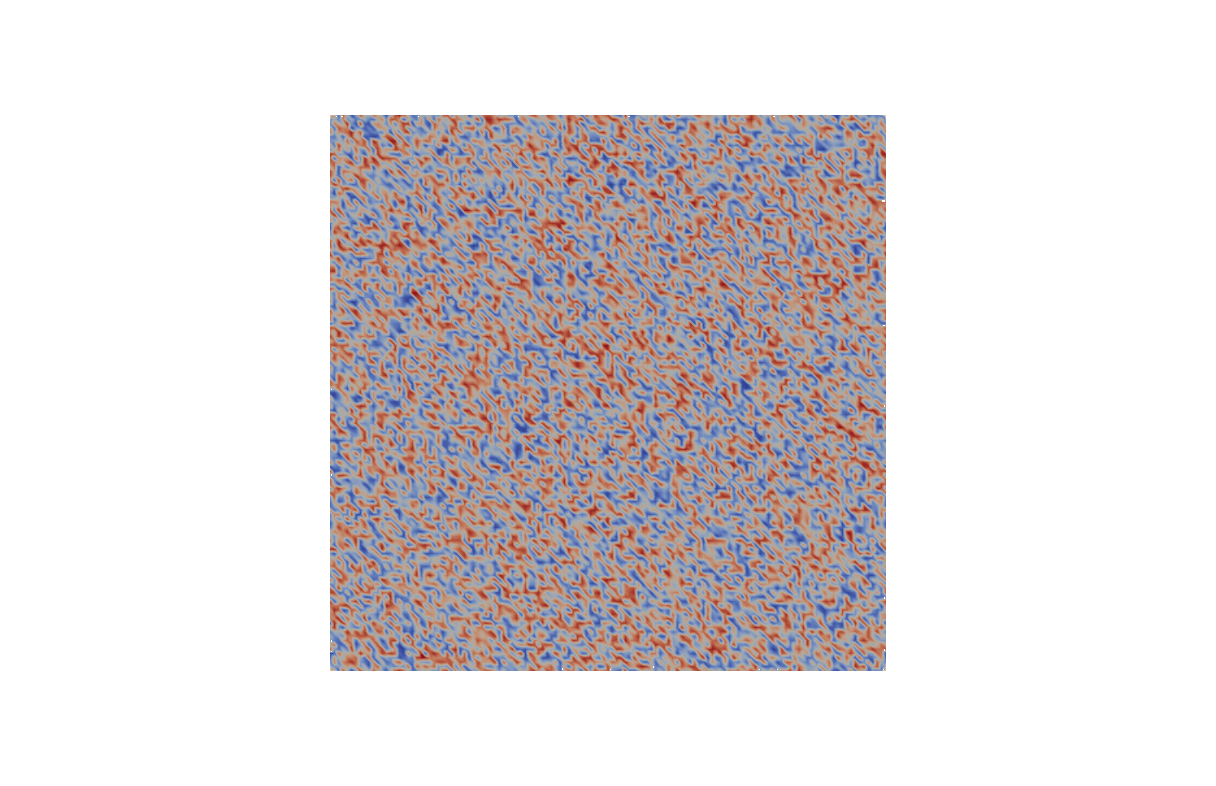}
    &\includegraphics[trim={320, 80, 320, 80}, clip, width=.2\textwidth]{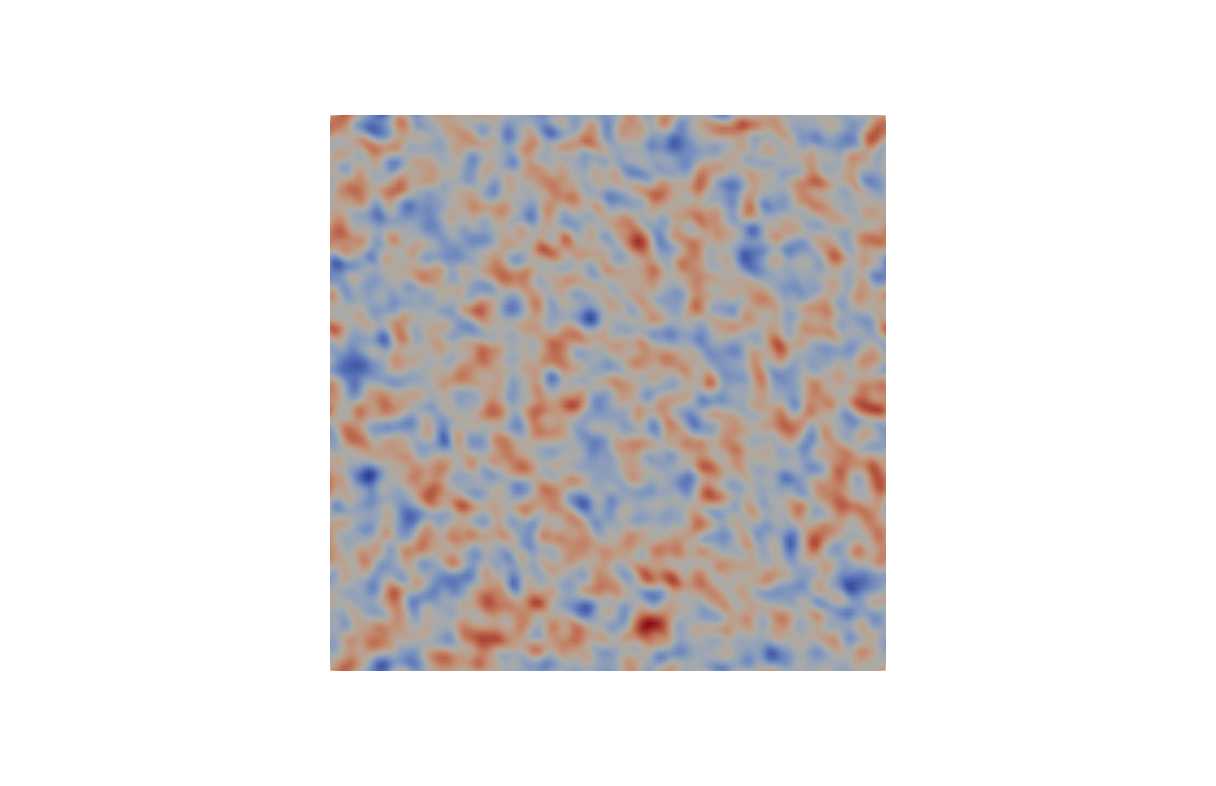}
    &\includegraphics[trim={320, 80, 320, 80}, clip, width=.2\textwidth]{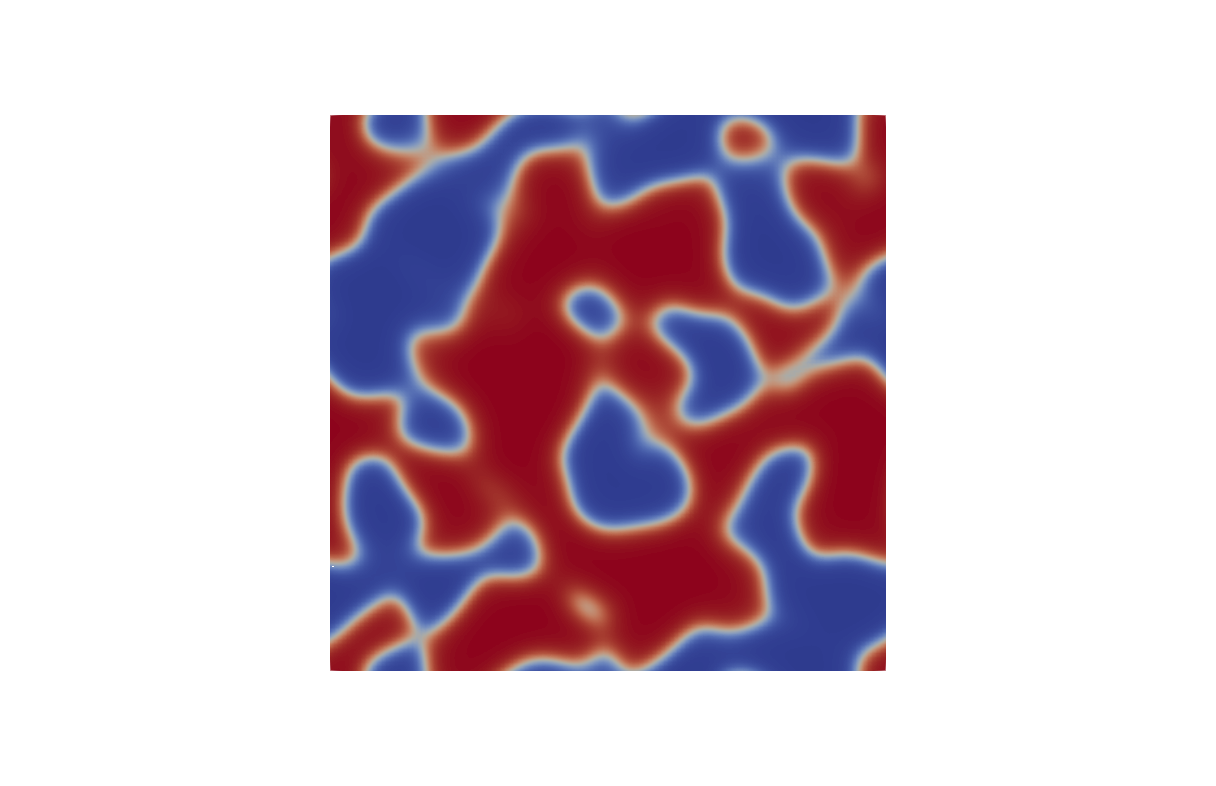}
    &\includegraphics[trim={320, 80, 320, 80}, clip, width=.2\textwidth]{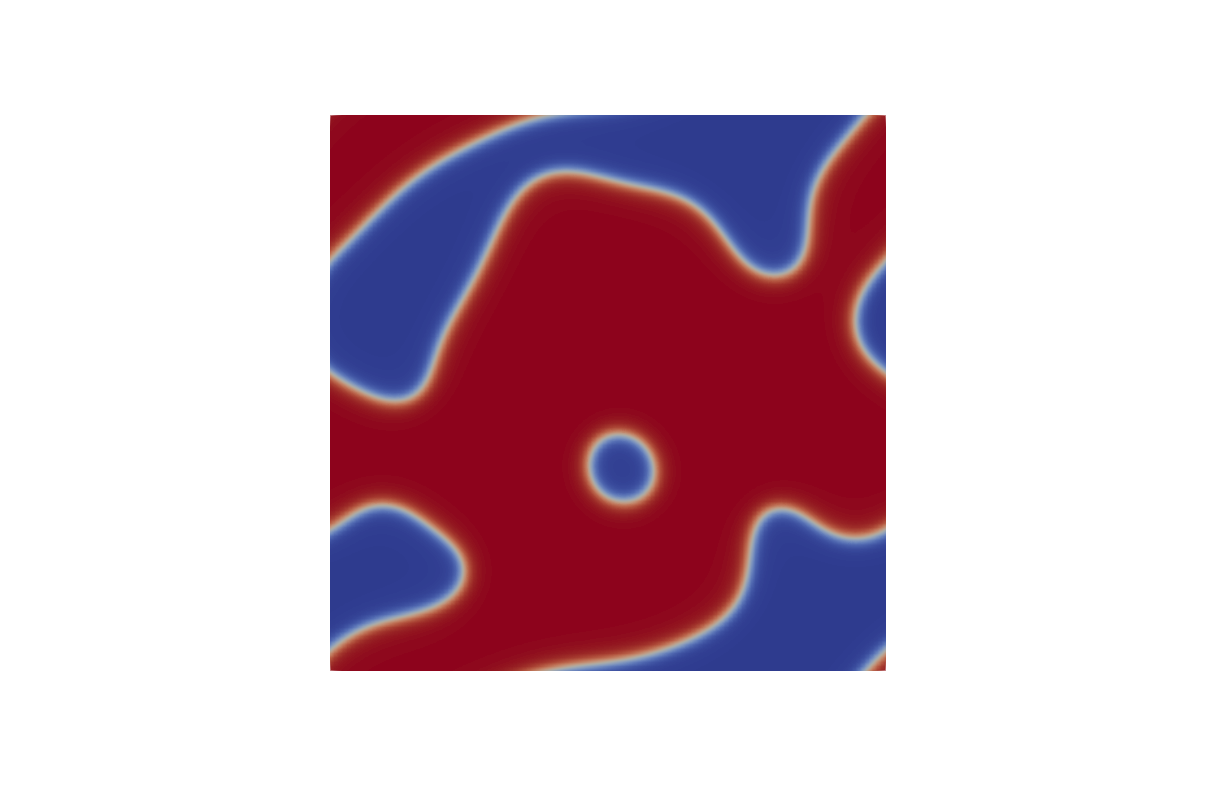} \\
     \includegraphics[trim={320, 80, 320, 80}, clip, width=.2\textwidth]{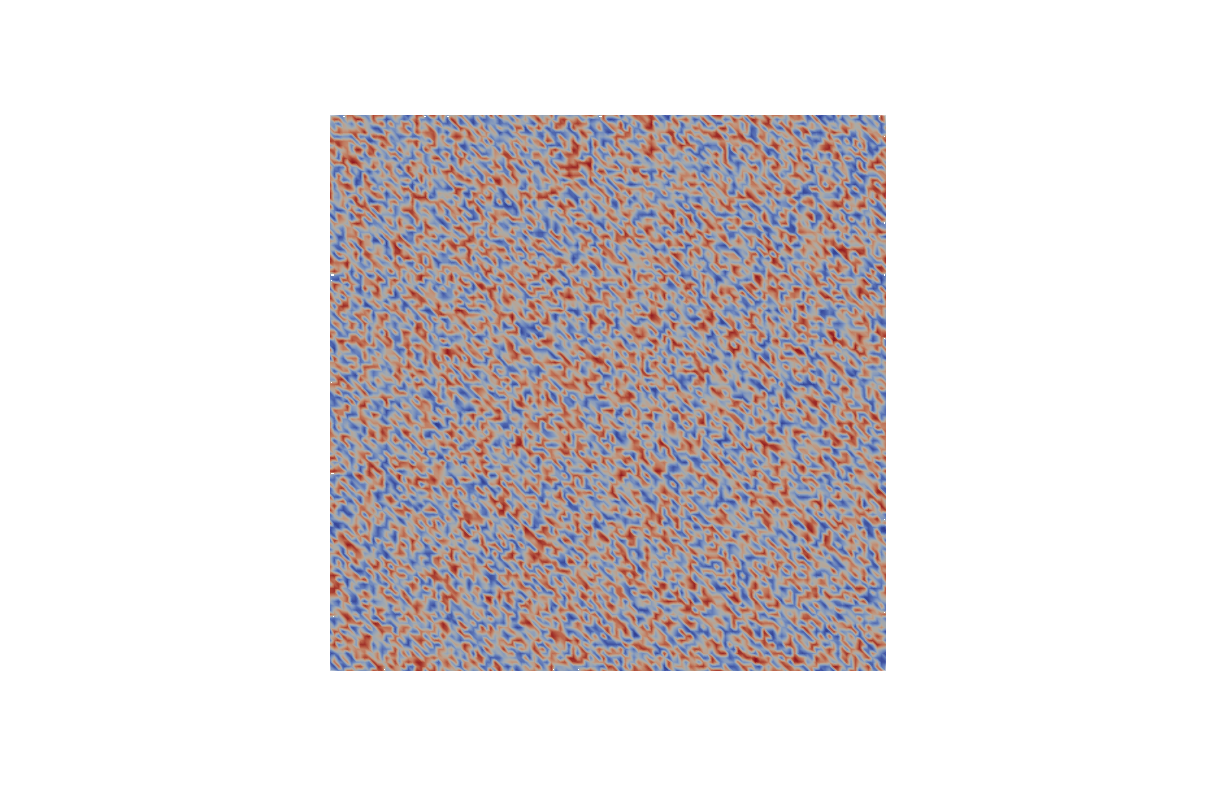}
    &\includegraphics[trim={320, 80, 320, 80}, clip, width=.2\textwidth]{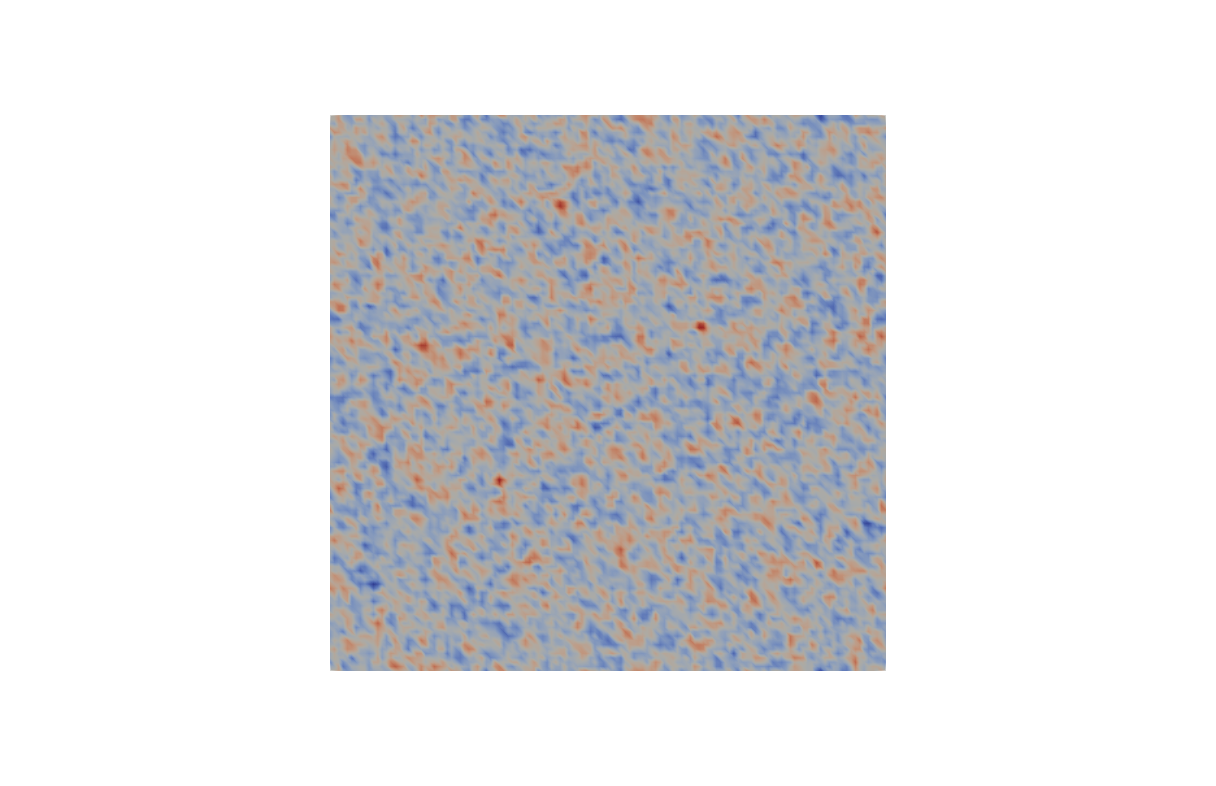}
    &\includegraphics[trim={320, 80, 320, 80}, clip, width=.2\textwidth]{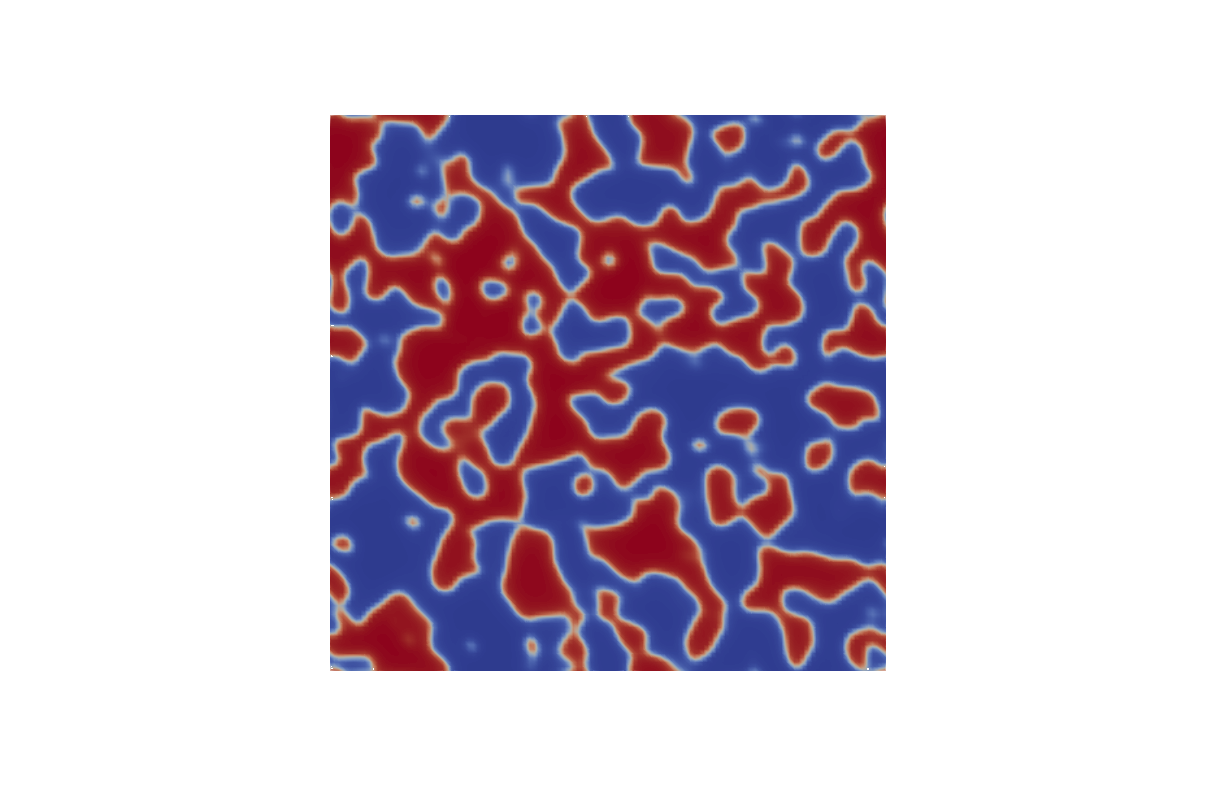}
    &\includegraphics[trim={320, 80, 320, 80}, clip, width=.2\textwidth]{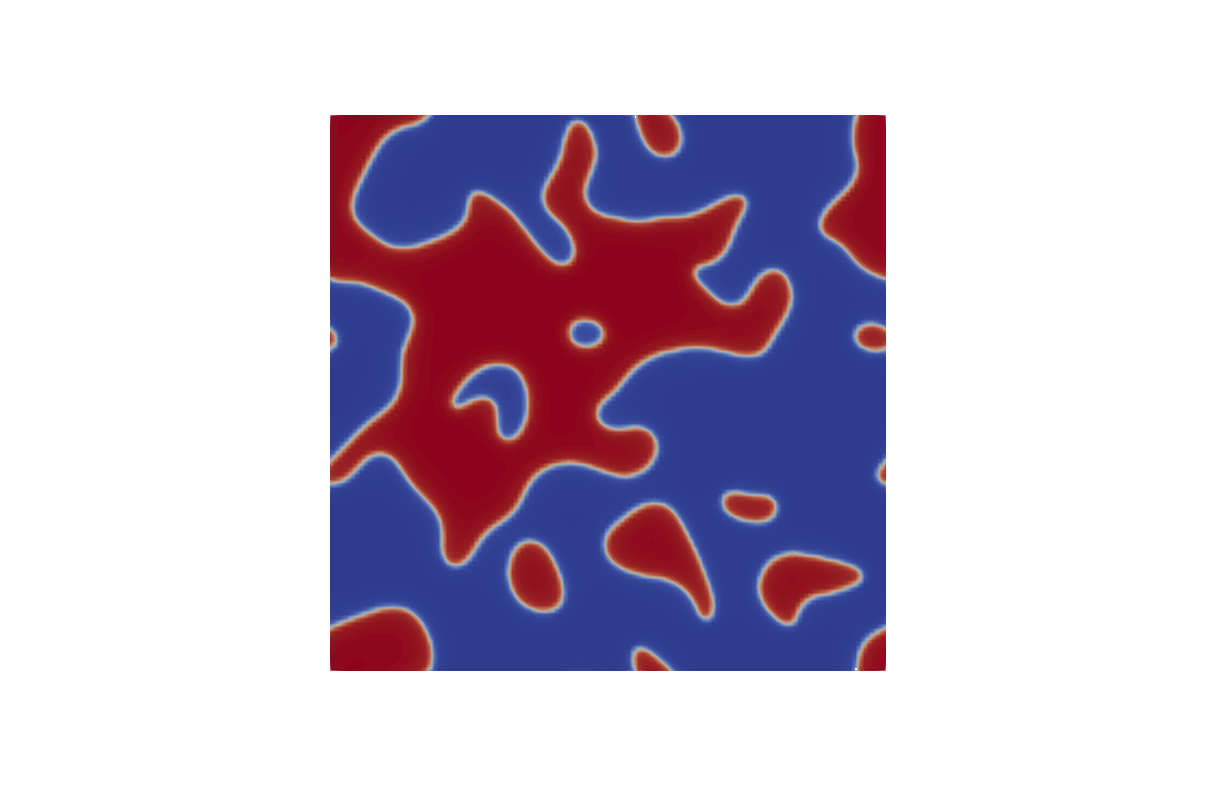} \\
      $t = 0$ & $t=1$ & $t=10$ & $t=50$
    \end{tabular}
    \caption{Allen-Cahn diffusion under different $\alpha$. Plots from top to bottom show the states for $\alpha=1,\, 0.8$ and $0.6$.}\label{fig:allen-cahn-comp}
\end{figure}

\section{Conclusion}
In this paper we have presented a quadrature scheme to numerically solve steady state diffusion equations involving fractional power of regularly accretive operator. Fully discrete error estimates have been demonstrated which show  that the quadrature error decays exponentially with respect to the step size $\tau$. The cost of the scheme is in solving some elliptic problems that can be executed in parallel. By balancing the truncation errors and quadrature error, the scheme is root-exponentially convergent with respect to the number of solves.  

\section*{Acknowledgements}
The authors would like to thank the anonymous referees for their valuable comments and suggestions.
The work of Beiping Duan was supported by National Natural Science Foundation of China(grant No. 12201418).
The work of Zongze Yang was supported in part by a grant from the Research Grants Council of the Hong Kong Special Administrative Region,
China (GRF Project No. PolyU15300920), and an internal grant of The Hong Kong Polytechnic University (Project ID: P0036728, Work Programme: W18K).

\bibliographystyle{siamplain}
\bibliography{references}
\end{document}